\documentclass[a4paper, 11pt]{article}

\usepackage[T1]{fontenc}
\usepackage[english]{babel}
\usepackage[utf8]{inputenx}
\usepackage{amsmath,amsfonts,amssymb,amsthm,mathrsfs,pinlabel,graphicx}
\usepackage[all]{xy} 
\SelectTips{cm}{}
\usepackage[colorlinks,citecolor=blue, linkcolor=blue, linktocpage]{hyperref}
\usepackage{tikz}
\usetikzlibrary{arrows,positioning}
\usepackage[parfill]{parskip}

\newif\iftikziii

\iftikziii
\usepackage{tikz-cd}
\fi

\begingroup
    \makeatletter
    \@for\theoremstyle:=definition,remark,plain\do{%
        \expandafter\g@addto@macro\csname th@\theoremstyle\endcsname{%
        \addtolength\thm@preskip\parskip
        }%
    }
\endgroup

\makeatletter
	\renewcommand{\@makecaption}[2]{
	\vspace{\abovecaptionskip}
	\sbox{\@tempboxa}{\textbf{#1}. #2}%
	\ifdim \wd\@tempboxa >\linewidth
		\textbf{#1}. #2\par
	\else
		\global\@minipagefalse
		\makebox[\linewidth]{\hfil\usebox{\@tempboxa}\hfil}%
	\fi
	\vspace{\belowcaptionskip}}
	
	\newcommand{\LineaCabecera}{\makebox[0pt][l]{\rule[-1.5ex]{\textwidth}{0.4pt}}} 

	\newcommand{\ps@memoire}{%
	\renewcommand{\@oddhead}{\LineaCabecera{\rightmark} \hfill\thepage}
	\renewcommand{\@evenhead}{}
	\renewcommand{\@oddfoot}{}
	\renewcommand{\@evenfoot}{}
	\renewcommand{\sectionmark}[1]{\markright{\thesection.\ ##1}}}
\makeatother

\numberwithin{equation}{section}
\numberwithin{figure}{section}
\allowdisplaybreaks

\newcommand{\bgm}{\color{black}{}}
\newcommand{\egm}{\normalcolor{}}

\newcommand{\centre}[1]{\begin{array}{c} #1 \end{array}}

\newtheoremstyle{mytheorem}{15pt}{15pt}{\itshape}{}{\bfseries}{.}{ }{} 
\theoremstyle{mytheorem}
\newtheorem{theorem}{Theorem}[section]
\newtheorem*{theorem*}{Theorem}
\newtheorem{lemma}[theorem]{Lemma}
\newtheorem{proposition}[theorem]{Proposition}
\newtheorem{corollary}[theorem]{Corollary}

\theoremstyle{remark} 
\newtheorem{remark}[theorem]{Remark}

\theoremstyle{definition} 

\newcommand{\form}[2]{\langle #1 , #2 \rangle}
\newcommand{\formempty}{\langle \cdot , \cdot \rangle}
\newcommand{\SSform}[2]{\langle #1 , #2 \rangle_ {\operatorname{s}}}
\newcommand\SSempty{\langle \cdot , \cdot \rangle_{\operatorname{s}}}
\newcommand\pai{\rho}
\newcommand{\by}[1]{\stackrel{\eqref{#1}}{=}}
\newcommand\dep{{\partial}}
\newcommand\up{\vspace{-0.8cm}}

\newcommand\Mag{{{\mathbf{\mathsf{Mag}}}}}
\newcommand\Cob{{{\mathbf{\mathsf{Cob}}}}}
\newcommand\Lagr{{{\mathbf{\mathsf{Lagr}}}}}
\newcommand\pLagr{{{\mathbf{\mathsf{pLagr}}}}}
\newcommand\grMod{{{\mathbf{\mathsf{grMod}}_R}}}

\newcommand\Uni{{{\mathbf{\mathsf{U}}}_{{R}}}}
\newcommand\pUni{{{\mathbf{\mathsf{pU}}}_{{R}}}}
\newcommand\UniQ{{{\mathbf{\mathsf{U}}}^{\operatorname{r}}_{{Q}}}}
\newcommand\pUniQ{{{\mathbf{\mathsf{pU}}}^{\operatorname{r}}_{{Q}}}}
\newcommand\mcg{{{\mathsf{MCG}}}}
\newcommand\HCob{{{\mathsf{C}}}}
\newcommand\Alex{{{\mathsf{A}}}}
\newcommand\vol{{{\omega}}}
\newcommand\Pl{{{\mathsf{Pl}}}}

\DeclareMathOperator{\cl}{\operatorname{cl}}

\DeclareMathOperator{\ord}{\operatorname{ord}}

\DeclareMathOperator{\Tors}{\operatorname{Tors}}

\DeclareMathOperator{\Aut}{\operatorname{Aut}}
\DeclareMathOperator{\coker}{\operatorname{coker}}
\DeclareMathOperator{\Ann}{\operatorname{Ann}}

\DeclareMathOperator{\Hom}{\operatorname{Hom}}
\DeclareMathOperator{\Id}{\operatorname{Id}}

\newcommand\cA{{\mathcal A}}

\newcommand\RR{{\mathbb R}}

\newcommand\ZZ{{\mathbb Z}}

\newcommand\gL{{\Lambda}}

\newcommand\gG{{\Gamma}}

\newcommand\ga{{\alpha}}
\newcommand\gb{{\beta}}
\newcommand\ggm{{\gamma}}
\newcommand\gd{{\delta}}

\newcommand\gl{{\lambda}}

\newcommand\gvf{{\varphi}}

\newcommand\bfc{{\mathbf{c}}}

\mathchardef\ordinarycolon\mathcode`\:
\mathcode`\:=\string"8000
\begingroup \catcode`\:=\active
\gdef:{\mathrel{\mathop\ordinarycolon}}
\endgroup

\title{A functorial extension of the Magnus representation\\ to the category of three-dimensional cobordisms} 

\author{Vincent Florens \and Gw\'ena\"el Massuyeau \and Juan Serrano de Rodrigo}

\newcommand{\Addresses}{{
		
		\bigskip
		
		\footnotesize
		
		\begin{tabular}{l}
		\textsc{Vincent Florens} \\ 
		\textsc{LMA, Universit\'e de Pau \& CNRS} \\ 
		\textsc{Avenue de l'Universit\'e} \\
		\textsc{64000 Pau, France} \\
		\texttt{vincent.florens@univ\hbox{-}pau.fr}
		\end{tabular}
		
		\medskip
		
		\begin{tabular}{lcl}
		\textsc{Gw\'ena\"el Massuyeau} &&  \\ 
		\textsc{IRMA, Universit\'e de Strasbourg \& CNRS} & \quad \emph{and} \quad & \textsc{IMB,  Universit\'e de Bourgogne \& CNRS }  \\
		\textsc{7 rue Ren\'e Descartes} && \textsc{9 avenue Alain Savary}  \\
		\textsc{67084 Strasbourg, France} && \textsc{21000 Dijon, France}  \\
		\texttt{massuyeau@math.unistra.fr} &&
		\end{tabular}
				
		\medskip
		
		\begin{tabular}{l}
		\textsc{Juan Serrano de Rodrigo} \\
		\textsc{Dpto. de Matem\'aticas, Universidad de Zaragoza} \\
		\textsc{Calle Pedro Cerbuna 12} \\
		\textsc{50009 Zaragoza, Spain} \\
		\texttt{serrano.de.rodrigo@gmail.com}
		\end{tabular}
	}}
	
\begin{document}
	

\maketitle

\begin{abstract} 
Let $R$ be an integral domain and $G$ be a subgroup of its group of units. We consider the  category $\Cob_G$ of $3$-dimensional cobordisms between oriented surfaces with connected boundary, equipped with a representation of their fundamental group  in $G$. Under some mild conditions on $R$, we construct a  monoidal functor from  $\Cob_G$ to the category $\pLagr_R$ consisting of ``pointed Lagrangian relations'' between skew-Hermitian $R$-modules. We call it the ``Magnus functor'' since it contains the Magnus representation of mapping class groups as a special case. Our construction is inspired from the work of Cimasoni and Turaev on the extension of the Burau representation of braid groups to the category of tangles. It can also be regarded as a $G$-equivariant version of a TQFT-like functor that has been described by Donaldson. The study and computation of the Magnus functor  is carried out using classical techniques of low-dimensional topology. When $G$ is a free  abelian group and $R=\ZZ[G]$ is the group ring of $G$, we relate the Magnus functor to the ``Alexander functor'' (which has been introduced in a prior work using Alexander-type invariants), and we deduce a factorization formula for the latter. 
\end{abstract}
	





 
 \bigskip \bigskip 

\section{Introduction} \label{Introduction}

Let $\Sigma$ be a compact connected oriented surface with $\partial \Sigma \neq \varnothing$, and let $\pi:= \pi_1(\Sigma,\star)$ be its fundamental group based at a point $\star \in \partial \Sigma$.
The \emph{mapping class group}  $\mcg(\Sigma)$ consists of the isotopy classes of (orientation-preserving) self-homeomorphisms of $\Sigma$ fixing the boundary pointwise. 
``Magnus representations''  usually refer to those ``representations'' of subgroups of $\mcg(\Sigma)$ 
that are defined by assigning to an $f\in \mcg(\Sigma)$ the matrix with entries in $\ZZ[\pi]$ consisting of Fox's free derivatives  of $f_*:\pi \to \pi$ with respect to a fixed basis of $\pi$. Thus they have a group-theoretical definition, which goes through the automorphism group of $\pi$.

Birman coined the  terminology ``Magnus representations''  in the third chapter of her book \cite{Bi74}, where
it is observed   that these kinds of ``representations''  arise from  matrix representations of  free groups dating back to Magnus \cite{Ma39}.
Birman gave there an algebraic exposition and survey of these ``representations'', explaining for instance how far they are from being  group homomorphisms, or analyzing their kernels and images. One of her motivations  was to give a unified  treatment of the \emph{Burau representation} of the braid group and the \emph{Gassner representation} of the pure braid group. These correspond to the case where  the surface $\Sigma$ is a disk with marked points or holes, and are defined from the ``Magnus representations'' by reducing the coefficients in $\ZZ[\pi]$ to some appropriate commutative~rings.

The Gassner representation of  the pure braid group was later extended to string links (also called ``pure tangles'') by Le Dimet \cite{LD92}.
Kirk, Livingston and Wang \cite{KiLiWa01} gave a topological interpretation of this extension and a simple proof of its invariance under concordance. Their approach is based on a natural action of the monoid of string links on the twisted homology of a  punctured disk, and relies on the topological interpretation of Fox's free derivatives in terms of universal covers. 
\bgm In these works, the study of the  Burau and Gassner representations is partly motivated by their tight connections with the Alexander polynomial of links. \egm
More recently, Cimasoni and Turaev \cite{CiTu05I,CiTu05II} extended \bgm these \egm representations to arbitrary tangles. Their invariant is defined as a functor from the category of (colored) tangles to the category of ``Lagrangian relations'' between skew-Hermitian modules.  
 
In the case of a surface $\Sigma$ of positive genus, Magnus representations have been used and studied by Morita \cite{Mo93}, Suzuki \cite{Su03} and Perron \cite{Pe06} among others.
In this case too, there is a group-theoretical definition in terms of Fox's free derivatives as well as a topological definition using twisted homology  (see \cite{Su05}).
Furthermore, the Magnus representation is extended in \cite{Sa08} to the monoid of homology cobordisms \bgm (also called ``homology cylinders''), which 
are  higher-genus versions  of string links. \egm
\bgm In these works, the study of the Magnus representation is driven by its relations 
to the Alexander polynomial and Reidemeister torsions of closed $3$-manifolds,
and by its role in the study of the ``Johnson homomorphisms''. \egm
We refer to Sakasai \cite{Sa12} for an overview of these topics.
In this paper, inspired by the above-mentioned work of Cimasoni and Turaev,  we construct 
 a functorial extension of the  Magnus representations  	
  to a certain category of $3$-dimensional cobordisms.

 To be more specific, our framework is  the category $\Cob$ of $3$-dimensional cobordisms
 introduced by Crane and Yetter \cite{CrYe99}. The objects of this category are  compact connected oriented surfaces $F_g$ of arbitrary genus $g \geq 0$, with exactly one boundary component,
 and the morphisms  are $3$-dimensional cobordisms with corners between such surfaces. 
\bgm The category  $\Cob$  originates  from the study of $3$-dimensional Topological Quantum Field Theories (TQFT's).  
We consider here the refinement $\Cob_G$ of this  category
 where the surfaces and $3$-manifolds are equipped with a representation of their fundamental group in a \emph{fixed} group $G$. \egm

\bgm Specifically, we assume  that \egm $G$ is a subgroup of the group of units of an integral domain~$R$, and that $R$ has an involutive ring endomorphism which extends the  map of~$G$ given by $g \mapsto g^{-1}$. 
Recall from \cite{CiTu05I} that the category $\Lagr_R$ of \emph{Lagrangian relations} is defined as follows: 
objects are finitely generated $R$-modules $H$  equipped   with a non-degenerate skew-Hermitian form $\rho$, and morphisms $(H_1,\rho_1) \to (H_2,\rho_2)$ are Lagrangian submodules of $\big(H_1\oplus H_2,(-\rho_1)\oplus \rho_2\big)$.   
\bgm The following construction can be regarded as a TQFT-like extension of the Magnus representations.\\

\noindent
\textbf{Theorem I.}
\emph{There exists  a  \bgm functor \egm  
$
\Mag:= \Mag_{R,G}~ \colon \Cob_G \to \bgm \Lagr_R \egm
$ 
which is defined as follows. At the level of objects, $\Mag$  assigns  to any  pair $\big(F_g, \varphi~ \colon \pi_1(F_g,\star) \rightarrow G\big)$  
the \bgm skew-Hermitian $R$-module \egm 
$$
\bgm \big(H_1^\varphi(F_g,\star),\, \langle  \cdot , \cdot  \rangle_s \big) \egm
$$
where $\star \in \partial F_g$ \bgm and \egm  $ \langle \cdot ,\cdot  \rangle_s~ \colon  H_1^\varphi(F_g,\star) \times H_1^\varphi(F_g,\star) \to R$
 is \bgm a version of the equivariant   intersection form with  coefficients in $R$ twisted by $\varphi$. \egm  
At the level of morphisms, $\Mag$ assigns to any cobordism $(M,\varphi)$ between $(F_{g_-},\varphi_-)$ and $(F_{g_+},\varphi_+)$  the (closure of) the kernel of the $R$-linear map 
$$
(-m_-)\oplus m_+~ \colon  H_1^{\varphi_-}(F_{g_-},\star) \oplus H_1^{\varphi_+}(F_{g_+},\star) \longrightarrow H_1^\varphi(M,\star)
$$
induced by the inclusions $m_\pm~ \colon  F_{g_\pm} \to \partial M \subset M$.}\\

\noindent
One  difficulty to adapt the work of Cimasoni and Turaev from tangles to cobordisms lies in the construction of a skew-Hermitian form on  $H_1^\varphi(F_g,\star)$: 
 here the form $ \langle  \cdot , \cdot  \rangle_s$ is derived from the homotopy intersection pairing that Turaev introduced in \cite{Tu78}. 
The main advantage of considering the twisted homology of $F_g$ relative to a base point $\star$ (instead of the absolute twisted homology as in \cite{CiTu05I}) is that the module associated to the surface $F_{g}$ is always free of rank $2g$, regardless of the way coefficients are twisted by $\varphi$. 
\bgm Another difference with their work is that we will deal with monoidality.
Indeed, the  boundary-connected sums of  surfaces and $3$-manifolds induce a  monoidal structure on  the category $\Cob_G$.
 We  introduce the category $\mathsf{pLagr}_R$ of \emph{pointed Lagrangian relations}
 where each  skew-Hermitian $R$-module is endowed with a distinguished element of its rationalization.
 This refinement of the category $\mathsf{Lagr}_R$ has a monoidal structure defined  by a skew version of the direct sum,
 and the above functor $\Mag$ can be refined to preserve these monoidal structures. 
(See Theorem \ref{MainTheorem} and Proposition \ref{prop:monoidality} for precise statements.) \egm 
\bgm Apart for  the aforesaid differences, the proof of Theorem I  follows essentially \egm the same lines as the construction of Cimasoni and Turaev.

\bgm There is  a  relation  of  $4$-dimensional homology cobordism between $3$-dimensional cobordisms. 
 We denote this  by~$\sim_H$ and call it the relation of \emph{homology concordance}, since it is an analogue of the concordance relation for tangles. \egm
As one may expect, the Magnus functor descends to the quotient category $\Cob_G/\! \sim_H$ (see Proposition \ref{prop:h_cobordism_rel}). For instance, when $G=\{ 1 \}$ and $R= \mathbb{R}$, the functor $\Mag$ 
provides a   functor from the quotient category  $\Cob/\!\sim_H$ 
to the category of Lagrangian relations between symplectic $\mathbb{R}$-modules. This is essentially the TQFT-like functor introduced by Donaldson \cite{Do99} as a tool to re-prove the surgery formulas for  the Casson invariant and $3$-dimensional Seiberg--Witten invariants. Under some homological assumptions on the cobordisms, 
this ``TQFT'' is  also equivalent to a construction of Frohman and Nicas \cite{FN91} \bgm which involves  moduli spaces of flat $U(1)$-connections. \egm

\bgm In the recent literature, the equivalence  relation $\sim_H$ 
has been mainly studied on the monoid $\mathsf{C}(F_g)$ of \emph{homology cobordisms}  over the surface $F_g$:
see for instance \cite{GL05,Mo08,CFK11,CST16}.
We consider here the submonoid $\mathsf{C}^\varphi(F_g)$  of $\mathsf{C}(F_g)$ 
consisting of those homology cobordisms \egm that are compatible with a fixed representation $\varphi~ \colon  \pi_1(F_g,\star) \to G$. 
Since it  \bgm can be viewed as \egm a submonoid of the monoid of endomorphisms of the object $(F_g,\varphi)$ in $\Cob_G$, 
we   consider  the restriction of $\Mag$ to $\mathsf{C}^\varphi(F_g)$ and we find that it is equivalent to the  Magnus representation 
$$
r^\varphi ~ \colon  \mathsf{C}^\varphi(F_g) \bgm /\! \sim_H \egm  \rightarrow \Aut_{Q} \big(Q \otimes_R H_1^{\varphi} (F_g,\star)\big)
$$ 
where $Q := Q(R)$ is the field of fractions of $R$ (see Proposition \ref{prop:magnus}). Here, following  \cite{CiTu05I} again, we  regard  unitary isomorphisms between skew-Hermitian  $R$-modules as morphisms in the category $\Lagr_R$   \bgm by considering their set-theoretical graphs.

In particular, we obtain that the representations of $\mcg(F_g)$ arising from the functor $\Mag$ coincide with the usual Magnus representations,
so that they can be computed very easily using Fox's free differential calculus. 
More generally, we explain in Section \ref{subsec:Heegaard} how to compute $\Mag$ on an arbitrary cobordism which \bgm is presented  by a Heegaard splitting. 
 For instance, these techniques may be applied to compute $\Mag$  on the generators of the monoidal category $\Cob_G$ 
 that   arise from the generating system of $\Cob$  given  in \cite{Ke03}. 
According to \cite{CrYe99}, the monoidal category $\Cob$ is braided and the object $F_1$    therein  is a ``braided Hopf algebra'':
since the functor $\Mag$ is monoidal, this rich algebraic structure reflects in the monoidal category~$\pLagr_R$. \egm

Finally, we apply the Magnus functor $\Mag$ to the study of the  \emph{Alexander functor} $\Alex$ introduced in \cite{FlMa14},
\bgm which provides a kind of TQFT for the Alexander polynomial of knots in homology 3-spheres. \egm
This functor is defined on  $\mathsf{Cob}_G$ too, but takes values in the category $\grMod_{,\pm G}$ of  graded $R$-modules and $R$-linear maps which are only defined up to multiplication by an element of $\pm G \subset R$, and which may shift the degree. It can be constructed using either the Alexander function introduced by Lescop  \cite{Lescop} for $3$-manifolds with boundary, or using the theory of Reidemeister torsions. The functor $\mathsf{A}$  assigns to any pair $(F_g,\varphi)$ the graded $R$-module $\Lambda H_1^\varphi(F_g,\star)$, and it assigns to any cobordism $(M,\varphi)$ between $(F_{g_-},\varphi_-)$ and $(F_{g_+},\varphi_+)$ an $R$-linear map of degree $g_+-g_-$. Our second  main  result is the following  (see Theorem \ref{th:Alex_Magnus} for a precise statement).\\

\noindent   
\textbf{Theorem II.}
\emph{For any cobordism $(M,\varphi)$ from $(F_{g_-},\varphi_-)$ to $(F_{g_+},\varphi_+)$, the $R$-module $\Mag(M,\varphi)$ is tantamount to an $R$-linear map of degree $g_+-g_-$
$$
\Mag_W (M,\varphi)~ \colon  \Lambda H_1^{\varphi_-}(F_{g_-},\star) \longrightarrow \Lambda H_1^{\varphi_+}(F_{g_+},\star)
$$
which is defined up to multiplication by an element of $\pm G$ and satisfies
$$
\Alex(M,\varphi) = \Delta(M,W) \cdot \Mag_W (M,\varphi),
$$ 
where $\Delta(M,W)\in R/\!\pm G$ is a kind of ``relative'' Alexander polynomial for the $3$-manifold $M$.}\\
  
\noindent   
The above factorization formula for $ \Alex(M,\varphi)$ depends on the choice of a free submodule 
\bgm $W$ of $ H_1^{\varphi_-}(F_{g_-},\star) \oplus H_1^{\varphi_+}(F_{g_+},\star)$ \egm
of rank $g_-+g_+$, which is \emph{rationally} (i.e$.$ after taking coefficients in~$Q$) a supplementary subspace of $\Mag(M,\varphi)$. 
\bgm In addition to giving an ``operator'' viewpoint on the Magnus functor,
Theorem II fully computes the Alexander functor by showing that $\Alex(M,\varphi)$ splits into two parts: 
an ``operator'' part --- namely $\Mag_W (M,\varphi)$ ---  which is invariant under homology concordance $\sim_H$, and a ``scalar'' part ---  namely $\Delta(M,W)$ --- which does not have such property.
This generalizes a phenomenon that has  been observed in the special case of homology cobordisms \cite{FlMa14}: see  Remark~\ref{rem:hc}. \egm
We conclude by mentioning that the proof of this formula can be  adapted  to the situation of tangles, which implies a similar relationship between  the Alexander representation of  tangles constructed in \cite{BiCaFl15} (see also \cite{Ar10}, \cite{DaFl16}) and the functor of Cimasoni and Turaev.

The paper is organized as follows. In Section 2, we introduce the monoidal categories $\mathsf{Cob}_G$ and $\mathsf{pLagr}_R$. Section 3 is devoted to equivariant intersection forms for compact oriented surfaces. The functor  $\Mag$ is constructed in Section 4, where we also prove its monoidality and invariance under homology concordance. In Section 5, we give some examples and recipes for computations. Finally, Section 6 is devoted to the relation with the Alexander functor $\mathsf{A}$.
\bgm The paper ends with an appendix which briefly recalls the terminology of monoidal categories. \egm\\

\textbf{Conventions.} 
Let $X$ be a topological space with base point $\star$. The maximal abelian cover of $X$  based at $\star$ is denoted by $p_X~ \colon \widehat{X}\to X$, 
and the preferred lift of $\star$ is denoted by $\widehat{\star}$. (Here we assume  the appropriate conditions  on $X$ to have a universal cover.) For any oriented loop $\alpha$ in $X$ based at $\star$, the unique lift of $\alpha$ to $\widehat X$ starting at $\widehat \star$ is denoted by $\widehat \alpha$. If $X$ is an oriented manifold, we denote by $-X$ the same manifold with the opposite orientation.

Unless otherwise specified, (co)homology groups are taken with coefficients in the ring of integers~$\ZZ$; (co)homology classes are denoted with square brackets $[-]$. For any subspace $Y \subset X$ such that $\star \in Y$ and any ring homomorphism $\varphi~ \colon  \ZZ[H_1(X)] \to R$, we denote by $H^{\varphi}(X,Y)$ the \emph{$\varphi$-twisted homology} of the pair $(X,Y)$, namely
$$ H^{\varphi}(X,Y) := H(C^\varphi(X,Y)) \quad  \hbox{where} \ C^\varphi(X,Y) :=  R \otimes_{\ZZ[H_1(X)]} C\big(\widehat{X},p_X^{-1}(Y)\big). $$
If $(X',Y')$ is another topological pair and $f~ \colon (X',Y') \to (X,Y)$ is continuous,  the corresponding homomorphism $H(X') \to H(X)$ is still denoted by $f$. If a base point $\star'\in Y'$ is given and $f(\star')=\star$, the $R$-linear map $H^{\varphi f}(X',Y') \to H^{\varphi }(X,Y)$ induced by $f$ is also denoted by~$f$.


\section{The categories of cobordisms and Lagrangian relations} \label{Categories}

 In this section, we introduce two categories which will be respectively  the source and the target  of the Magnus functor to be constructed in Section \ref{FunctorFG}.

\subsection{The category  $\Cob$} 

We first recall the definition of the category  $\Cob$ of  {\itshape 3-dimensional cobordisms} introduced by Crane and Yetter \cite{CrYe99}. The objects of $\Cob$ are non-negative integers $g\geq 0$: the object $g$ refers to a compact, connected, oriented  surface $F_g$ of genus $g$ with one boundary component. The surface $F_g$ is fixed and will play the role of ``model'' surface.  Furthermore, we assume that the boundary component $\dep F_g$ is identified with $S^1$ and a base point $\star \in S^1 =\dep F_g$ is fixed. 

For any integers $g_+,g_-\geq 0$, a morphism $g_-\to g_+$ in $\Cob$ is a \emph{cobordism} from the surface $F_{g_{-}}$ to the surface $F_{g_{+}}$: specifically, this is an equivalence class of pairs $(M, m)$ consisting of a compact, connected, oriented 3-manifold $M$ and an orientation-preserving homeomorphism $m~ \colon  F ( g_{-}, g_{+} ) \to \dep M$, where
$$ F ( g_{-}, g_{+} ) := - F_{g_{-}} \cup_{S^1 \times \lbrace -1 \rbrace} \left( S^1 \times [-1, 1] \right) \cup_{S^1 \times \lbrace 1 \rbrace} F_{g_{+}}; $$
here two cobordisms $(M, m)$ and $(M', m')$ are said to be {\itshape equivalent} if there exists a homeomorphism $f~ \colon  M \to M'$  such that $m' = f|_{\dep M} \circ m$. Let $m_{\pm}~ \colon  F_{g_{\pm}} \to M$ be the composition of $m\vert_{F_{g_{\pm}}}$  with the inclusion of $\dep M$ into $M$, and  set $\partial_\pm M := m_\pm(F_{g_\pm})$:

$$
\labellist
\scriptsize\hair 2pt
\pinlabel {$\partial_+ M$} [r] at 1 186
 \pinlabel {$\partial_-M$} [r] at 0 75
 \pinlabel {$M$}  at 92 128
 \pinlabel {$m_+$} [r] at 91 222
 \pinlabel {$m_-$} [r] at 90 38
 \pinlabel {$F_{g_+}$} [r] at 2 255
 \pinlabel {$F_{g_-}$} [r] at 3 4
\endlabellist
\centering
\includegraphics[scale=0.29]{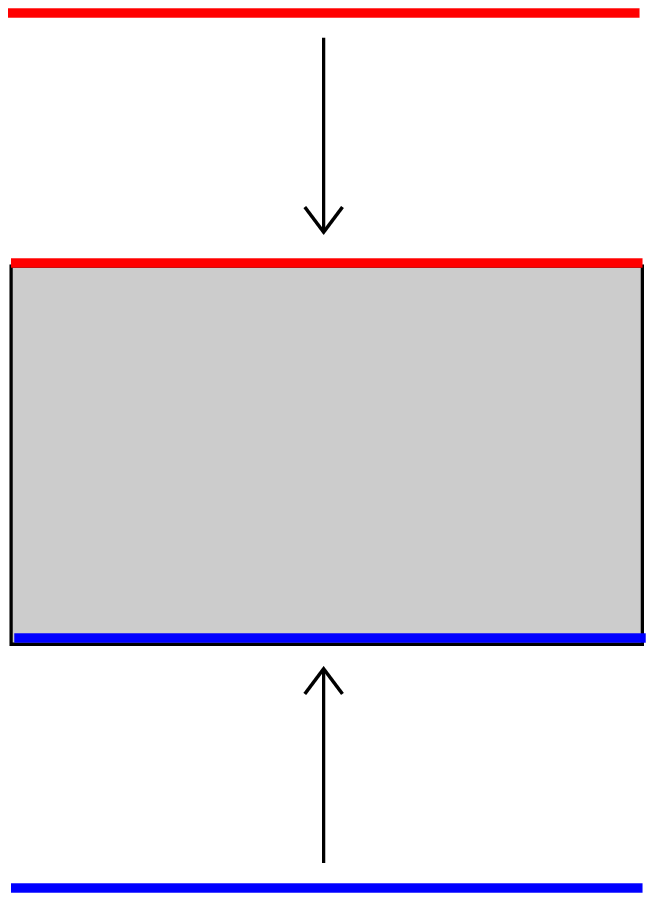}
$$

In the sequel, we will denote a cobordism simply by an upper-case letter $M,N,\dots$ meaning that the boundary-parametrization is denoted by the corresponding lower-case letter $m,n,\dots$ The composition $N \circ M$ of two cobordisms $M\in \Cob(g_-,g_+)$ and $N\in \Cob(h_-,h_+)$ is defined when $g_+=h_-$ by gluing $N$ \lq\lq on the top of\rq\rq{} $M$, i.e$.$ $\partial_+ M$ is identified with  $\partial_- N$ using the boundary parametrizations $m_+$ and $n_-$:

$$
\centre{\labellist
\scriptsize\hair 2pt
 \pinlabel {$N$}  at 92 128
 \pinlabel {$n_+$} [r] at 91 222
 \pinlabel {$n_-$} [r] at 90 38
 \pinlabel {$F_{h_+}$} [r] at 2 255
 \pinlabel {$F_{h_-}$} [r] at 3 4
\endlabellist
\centering
\includegraphics[scale=0.28]{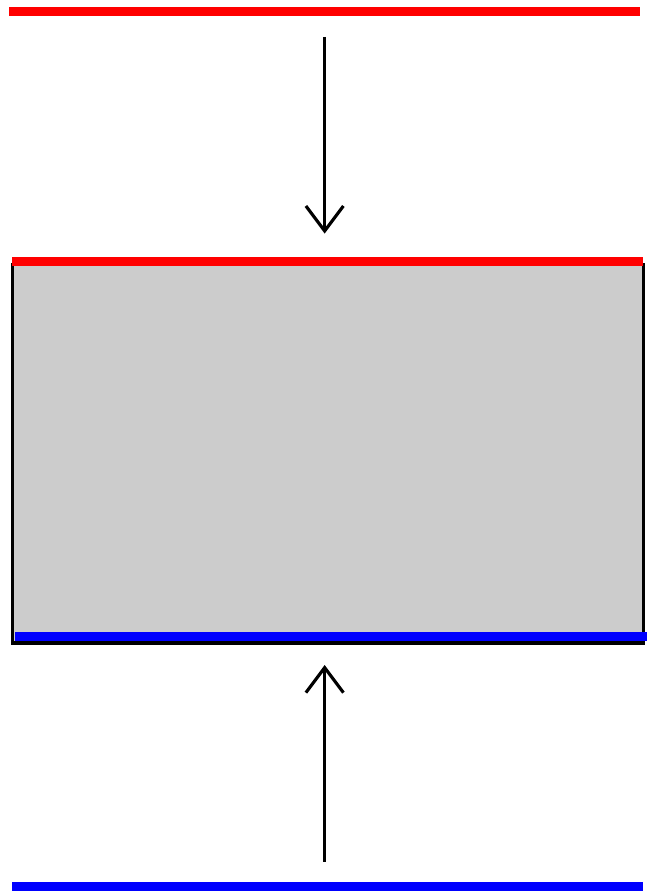}} \quad \circ \quad  
\centre{\labellist
\scriptsize\hair 2pt
 \pinlabel {$M$}  at 92 128
 \pinlabel {$m_+$} [r] at 91 222
 \pinlabel {$m_-$} [r] at 90 38
 \pinlabel {$F_{g_+}$} [r] at 2 255
 \pinlabel {$F_{g_-}$} [r] at 3 4
\endlabellist
\includegraphics[scale=0.28]{cobordism}} 
\quad := \quad 
\centre{\centering
\labellist
\scriptsize\hair 2pt
 \pinlabel {$F_{h_+}$} [r]  at 3 362
 \pinlabel {$F_{g_-}$} [r] at 2 4
 \pinlabel {$N$}  at 93 232
 \pinlabel {$M$}  at 92 126
 \pinlabel {$m_{-}$} [r] at 92 37
 \pinlabel {$n_+$} [r] at 91 331
\endlabellist
\centering
\includegraphics[scale=0.25]{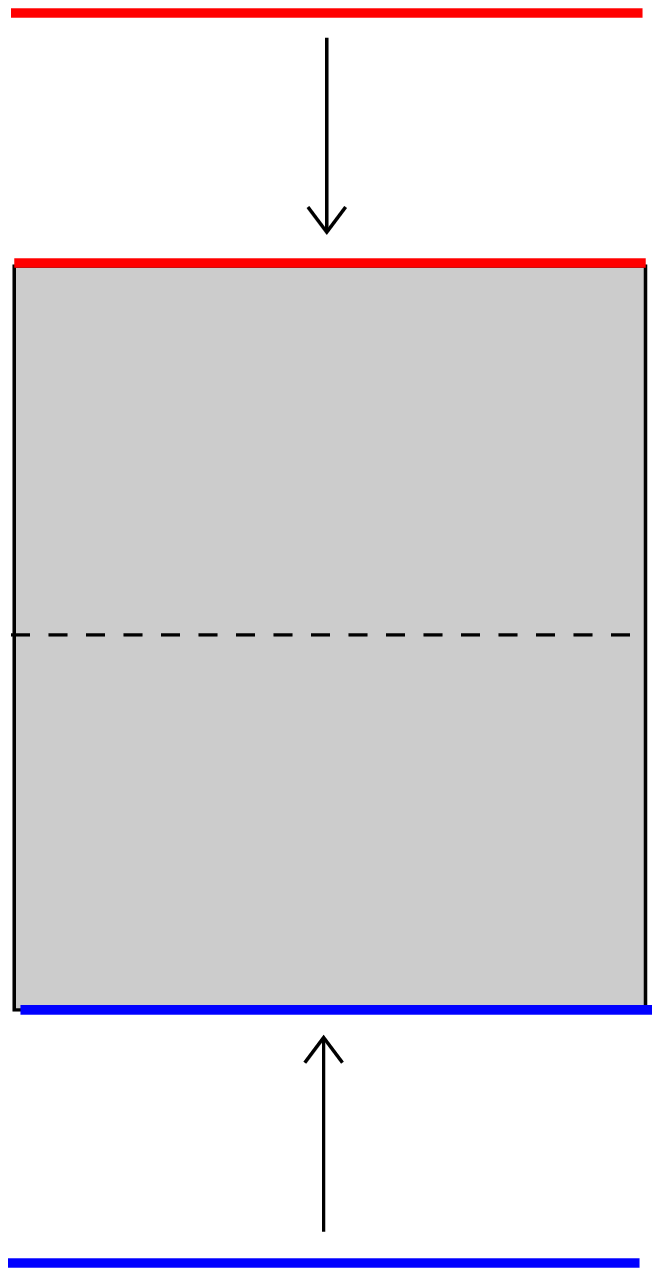}}
$$

For any integer $k \geq 0$, the identity of the object $k$  in $\Cob$ is the cylinder $F_k \times [-1, 1]$ with the boundary-parametrization defined by the identity maps. 

The category $\Cob$ can be enriched with a strict monoidal structure \cite{CrYe99}. 
\bgm (See Appendix A for a  brief review of the terminology of monoidal categories.) \egm 
We assume that, for any integer $g\geq 1$, the model surface $F_g$ is constructed by doing the iterated boundary-connected sum of $g$ copies of the model surface $F_1$ in genus $1$. Thus, for any  $g,k\geq 0$, the boundary-connected sum $F_g\, \sharp_\partial\, F_k$ is identified with $F_{g+k}$. The tensor product in the category $\Cob$ is defined  by $g \boxtimes k := g+k$ at the level of objects, and it is defined by $M \boxtimes N := M \sharp_\partial N$ at the level of morphisms:

$$
\centre{\labellist
\scriptsize\hair 2pt
 \pinlabel {$M$}  at 92 128
 \pinlabel {$m_+$} [r] at 91 222
 \pinlabel {$m_-$} [r] at 90 38
 \pinlabel {$F_{g_+}$} [r] at 2 255
 \pinlabel {$F_{g_-}$} [r] at 3 4
\endlabellist
\includegraphics[scale=0.3]{cobordism}} 
\quad \boxtimes \quad  
\centre{\labellist
\scriptsize\hair 2pt
 \pinlabel {$N$}  at 92 128
 \pinlabel {$n_+$} [r] at 91 222
 \pinlabel {$n_-$} [r] at 90 38
 \pinlabel {$F_{h_+}$} [r] at 2 255
 \pinlabel {$F_{h_-}$} [r] at 3 4
\endlabellist
\centering
\includegraphics[scale=0.3]{cobordism}} 
\quad := \qquad \qquad 
\centre{\labellist
\scriptsize \hair 2pt
 \pinlabel {$M$}   at 96 130
 \pinlabel {$N$}  at 277 132
 \pinlabel {$m_+ \sharp_\partial n_+$} [r] at 182 221
 \pinlabel {$m_- \sharp_\partial n_-$} [r] at 181 37
 \pinlabel {$F_{g_-+h_-}$} [r] at 2 5
 \pinlabel {$F_{g_++h_+}$} [r] at 2 255
\endlabellist
\centering
\includegraphics[scale=0.3]{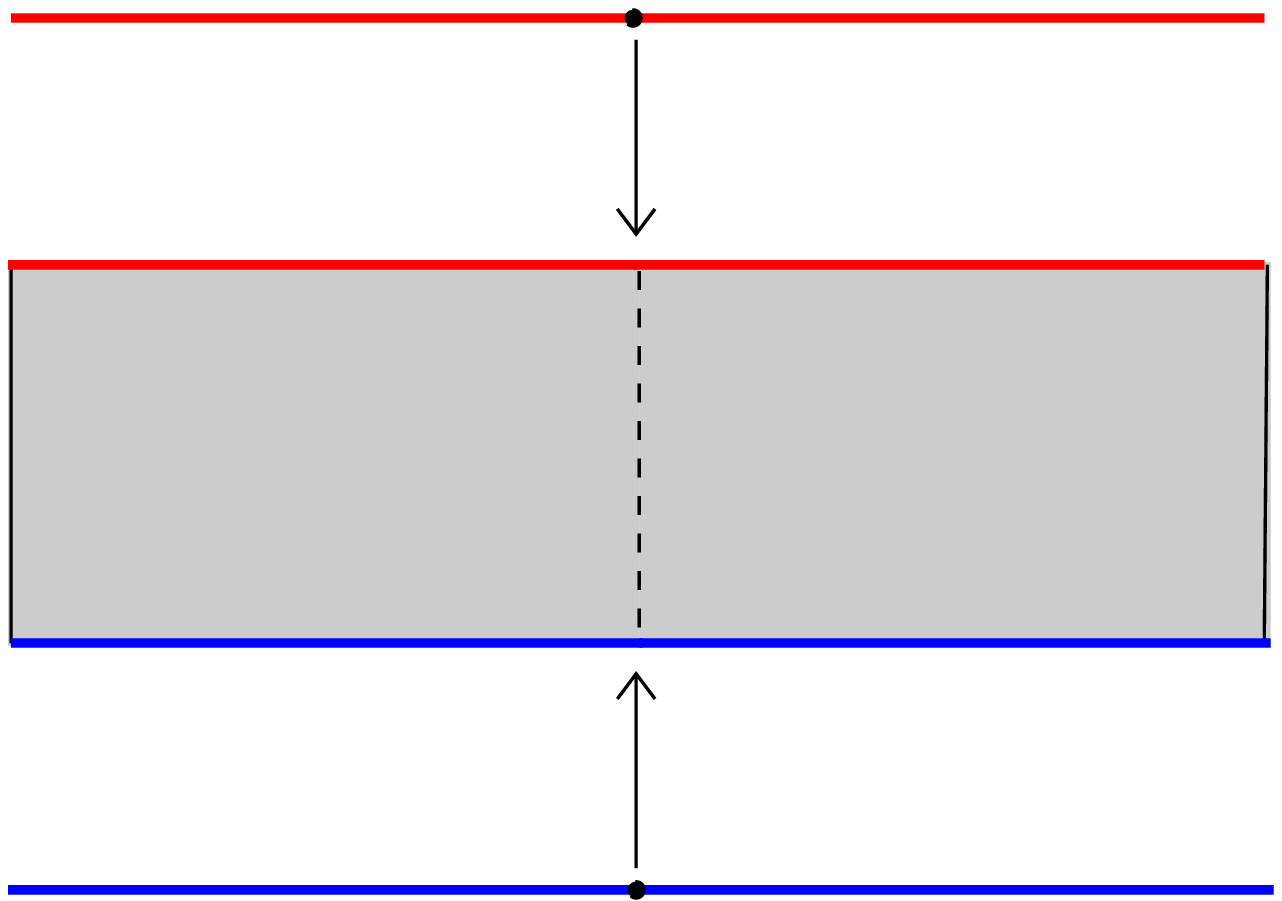}}
$$
\bgm The unit object of the monoidal category $\Cob$ is the integer $0$. \egm


\subsection{The category  $\Cob_G$}

Let $G$ be an abelian group. We now define the category $\Cob_{G}$ of {\itshape $3$-dimensional cobordisms with $G$-representations} following \cite{FlMa14}. The objects of $\Cob_{G}$ are pairs $(g, \gvf)$  consisting of an integer $g \geq 0$ and a group homomorphism $\gvf~ \colon  H_1 (F_g ) \to G$. A morphism $(g_{-}, \gvf_{-}) \rightarrow (g_{+}, \gvf_{+})$ in $\Cob_{G}$ is a pair $(M, \gvf)$ consisting of a cobordism $M \in \Cob(g_{-}, g_{+})$ and a group homomorphism $\gvf~ \colon  H_1 (M ) \to G$ such that $\gvf \circ m_{\pm} = \gvf_{\pm}$. The composition of two morphisms $(M, \gvf) \in \Cob_{G}((g_{-}, \gvf_{-}), (g_{+}, \gvf_{+}))$ and $(N, \psi) \in \Cob_{G}((h_{-}, \psi_{-}), (h_{+}, \psi_{+}))$ such that $(g_{+}, \gvf_{+}) = (h_{-}, \psi_{-})$, is defined by 
$$ (N, \psi) \circ (M, \gvf) := (N \circ M, \psi + \gvf)$$
where $N \circ M$ is the composition in $\Cob$ and $\psi + \gvf~ \colon  H_1 ( N \circ M ) \rightarrow G$ is defined from $\gvf$ and $\psi$ by using the Mayer--Vietoris theorem.

The strict monoidal structure of $\Cob$ extends to $\Cob_G$ in the following way. The tensor product of objects is defined by $ (g,\varphi) \boxtimes (h,\psi) := (g+h,\varphi \oplus \psi) $ where $H_1(F_{g+h})= H_1(F_g \sharp_{\partial} F_h)$ is identified with $H_1(F_g) \oplus H_1(F_h)$ using the Mayer--Vietoris theorem; the tensor product of morphisms is defined by $ (M,\varphi) \boxtimes (N,\psi) :=  (M \sharp_\partial N, \varphi \oplus \psi) $ where $H_1(M \sharp_{\partial} N)$ is identified with $H_1(M)\oplus H_1(N)$ using the Mayer--Vietoris theorem again.
\bgm The unit object of the monoidal category $\Cob_G$ is the pair consisting of the integer $0$ 
and the trivial group homomorphism $H_1(F_0) \to G$. \egm

 
\subsection{The category $\Lagr_R$}

We review the category of \emph{Lagrangian relations} introduced by Cimasoni and Turaev in~\cite{CiTu05I}. We first recall from~\cite[Section 2.1]{CiTu05I} some basic terminology. Let $R$ be a commutative ring without zero-divisors, and let  $R \rightarrow R,r \mapsto \overline{r}$ be an involutive ring homomorphism. A {\itshape skew-Hermitian form} on a $R$-module $H$ is a  map $\pai~ \colon  H \times H \to R$ which is {\itshape sesquilinear} and {\itshape skew-symmetric} in the sense that
\begin{flalign*}
(i)    	\quad  &\pai ( r x + r' x', y ) = r\, \pai(x, y) + r'\, \pai(x', y) &\\
(i')    \quad  &\pai ( y,r x + r' x') = \overline{r}\, \pai(y,x) + \overline{r'}\, \pai(y,x') &\\
(ii) 	\quad  &\pai (x, y) = - \overline{\pai (y, x)} &
\end{flalign*}
for all $ x, x', y \in H$ and for all $r, r' \in R$. Such a form $\pai$ is {\itshape non-degenerate} 
if the adjoint map  $ H \to \Hom_R(H,R)$ defined by  $x \mapsto \pai(x, \cdot)$ is injective.
A {\itshape skew-Hermitian $R$-module} is a finitely generated $R$-module $H$ equipped with a non-degenerate skew-Hermitian form $\pai$. 
(In particular, the $R$-module $H$ has no torsion.) 
Given a submodule $A$ of $H$, one can consider its {\itshape annihilator} with respect to $\pai$ 
$$\Ann(A) := \lbrace x \in H \, : \, \pai(x, A) = 0  \rbrace$$
and its {\itshape closure} 
$$\operatorname{cl}(A) := \lbrace x \in H \, : \, \exists r \in R\setminus\{0\}, \, r x \in A \rbrace.$$
A submodule $A$ of $H$ is said to be {\itshape Lagrangian} if $A = \Ann(A)$.

We now recall some material from \cite[Section 2.3]{CiTu05I}. Let $(H_1,\pai_1)$ and $(H_2,\pai_2)$ be skew-Hermitian $R$-modules. A {\itshape Lagrangian relation} between $(H_1,\pai_1)$ and $(H_2,\pai_2)$ is a submodule $N$ of $H_1 \oplus H_2$ which is Lagrangian with respect to the skew-Hermitian form $(-\pai_1)\oplus \pai_2$; in this case, we denote  $N~ \colon  (H_1,\pai_1) \Rightarrow (H_2,\pai_2)$. According to  \cite[Theorem 2.7]{CiTu05I}, there is a category  $\Lagr_{R}$ whose objects are skew-Hermitian $R$-modules, and whose morphisms are Lagrangian relations and are composed in the following way. The {\itshape composition} $N_2 \circ N_1$ of two Lagrangian relations $N_1~ \colon  (H_1,\pai_1) \Rightarrow (H_2,\pai_2)$ and $N_2~ \colon  (H_2,\pai_2) \Rightarrow (H_3,\pai_3)$ is the closure $\operatorname{cl}(N_2 \check \circ N_1)$ of the following submodule of $H_1 \oplus H_3$:
 \begin{equation} \label{eq:pre-composition}
 N_2 \check \circ N_1 := \big\lbrace (h_1 ,  h_3) \in H_1 \oplus H_3 \, :  \, \exists h_2 \in H_2, \ (h_1,  h_2 ) \in N_1 \; \text{and} \; ( h_2, h_3) \in N_2 \big\rbrace
\end{equation}
Note that, for any skew-Hermitian $R$-module $(H,\pai)$, the diagonal $ \big\lbrace (h,  h) \in H \oplus H \, \big| \, h \in H \big\rbrace$ is a Lagrangian relation $(H,\pai) \Rightarrow (H,\pai)$, and constitutes  the identity of the object $(H,\pai)$ in the category $\Lagr_{R}$. 

We finally outline the relationship between Lagrangian submodules and graphs of unitary isomorphisms following \cite[Section 2.4]{CiTu05I}. Let  $Q := Q(R)$ be the field of fractions of $R$: there is a unique way to extend the involution $r \mapsto \overline{r}$ of $R$  to a ring homomorphism of $Q$. For any skew-Hermitian $R$-module $(H,\pai)$, 
\bgm consider the \emph{rationalization} $H_Q := Q \otimes_R H$ of $H$: \egm 
note that,  since $H$ is  torsion-free, it embeds into $H_Q$ by the map $h \mapsto 1 \otimes h$. Let $\pai~ \colon H_Q \times H_Q \to Q$ be the extension of $\pai$  defined by
$$ \pai( q \otimes x, q' \otimes x' ) := q\, \overline{q'}\, \pai(x,x') $$
for any $x,x'\in H$ and $q,q' \in Q$. 

A {\itshape unitary $Q$-isomorphism}  (respectively, {\itshape unitary $R$-isomorphism}) between  two skew-Hermitian $R$-modules $(H_1,\pai_1)$ and $(H_2,\pai_2)$ is a $Q$-linear isomorphism $\psi~ \colon (H_1)_Q \to (H_2)_Q$ (respectively, a $R$-linear isomorphism $\psi~ \colon H_1 \to H_2$) such that ${\pai_2 \circ (\psi \times \psi) = \pai_1}$. Let $\UniQ$ (respectively, $\Uni$) be the category whose objects are skew-Hermitian $R$-modules and whose morphisms are unitary $Q$-isomorphisms (respectively, unitary $R$-isomorphisms). 
According to  \cite[Theorem 2.9]{CiTu05I}, there are embeddings of categories $\Uni \hookrightarrow \UniQ$, $\Uni \hookrightarrow \Lagr_{R}$ 
and $\UniQ \hookrightarrow \Lagr_{R}$ which fit in the commutative diagram
\begin{center}
	\begin{tikzpicture}[>=angle 90,scale=2.2,text height=1.5ex, text depth=0.25ex]
	\node (a0) at (0,1) {$\Uni$};
	\node (a1) [right=of a0] {$\UniQ$};
	\node (a2) [below=of a1] {$\Lagr_R$};
	
	\draw[right hook->]
	(a0) edge (a1)
	(a1) edge (a2)
	(a0) edge (a2);	
	\end{tikzpicture}
\end{center}
	
and which are  defined by $\psi \mapsto \bgm \Id_Q \egm \otimes_R \psi$, $\psi \mapsto \gG_\psi$ and $\psi \mapsto \gG_\psi^{\operatorname{r}}$, respectively. 
Here $\gG_\psi$ denotes the \bgm set-theoretical \egm graph of a $R$-linear isomorphism  \bgm $\psi: H_1 \to H_2$ defined by 
$$ 
\gG_{\psi} := \big\lbrace (h_1 , h_2) \in H_1 \oplus H_2 \,  : \, h_2 =\psi(h_1) \big\rbrace,
$$ 
\egm
while $\gG^{\operatorname{r}}_{\psi}$ denotes the {\itshape restricted graph} of a $Q$-linear isomorphism $\psi~ \colon  (H_1)_Q \to (H_2)_Q$
and is defined \bgm similarly \egm by
$$ 
\gG^{\operatorname{r}}_{\psi} := \big\lbrace (h_1 , h_2) \in H_1 \oplus H_2 \,  : \, h_2 =\psi(h_1) \big\rbrace.
$$

 
\subsection{The category $\pLagr_R$}  \label{pLagr}

Let $R$ be a commutative ring without zero-divisors, and let  $R \rightarrow R,r \mapsto \overline{r}$ be an involutive ring homomorphism.
We introduce a refinement of the category $\Lagr_R$, which seems to be new.

A \emph{pointed skew-Hermitian $R$-module} is a skew-Hermitian $R$-module $(H,\rho)$ equipped with a distinguished element $s \in H_Q \bgm = Q \otimes_R H \egm   $ \bgm of its rationalization, \egm
satisfying $\rho(s,s)=0$  and $\rho(s,H) \subset R$, where $\rho~ \colon H_Q \times H_Q \to Q$ denotes here the extension of $\rho$ to $H_Q $.  
 Let $(H_1,\pai_1,s_1)$ and $(H_2,\pai_2,s_2)$ be pointed skew-Hermitian modules: a \emph{pointed Lagrangian relation} $N~ \colon  (H_1,\pai_1,s_1) \Rightarrow (H_2,\pai_2,s_2)$ is a Lagrangian submodule $N$ of $(H_1 \oplus H_2, (-\pai_1) \oplus  \pai_2)$ such that 
$$(s_1,s_2) \in N_Q := Q \otimes_R N \subset (H_1)_Q \oplus (H_2)_Q.$$
 The  composition of Lagrangian relations induces a composition rule for pointed Lagrangian relations. Thus we get the category $\pLagr_R$ of \emph{pointed Lagrangian relations}.  Similarly, we can define some refinements $\pUni$ and $\pUniQ$ of the categories  $\Uni$ and $\UniQ$, respectively, by requiring that (the extensions of) the unitary $R$-isomorphisms and  the unitary $Q$-isomorphims preserve the distinguished elements. All these categories fit together into the commutative diagram
\begin{center}
	\iftikziii
\else
	\includegraphics{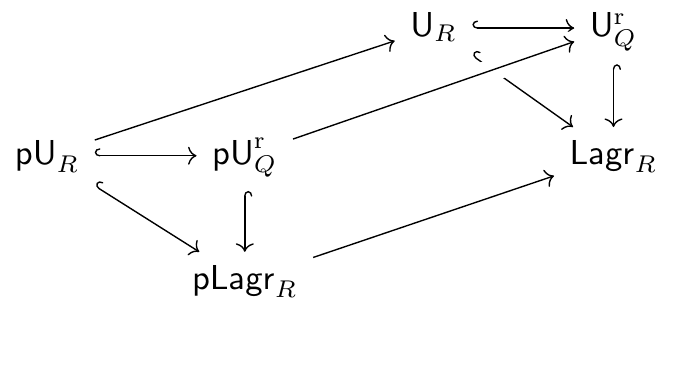}
	\fi
\end{center} \up
where the arrows $\pLagr_R \to \Lagr_R$, $\pUni \to\Uni$, $\pUniQ \to \UniQ$ denote the forgetful functors.

We now enrich the category $\pLagr_R$  with a monoidal structure. 
\bgm (See Appendix A for a  brief review of the terminology of monoidal categories.) \egm
For any  pointed skew-Hermitian modules $(H,\rho,s)$ and $(H',\rho',s')$, let  $\rho\, {{}_{s}\oplus_{s'}} \rho'~ \colon (H \oplus H') \times (H \oplus H') \to R$  be  defined by
$$(\rho\, {{}_{s}\oplus_{s'}}\rho')(h_1+h_1',h_2+h'_2)  := \rho(h_1,h_2)+   \rho'(h_1',h'_2) + \rho(h_1,s)\, \rho'(s',h'_2)  - \rho(s,h_2)\, \rho'(h'_1,s')$$
for any $h_1,h_2\in H$ and  $h'_1,h'_2\in H'$.

\begin{lemma} \label{lem:mon_ob}
For any objects $(H,\rho,s), (H',\rho',s')$ in $\pLagr_R$, the triple
$(H\oplus H', \rho\, {{}_{s}\oplus_{s'}} \rho', {s+s'})$ defines another object  in $\pLagr_R$.
\end{lemma}

\begin{proof}
The fact that $\rho\, {{}_{s}\oplus_{s'}} \rho'$ is sesquilinear and skew-symmetric follows easily from the same properties for $\rho$ and $\rho'$.
To prove  its non-degeneracy, assume that m $h_1\in H$ and $h'_1 \in H'$ satisfy
\begin{equation} \label{eq:zero}
\forall h_2 \in H, \ \forall h'_2 \in H', \quad  \rho(h_1,h_2) + \rho'(h_1',h'_2) + \rho(h_1,s)\, \rho'(s',h'_2) - \rho(s,h_2)\, \rho'(h'_1,s') = 0.
\end{equation} 
Taking $h_2  =  s$ and $h'_2 = 0$ in \eqref{eq:zero}, we obtain  $\rho(h_1,s) = 0$. Next, taking $h_2  = 0$ in \eqref{eq:zero}, we obtain $\rho'(h'_1,h'_2)=0$ for any $h'_2 \in H'$; it follows that $h'_1=0$. Similarly, we obtain $h_1 = 0$. This shows that  $(H\oplus H',\rho\, {{}_{s}\oplus_{s'}} \rho')$ is a skew-Hermitian $R$-module. That it can be upgraded to a pointed skew-Hermitian $R$-module by adjoining the element $s+s'$  is easily verified.
\end{proof}

\begin{lemma} \label{lem:mon_mor}
For any two morphisms $N~ \colon  (H,\pai,s) \Rightarrow (K,\tau,t)$ and $N'~ \colon  (H',\pai',s') \Rightarrow (K',\tau',t')$  in $\pLagr_R$, the submodule 
$$ N \oplus N' \subset (H \oplus K) \oplus (H' \oplus K') \simeq (H \oplus H') \oplus (K \oplus K') $$
is a morphism from $(H\oplus H', \rho\, {{}_{s}\oplus_{s'}} \rho', {s+s'})$ to $(K\oplus K', \tau\, {{}_{t}\oplus_{t'}} \tau', {t+t'})$ in $\pLagr_R$.
\end{lemma}

\begin{proof}
Clearly $(s+s')+(t+t')= (s+t)+(s'+t')$ belongs to $N_Q \oplus N'_Q$, and it remains to verify that $N \oplus N'$ is Lagrangian with respect to the skew-Hermitian form 
$$ \kappa:= -( \rho\, {{}_{s}\oplus_{s'}} \rho') \oplus (\tau\, {{}_{t}\oplus_{t'}} \tau')~ \colon  (H \oplus H' \oplus K \oplus K') \times  (H \oplus H' \oplus K \oplus K')  \longrightarrow R. $$
For any $h_1,h_2 \in H$,  $h'_1,h'_2 \in H'$,  $k_1,k_2 \in K$,  $k'_1,k'_2 \in K'$, we have
\begin{eqnarray}
\label{kappa_expanded} && \kappa(h_1+h'_1+k_1+k'_1, h_2+h'_2+k_2+k'_2) \\
\notag & = & 
-\rho(h_1,h_2) - \rho'(h'_1,h'_2) - \rho(h_1,s)\, \rho'(s',h'_2)  + \rho(s,h_2)\, \rho'(h'_1,s') \\
\notag && + \tau(k_1,k_2) + \tau'(k'_1,k'_2) + \tau(k_1,t)\, \tau'(t',k'_2)  - \tau(t,k_2)\, \tau'(k'_1,t').
\end{eqnarray}
To prove that $N \oplus N' \subset \Ann(N \oplus N' )$, consider some arbitrary elements 
$$ (h_1+k_1)+ (h'_1+k'_1) \in  N \oplus N' \subset (H \oplus K) \oplus (H '\oplus K') $$
and
$$ (h_2+k_2)+ (h'_2+k'_2) \in  N \oplus N' \subset (H \oplus K) \oplus (H '\oplus K'). $$
 Then it can be verified that the eight summands in \eqref{kappa_expanded} cancel pairwisely: 
for instance, the facts that $h_1 + k_1, s+ t\in N_Q$ and $s' + t',  h'_2+k'_2 \in N'_Q$ imply that $\rho(h_1,s) = \tau(k_1,t)$
and $\rho'(s',h'_2) =\tau'(t',k'_2)$, respectively, so that $ \rho(h_1,s)\, \rho'(s',h'_2) =  \tau(k_1,t)\, \tau'(t',k'_2)$.
 We conclude that $\kappa(h_1+h'_1+k_1+k'_1, h_2+h'_2+k_2+k'_2)=0$.

To prove now that $ \Ann(N \oplus N' ) \subset N \oplus N'$, consider an arbitrary element
$$ (h_2+k_2)+ (h'_2+k'_2) \in \Ann( N \oplus N' )\subset (H \oplus K) \oplus (H '\oplus K'). $$
Using \eqref{kappa_expanded}, we  obtain the following: 
\begin{equation}
\label{eq:ppp}  \forall h_1+k_1 \in N_Q, \quad    -\rho(h_1,h_2)  - \rho(h_1,s)\, \rho'(s',h'_2)  + \tau(k_1,k_2)  + \tau(k_1,t)\, \tau'(t',k'_2) =  0   , 
\end{equation}
\begin{equation} \label{eq:ppp'}
 \forall h'_1+k'_1 \in N'_Q, \quad  - \rho'(h'_1,h'_2) + \rho(s,h_2)\, \rho'(h'_1,s')  + \tau'(k'_1,k'_2)  - \tau(t,k_2)\, \tau'(k'_1,t')    =  0.  
\end{equation}
Applying \eqref{eq:ppp} to $h_1+k_1  =  s+t$ gives $\tau(t,k_2)= \rho(s,h_2)$ and, since $h'_1+k'_1, s'+t'\in N'_Q$, we also have $\rho'(h'_1,s')= \tau'(k'_1,t')$.  Hence \eqref{eq:ppp'} now gives
 $ -\rho'(h'_1,h'_2) + \tau'(k'_1,k'_2) =0$  for all $h'_1+k'_1 \in N'$,
and it follows that $h'_2+k'_2 \in \Ann(N')\subset N'$.
We obtain in a similar way that $h_2+k_2 \in \Ann(N)\subset N$ and we conclude that $(h_2+k_2)+ (h'_2+k'_2) \in N \oplus N'$.
\end{proof}

\begin{proposition} \label{prop:mono}
There \bgm exists \egm  a monoidal structure on the category $\pLagr_R$ \bgm whose tensor product is \egm  defined by 
\begin{equation} \label{eq:tp_ob}
(H,\pai,s) \boxtimes (H',\pai',s') := (H\oplus H', \rho\, {{}_{s}\oplus_{s'}} \rho', s+ s')
\end{equation}
for any objects  $(H,\rho,s), (H',\rho',s')$ in $\pLagr_R$, and by 
\begin{equation} \label{eq:tp_mor}
N \boxtimes N' := N \oplus N'  
\end{equation}
for any morphisms  $N~ \colon  (H,\pai,s) \Rightarrow (K,\tau,t)$  and $N'~ \colon  (H',\pai',s') \Rightarrow (K',\tau',t')$  in $\pLagr_R$.
\end{proposition}

\begin{proof}
Using Lemma \ref{lem:mon_ob} and Lemma \ref{lem:mon_mor}, 
we first verify that \eqref{eq:tp_ob} and  \eqref{eq:tp_mor} define a bifunctor $\boxtimes~ \colon   \pLagr_R \times \pLagr_R \to \pLagr_R$. That $\boxtimes$ maps the identities to the identities is clear, and we only need to check that $\boxtimes$ preserves the compositions.
Consider some morphisms in $\pLagr_R$
$$
\mathcal{H}_1 \stackrel{N_1\ }{\Longrightarrow} \mathcal{H}_2 \stackrel{N_2\ }{\Longrightarrow} \mathcal{H}_3
\quad \hbox{and} \quad \mathcal{H}'_1 \stackrel{N'_1\ }{\Longrightarrow} \mathcal{H}'_2 \stackrel{N'_2\ }{\Longrightarrow} \mathcal{H}'_3
$$
where $\mathcal{H}_i:=(H_i,\rho_i,s_i)$ and $\mathcal{H}'_i:=(H'_i,\rho'_i,s'_i)$ for each $i\in\{1,2,3\}$. 
The submodule $(N_2 \circ N_1) \boxtimes (N'_2 \circ N'_1)$ consists of those elements $(h_1 + h'_1) + (h_3+h'_3)$  of $(H_1 \oplus H'_1) \oplus  (H_3 \oplus H'_3)$ such that 
$h_1 +h_3 \in \cl (N_2 \check \circ N_1)$  and $h'_1 +h'_3 \in \cl (N'_2 \check \circ N'_1)$
or, equivalently, such that
$$
\exists r\in R\setminus \{0\}, \ \exists h_2 \in H_2, \ \left\{ \begin{array}{l} rh_1+h_2 \in N_1 \\  h_2 +r h_3 \in N_2 \end{array}\right.
\quad \hbox{and} \quad 
\exists r'\in R\setminus \{0\}, \ \exists h'_2 \in H'_2, \ \left\{ \begin{array}{l} r'h'_1+h'_2 \in N'_1 \\  h'_2 +r' h'_3 \in N'_2 \end{array}\right.;
$$
on the other hand, the submodule $(N_2 \boxtimes N'_2) \circ (N_1 \boxtimes N'_1)$  consists of those elements $(h_1 +h'_1) + (h_3+h'_3)$  of ${(H_1 \oplus H'_1)} \oplus  {(H_3 \oplus H'_3)}$ such that
$$
\exists r\in R\setminus \{0\}, \ \exists (h_2+h'_2) \in (H_2\oplus H'_2), \ \left\{  \begin{array}{l} r(h_1+h'_1) + (h_2+h'_2) \in N_1 \boxtimes N'_1 \\  (h_2+h'_2) + r (h_3+h'_3) \in N_2 \boxtimes N'_2 
\end{array} \right.
$$
or, equivalently, such that
$$
\exists r\in R\setminus \{0\}, \ \exists (h_2+h'_2) \in (H_2 \oplus H'_2), \ \left\{  \begin{array}{l} r h_1+h_2  \in N_1,   \ r h'_1+h'_2  \in N'_1 \\  h_2+ rh_3 \in N_2, \   h'_2+ rh'_3 \in N'_2
\end{array} \right.;
$$
it follows that $(N_2 \circ N_1) \boxtimes (N'_2 \circ N'_1) = (N_2 \boxtimes N'_2) \circ (N_1 \boxtimes N'_1)$.

The \emph{trivial} pointed skew-Hermitian module, namely $\mathcal{I}:=(\lbrace 0 \rbrace , 0, 0)$, will be the unit object in $\pLagr_R$. For any pointed skew-Hermitian module $\mathcal{H}:=(H,\rho,s)$, 
the obvious isomorphisms $\bgm  \lbrace 0 \rbrace  \oplus H    \egm \to H$ and $\bgm H \oplus \lbrace 0 \rbrace \egm \to H$ define some isomorphisms 
$$
L_{\mathcal{H}}~ \colon  \bgm \mathcal{I} \boxtimes \mathcal{H}     \egm \longrightarrow \mathcal{H}
\quad \hbox{and} \quad R_{\mathcal{H}}~ \colon   \bgm  \mathcal{H} \boxtimes \mathcal{I}\egm  \longrightarrow \mathcal{H}
$$ 
in the category $\pUni$. By applying the ``graph'' functor $\pUni \to \pLagr_R$, we obtain some isomorphisms  $L_{\mathcal{H}}, R_{\mathcal{H}}$ in $\pLagr_R$ which are natural in $\mathcal{H}$. They will be the unit constraints in $\pLagr_R$.

To define the associativity constraints in $\pLagr_R$, consider some pointed skew-Hermitian modules  $\mathcal{H} := (H, \rho, s)$, $\mathcal{H}' := (H', \rho', s')$ and $\mathcal{H} '':= (H'', \rho'', s'')$. We claim that the obvious isomorphism of $R$-modules $(H \oplus H') \oplus H'' \to H \oplus (H' \oplus H'')$ defines an isomorphism
$$
A_{\mathcal{H},\mathcal{H'},\mathcal{H''}}~ \colon  \big( \mathcal{H} \boxtimes \mathcal{H'}\big) \boxtimes \mathcal{H''}
\longrightarrow  \mathcal{H} \boxtimes \big( \mathcal{H'}  \boxtimes \mathcal{H''} \big)
$$
in the category $\pUni$.
Indeed, for all $h_1, h_2 \in H$, $h'_1, h'_2 \in H'$ and $h''_1, h''_2 \in H''$, we have 
\begin{eqnarray}
\notag &&\big( \rho\, {{}_{s}\oplus_{s'+s''}} ( \rho' {{}_{s'}\oplus_{s''}} \rho'' ) \big)(h_1+h_1'+h_1'',h_2+h'_2+h''_2)\\ 
\notag &=&  \rho(h_1, h_2) + ( \rho'  {{}_{s'}\oplus_{s''}} \rho'' ) (h'_1 + h''_1, h'_2 + h''_2) + \rho(h_1,s)\, ( \rho'  {{}_{s'}\oplus_{s''}} \rho'' )(s' + s'', h'_2 + h''_2) \\
\notag && - \rho(s,h_2)\, ( \rho'  {{}_{s'}\oplus_{s''}} \rho'' ) (h'_1 + h''_1, s' + s'') \\
\notag 	&=&  \rho(h_1, h_2) + \rho'(h'_1, h'_2) + \rho''(h''_1, h''_2) + \rho'(h'_1, s') \rho''(s'',h''_2) - \rho'(s', h'_2) \rho''(h''_1, s'') \\
\label{eq:un}	&&  + \rho(h_1, s)\, \big( \rho'(s', h'_2) + \rho''(s'', h''_2) \big) - \rho(s, h_2)\, \big( \rho'(h'_1, s') + \rho''(h''_1, s'') \big)
\end{eqnarray}
where  the last equality uses the fact that $\rho'(s', s') = \rho''(s'', s'') = 0$; on the other hand, we have
\begin{eqnarray}
\notag &&\big( (\rho\, {{}_{s}\oplus_{s'}} \rho')\, {}_{s+s'}\!\oplus_{s''}\rho''\big)(h_1+h_1'+h_1'',h_2+h'_2+h''_2)\\
\notag &=& (\rho\, {{}_{s}\oplus_{s'}} \rho') (h_1 + h'_1, h_2 + h'_2) + \rho''(h''_1, h''_2) + (\rho\, {{}_{s}\oplus_{s'}} \rho') (h_1 + h'_1, s + s')\, \rho''(s'',h''_2) \\
\notag &&\quad -  (\rho\, {{}_{s}\oplus_{s'}} \rho')(s + s', h_2 + h'_2)\, \rho''(h''_1, s'') \\
\notag &=& \rho(h_1, h_2) + \rho'(h'_1, h'_2) + \rho(h_1, s)\, \rho'(s', h'_2) - \rho(s, h_2)\, \rho'(h'_1,s') + \rho''(h''_1, h''_2)  \\
\label{eq:deux} & & + \big( \rho(h_1, s) + \rho'(h'_1, s') \big)\, \rho''(s'', h''_2) - \big( \rho(s, h_2) + \rho'(s', h'_2) \big)\, \rho''(h''_1, s'')
\end{eqnarray}
where the last equality uses the fact that $\rho(s, s) = \rho'(s', s') = 0$. Expanding  \eqref{eq:un} and \eqref{eq:deux}, we obtain that $\rho \,{{}_{s}\oplus_{s'+s''}} ( \rho' {{}_{s'}\oplus_{s''}} \rho'' ) =   (\rho {{}_{s}\oplus_{s'}} \rho')\, {}_{s+s'}\!\oplus_{s''}\rho''$ and  our claim follows. By applying the ``graph'' functor $\pUni \to \pLagr_R$, 
we obtain an isomorphism  $A_{\mathcal{H},\mathcal{H'},\mathcal{H''}}$ in $\pLagr_R$ which is  natural in $\mathcal{H},\mathcal{H'},\mathcal{H''}$.
That the associativity constraints $A$ and the unit constraints $L,R$ satisfy the coherence conditions of a monoidal category   is clear, since they are obviously satisfied in the category~$\pUni$. 
\end{proof}

\begin{remark}  \label{rem:not_strict}
Formally speaking, the  monoidal category  $\pLagr_R$ is not strict. But, since  its associativity and unit constraints arise from canonical bijections in set theory, we will assume in the sequel that  $\pLagr_R$ is strict monoidal.
\end{remark}

\begin{remark} \label{rem:monoidal}
The category $\Lagr_R$ itself has a strict monoidal structure, \bgm whose tensor product \egm 
is defined by $(H,\rho) \boxtimes (H',\rho') := (H \oplus H', \rho \oplus \rho')$ at the level of objects and by $N\boxtimes N':= N \oplus N'$ at the level of morphisms. The embedding of categories $\Lagr_R \to \pLagr_R$ that is defined by $(H,\rho) \mapsto (H,\rho,0)$ at the level of objects, and that  is the identity at the level of morphisms, is \bgm a strict monoidal functor. \egm
\end{remark}


\section{Intersection forms with twisted coefficients} \label{IntersectionForm}

Let $g\geq 0$ be an integer and consider the compact,  connected, oriented surface $F_g$ of genus $g$ with one boundary component.
  In this section, we review Turaev's ``homotopy intersection pairing'' on~$F_g$,
which is a non-commutative version of  Reidemeister's ``equivariant intersection forms''.
 An  advantage of the former with respect to the latter is a simpler definition, avoiding the use of covering spaces
and allowing for a straightforward verification of the main properties.


\subsection{The homotopy intersection pairing $\gl$}

Set $\pi := \pi_1(F_g,\star)$ where $\star \in \partial F_g$. We denote by $\ZZ[\pi]$ the group ring of $\pi$ and we denote by $a \mapsto \overline{a}$  the {\itshape antipode} of $\ZZ[\pi]$, which is the anti-homomorphism of rings defined by  $\overline{x} = x^{-1}$ for any $x\in \pi$. The {\itshape homotopy intersection pairing} of $F_g$ is the pairing 
$$
\gl~ \colon \ZZ[\pi] \times \ZZ[\pi] \longrightarrow \ZZ[\pi]
$$ 
introduced by Turaev in \cite{Tu78}. Recall that the map $\gl$ is $\ZZ$-bilinear and  $\gl(a,b) \in \ZZ[\pi]$ is defined as follows  for any $a,b\in \pi$. 
Let~$\nu$ be the oriented boundary curve of $F_g$. Let $l, r\in \partial  F_g$ be some additional points  such that $l < \star < r$ along $\nu$. Given an oriented path $\gamma$ in $F_g$  and two simple points $p<q$ along $\gamma$, we denote by $\gamma_{pq}$  the arc in $\gamma$ connecting $p$ to $q$, while the same arc with the opposite orientation is denoted by $\overline{\gamma}_{qp}$. Let $\alpha$ be a loop based at $l$ such that $\overline{\nu}_{\star l} \alpha \nu_{l \star}$ represents $a$ and let $\beta$ be a loop based at $r$ such that ${\nu}_{\star r} \beta \overline{\nu}_{r \star}$ represents $b$; we assume that these loops  are in transverse position and that $\alpha\cap \beta$ only consists of simple points of $\alpha$ and $\beta$:
$$
\labellist
\scriptsize\hair 2pt
 \pinlabel {$l$} [t] at 172 10
 \pinlabel {$r$} [t] at 538 10
 \pinlabel {$\alpha$} [tr] at 96 69
 \pinlabel {$\beta$} [tr] at 459 80
 \pinlabel {$p$} [b] at 347 187
 \pinlabel {$F_g$}  at 35 378
 \pinlabel {$\star$} at 355 -9
 \pinlabel {$\circlearrowleft$}  at 661 46
 \pinlabel {$\nu$} [t] at 663 3
\endlabellist
\centering
\includegraphics[height=3cm,width=6cm]{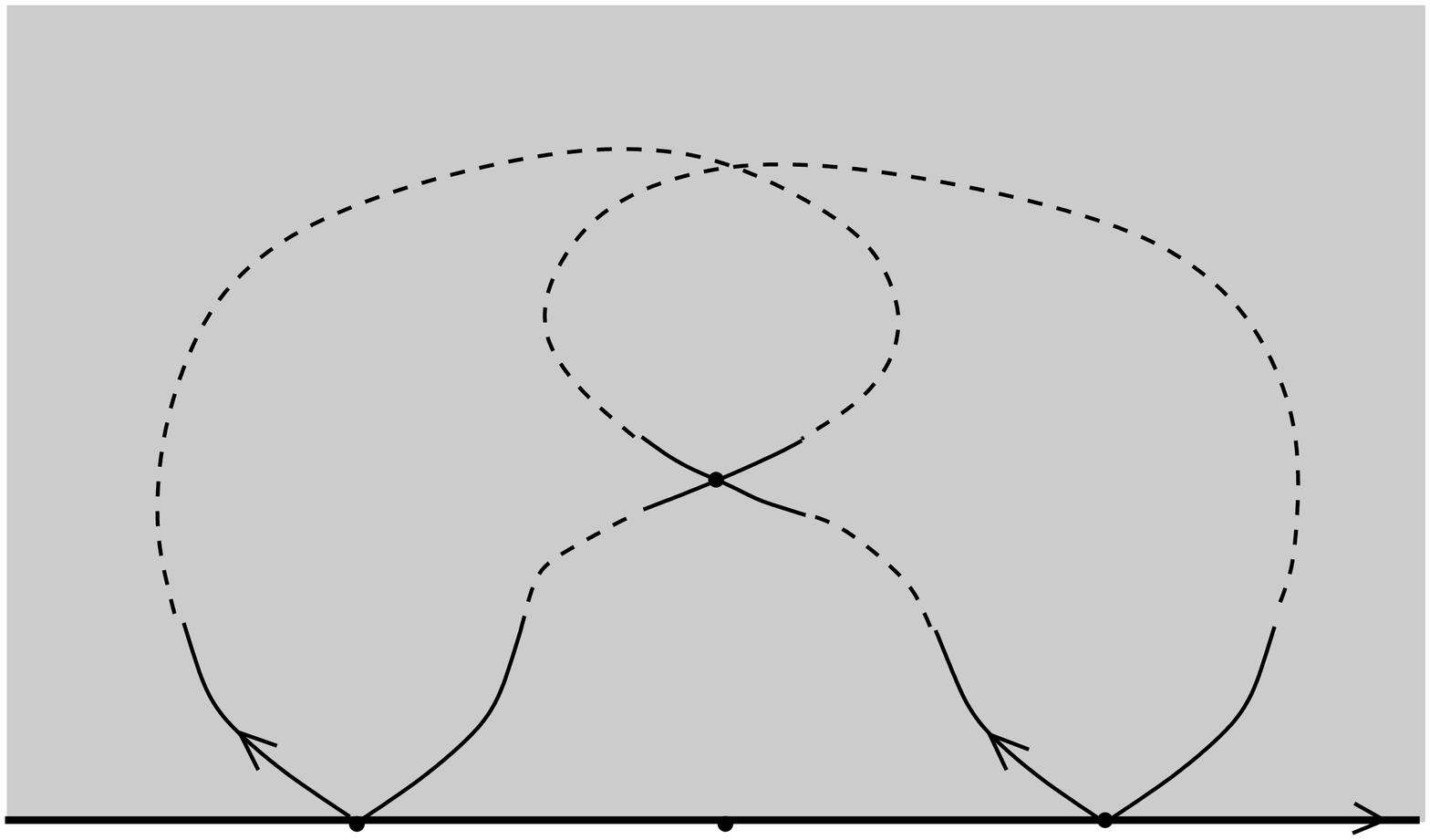}
$$
Then
\begin{equation} \label{eq:def-eta}
\gl(a,b) := \sum_{p \in \alpha \cap \beta} \varepsilon_p(\alpha,\beta)\, \overline{\nu}_{\star l} \alpha_{l p} \overline{\beta}_{pr} \overline{\nu}_{r \star}
\end{equation}

where the sign $\varepsilon_p(\alpha,\beta)=\pm 1$ is equal to $+1$ if, and only if, a unit tangent vector of $\alpha$ followed by a unit tangent vector of $\beta$ gives a positively-oriented frame of the oriented surface $F_g$. The pairing $\gl$ is   implicit in  Papakyriakopoulos' work \cite{Pa75}; see also \cite{Pe06}.

The pairing $\lambda$ has the following properties. First of all, $\gl(\ZZ1,\ZZ[\pi]) = \gl(\ZZ[\pi],\ZZ1)=0$: hence $\gl$ is determined by its restriction to the augmentation ideal $I(\ZZ[\pi])$ of the group ring $\ZZ[\pi]$. 
It is easily checked that  the restriction $\lambda: I(\ZZ[\pi]) \times I(\ZZ[\pi]) \to \ZZ[\pi]$  satisfies
\begin{align}
\label{eq:left} \gl( ax + a'x',y ) &= a \gl( x,y) +  a' \gl( x' , y), \\	
\label{eq:right}	\gl(y, ax + a'x' ) &=  \gl( y,x ) \overline{a} + \gl( y,x' ) \overline{a'}, \\
\label{eq:skew}	\gl(x, y )  & = - \overline{\gl(y, x )} + x \, \overline{y}, \\
\label{eq:boundary_curve}  \gl(x, \overline{\nu} -1) & = -x 
\end{align} 
for any $a,a' \in \ZZ[\pi]$ and $x,x',y \in I(\ZZ[\pi])$.  
The first three properties are reformulations of  \cite[(3),(4),(5)]{Tu78}, and the last one is observed in \cite[Theorem I.(i)]{Tu78}.
Note that, since $\pi$ is a free group of rank $2g$, $I(\ZZ[\pi])$ is a free  left $\ZZ[\pi]$-module of rank $2g$.  
 Furthermore, $I(\ZZ[\pi])$ can be identified with $H_1(F_g,\star \, ; \, \ZZ[\pi])$ \bgm  since it corresponds  to  the image of the  connecting homomorphism  
$$
\partial_*~ \colon H_1(F_g,\star \, ; \, \ZZ[\pi]) \longrightarrow H_0(\star; \ZZ[\pi]),
$$ 
in the long exact sequence of the pair $(F_g,\star)$, through the  isomorphism $H_0(\star; \ZZ[\pi]) \simeq \ZZ[\pi]$.\egm


\subsection{The twisted intersection form $\SSempty$}

Assume now the following:
\begin{equation} \label{eq:R_and_G}
	\left\{  \begin{array}{l} 
		\hbox{$R$ is a commutative ring without  zero-divisors  such that $2\neq 0 \in R$;} \\
		\hbox{$G \subset R^\times$ is a multiplicative subgroup of the group of units of $R$;}\\
		\hbox{$R$  has an involutive ring endomorphism $r \mapsto \overline{r}$ satisfying $\overline{x}=x^{-1}$ for all $x\in G$.}
	\end{array}\right.
\end{equation}
Any group homomorphism $\gvf~ \colon H_1(F_g) \to G$ induces \bgm a group homomorphism $ \pi \to G$, 
which extends to a ring homomorphism  $\gvf~ \colon \ZZ[\pi] \to R$ and \egm  gives $R$ the structure of a right $\ZZ[\pi]$-module. 
Thus we can consider the twisted homology group  $H^{\gvf}_1 (F_g,\star)$ which, as an $R$-module, can be identified~with
$$
R \otimes_{\ZZ[\pi]} H_1(F_g,\star \, ; \, \ZZ[\pi]) \simeq R \otimes_{\ZZ[\pi]} I(\ZZ[\pi]).
$$ 
Using this identification, we define a pairing  $\formempty~ \colon  H^{\gvf}_1 (F_g,\star) \times H^{\gvf}_1 (F_g,\star) \to R $ by setting
\begin{equation}   \label{eq:<-,->}
\forall r,r' \in R, \ \forall x,x'\in I(\ZZ[\pi]), \quad
\form{r \otimes x}{r' \otimes x'} := r \overline{r'}\, \varphi\big(\gl( x, x' )\big).
\end{equation}
By \eqref{eq:left} and \eqref{eq:right},  this pairing is well-defined and sesquilinear; but it is not quite skew-symmetric since \eqref{eq:skew} implies that $ \form{x}{y} = - \overline{\form{y}{x}} + \dep_{*}(x)\, \overline{\dep_{*}(y)}$ for any $x,y \in H^{\gvf}_1 (F_g,\star)$, where $\dep_{*}~\colon H^{\gvf}_1 (F_g,\star) \to R$ is the connecting homomorphism in the long exact sequence of the pair $(F_g,\star)$. 
Therefore, we will prefer to  $\formempty $ the skew-Hermitian form 
$$ 
\SSempty~ \colon H^{\gvf}_1 (F_g,\star) \times H^{\gvf}_1 (F_g,\star) \longrightarrow R 
$$
defined by 
\begin{equation} \label{eq:form_s}
\SSform{x}{y} := 2 \form{x}{y} - \dep_{*}(x) \, \overline{\dep_{*}(y)}
\end{equation}
and which will be referred to  as the {\itshape $\gvf$-twisted intersection form} of $F_g$. The following lemma, which gives a matrix presentation of $\SSempty$, is obtained by straightforward computations.

\begin{figure}[h]
	\centering
	\includegraphics[scale=0.45]{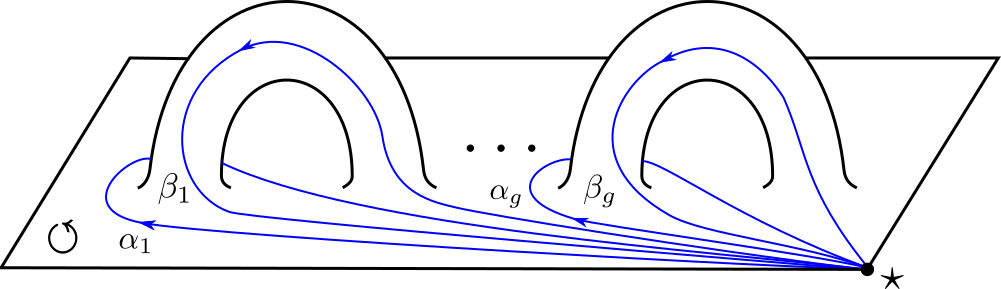}
	\caption{The system of meridians and parallels $(\alpha,\beta)$ on the surface $F_g$.}
	\label{fig:meridians_parallels}
\end{figure}

\begin{lemma} \label{lem:matrix_S}
Let $(\alpha,\beta):= (\alpha_1,\dots,\alpha_g,\beta_1,\dots, \beta_g)$ be a \lq\lq system of meridians and parallels\rq\rq{} in the oriented surface $F_g$ as shown in Figure \ref{fig:meridians_parallels}, and let $(a^\gvf,b^\gvf):= (a_1^\gvf, \dots, a_g^\gvf, b_1^\gvf,\dots, b_g^\gvf)$ be the basis of the free $R$-module $H_1^\gvf(F_g,\star)$ defined by \bgm the homology classes \egm
\begin{equation} \label{eq:a_b}
\forall i=1,\dots, g, \quad a_i^{\gvf} := \big[1 \otimes \widehat \alpha_i  \big], \ b_i^{\gvf} := \big[1 \otimes \widehat \beta_i  \big].
\end{equation}
The matrix of $\SSempty$ in the basis $(a^\gvf,b^\gvf)$ is the skew-Hermitian matrix of size $2g$
\begin{equation}  \label{eq:matrix}
	S^{\gvf} := \left(
	\begin{array}{c | c}
		S_{aa}^{\gvf} & S_{ab}^{\gvf} \\ [1pt] \hline \\ [-10pt]
		S_{ba}^{\gvf} & {S_{bb}^{\gvf}}		
	\end{array}
	\right)
\end{equation}
where 
$$
\hbox{\scriptsize	$S_{aa}^{\gvf} := \left(
	\begin{array}{ccccc}
	\overline{\gvf(\ga_1)}-\gvf(\ga_1) & -P^{\gvf}(\ga_1,\ga_2) & \cdots  & \cdots & -P^{\gvf}(\ga_1,\ga_g)\\[1ex]
	P^{\gvf}(\ga_2,\ga_1) &  \overline{\gvf(\ga_2)}-\gvf(\ga_2) & \cdots  & \cdots & -P^{\gvf}(\ga_2,\ga_g)\\[1ex]
	\vdots                &  \vdots                             & \ddots  &        & \vdots\\[1ex]
	\vdots                &  \vdots                &            & \overline{\gvf(\ga_{g-1})}-\gvf(\ga_{g-1}) & -P^{\gvf}(\ga_{g-1},\ga_g)\\[1ex]
	P^{\gvf}(\ga_g,\ga_1) &  P^{\gvf}(\ga_g,\ga_2) & \cdots     & P^{\gvf}(\ga_g,\ga_{g-1}) & \overline{\gvf(\ga_g)}-\gvf(\ga_g)
	\end{array}
	\right)$}
$$  
$$
\hbox{\scriptsize	$S_{ab}^{\gvf} := \left(
	\begin{array}{ccccc}
	Q^{\gvf}(\ga_1,\gb_1) & -P^{\gvf}(\ga_1,\gb_2) &  \cdots & \cdots                        & -P^{\gvf}(\ga_1,\gb_g)\\[1ex]
	P^{\gvf}(\ga_2,\gb_1) &  Q^{\gvf}(\ga_2,\gb_2) &  \cdots & \cdots                        & -P^{\gvf}(\ga_2,\gb_g)\\[1ex]
	\vdots                &  \vdots                &  \ddots &                               & \vdots\\[1ex]
	\vdots			      &  \vdots                &         & Q^{\gvf}(\ga_{g-1},\gb_{g-1}) & -P^{\gvf}(\ga_{g-1},\gb_g)\\[1ex]
	P^{\gvf}(\ga_g,\gb_1) &  P^{\gvf}(\ga_g,\gb_2) &  \cdots & P^{\gvf}(\ga_g,\gb_{g-1})     & Q^{\gvf}(\ga_g,\gb_g)
	\end{array}
	\right)  \bgm =  - \big(\overline{S_{ba}^{\gvf}}\big)^t  \egm $}   
$$
$$
\hbox{\scriptsize	$S_{bb}^{\gvf} := \left(
	\begin{array}{ccccc}
	\gvf(\gb_1)-\overline{\gvf(\gb_1)} & -P^{\gvf}(\gb_1,\gb_2) & \cdots  & \cdots & -P^{\gvf}(\gb_1,\gb_g)\\[1ex]
	P^{\gvf}(\gb_2,\gb_1) &  \gvf(\gb_2)-\overline{\gvf(\gb_2)} & \cdots  & \cdots & -P^{\gvf}(\gb_2,\gb_g)\\[1ex]
	\vdots                &  \vdots                             & \ddots  &        & \vdots\\[1ex]
	\vdots                &  \vdots                &            & \gvf(\gb_{g-1})-\overline{\gvf(\gb_{g-1})} & -P^{\gvf}(\gb_{g-1},\gb_g)\\[1ex]
	P^{\gvf}(\gb_g,\gb_1) &  P^{\gvf}(\gb_g,\gb_2) & \cdots     & P^{\gvf}(\gb_g,\gb_{g-1}) & \gvf(\gb_g)-\overline{\gvf(\gb_g)}
	\end{array}
	\right)$}
$$
with\,  $P^{\gvf}(x,y) := (1 - \gvf(x)) (1 - \overline{\gvf(y)})$,\, $Q^\gvf(x,y) := (\gvf(x)+1)(\overline{\gvf(y)} + 1) - 2$ \, for any $x,y\in H_1(F_g)$.
\end{lemma}

It can be verified that the determinant of the matrix $S^{\gvf}$ given by Lemma \ref{lem:matrix_S} is $4^g$: therefore the form $\SSempty$  is non-degenerate. Moreover,  \eqref{eq:boundary_curve} shows that 
\begin{equation} \label{eq:boundary_curve_bis} 
\forall x \in H^{\gvf}_1 (F_g,\star), \quad \SSform{x}{\nu} = 2\, \partial_*(x).
\end{equation}
 (Here, by a slight abuse of notation, we  simply denote by $\nu\in H^{\gvf}_1 (F_g,\star)$ the homology class $[1 \otimes \widehat{\nu}]$, 
where $\widehat{\nu} \subset \widehat{F}_g$ is the lift of the oriented curve $\nu$ to the maximal abelian cover that starts at the preferred lift $\widehat{\star}$ of the base-point $\star$.)
\bgm 
Let $Q(R)$ be the field of fractions of $R$ and set
$$
\nu/2 := (1/2) \otimes \nu\  \in Q(R) \otimes_R H^{\gvf}_1 (F_g,\star).
$$
It  follows from \eqref{eq:boundary_curve_bis} that  the triple \egm
$\big(H^{\gvf}_1 (F_g,\star), \SSform{\cdot}{\cdot},\nu/2\big) $
is a pointed skew-Hermitian $R$-module in the sense of Section \ref{pLagr}.


\subsection{The equivariant intersection form $S$}

We now recall  a few facts about  Reidemeister's equivariant intersection forms \cite{Re39}. Let $N$ be a piecewise-linear  compact connected oriented $n$-manifold, and let  $J,J'$ be two disjoint subsets of $\partial N$. (We possibly have $J=\varnothing$ or $J'=\varnothing$, or even $\partial N=\varnothing$.) We assume that $N$ is endowed with a triangulation $T$, such that $J$ is a subcomplex of $T$ and $J'$ is a subcomplex of the dual cellular decomposition $T^*$. 

Fix a ring $R$ and a multiplicative subgroup $G \subset R^\times$ as in \eqref{eq:R_and_G}, and let   $\gvf~ \colon H_1(N) \rightarrow G$  be a group homomorphism. The {\itshape  equivariant intersection form of $N$} (with coefficients in $R$ twisted by $\gvf$, relative to $J \sqcup J'$) is the sesquilinear map 
$$
S_N:=S_{N,\varphi,J \sqcup J'}~ \colon H^{\gvf}_q(N,J) \times H^{\gvf}_{n - q}(N, J') \longrightarrow R 
$$ 
defined for any $q\in \{0,\dots,n\}$ by 
$$ S_N\Big(\Big[\sum_{i} r_i \otimes x_i\Big], \Big[\sum_{i'}\, r'_{i'}\, \otimes x'_{i'}\Big]\Big) := 
\sum_{i,i'}\sum\limits_{h \in H_1(N)}  (hx_i \bullet  x'_{i'}) \,  \varphi(h^{-1})\, r_i\, \overline{r'_{i'}} $$
where $\sum_{i} r_i \otimes x_i \in C^{\gvf}_q(T,J) $ and $\sum_{i'}r'_{i'} \otimes x'_{i'} \in C^{\gvf}_{n-q}(T^*,J')$ are arbitrary cellular cycles. Here $r_i,r_{i'}'$ are elements of $R$, $x_i \subset \widehat{N}$ is a lift of an oriented $q$-simplex of $T$ not included in $J$,  
\bgm $h x_i$ is the image of $x_i$ under the deck transformation of $\widehat{N}$ corresponding to $h\in H_1(N)$, \egm
$x'_{i'} \subset \widehat{N}$ is a lift of an oriented $(n-q)$-cell of $T^*$ not included in $J'$, and  $(hx_i \bullet  x'_{i'})\in \{-1,0,+1\}$ denotes the  intersection number. We recall Blanchfield's duality theorem in the following form, which is adapted to our setting and can be easily deduced from  \cite[Theorem 2.6]{Bl57}.

\begin{theorem}[Blanchfield] \label{th:Blanchfield}
The left and right annihilators of $S_N$ are the torsion submodules of $H^{\gvf}_q(N,J)$ and $H^{\gvf}_{n - q}(N, J')$, respectively.
\end{theorem}

Assume now that $N:=F_g$. Let $l, r\in \partial F_g$ be some points such that  $l < \star < r$ if we follow $\partial F_g$ in the positive direction. Then the  arc joining $l$ to $\star$ in $\partial F_g \setminus \{r\}$ induces an isomorphism $H^{\gvf}_1(F_g,\star)  {\simeq} H^{\gvf}_1(F_g,l)$ and, similarly,  the  arc joining $\star$ to $r$ in $\partial F_g \setminus \{l\}$ induces  an  isomorphism $H^{\gvf}_1(F_g,\star) {\simeq} H^{\gvf}_1(F_g,r)$. It is easily deduced from the definitions  that the diagram 
\begin{equation} \label{com:forms}
\xymatrix @!0 @R=1.2cm @C=4cm {
H^{\gvf}_1(F_g,l ) \times H^{\gvf}_{1}(F_g, r) \ar[r]^-{S}  & R \ar@{=}[d] \\
H^{\gvf}_1(F_g,\star) \times H^{\gvf}_{1}(F_g, \star) \ar[r]_-{\formempty}   \ar[u]_-\simeq & R
} 
\end{equation}
is commutative, where we denote $S:= S_{F_g,\varphi,\{l\}\sqcup \{r\}}$  and    $\formempty$ is the pairing  defined by \eqref{eq:<-,->}.
Since the form $S$ is non-degenerate by Theorem \ref{th:Blanchfield}, \bgm we \egm recover from  \eqref{eq:form_s} and \eqref{eq:boundary_curve_bis} the fact that $\SSempty$ is non-degenerate.


\section{The Magnus functor} \label{FunctorFG}

In this section, we construct the Magnus functor and prove some of its properties. We fix a ring~$R$ and a multiplicative subgroup  $G \subset R^\times$  as in \eqref{eq:R_and_G}.


\subsection{The functor $\Mag$} \label{MainTheorem}

Here is the main result of this section.

\begin{theorem} \label{th:Magnus}
There is a functor $\Mag := \Mag_{R,G}~ \colon \Cob{}_G \rightarrow \pLagr_R$ defined by
$$
\Mag (g, \gvf) :=  \big(H^{\gvf}_1 (F_g,\star), \SSempty, \nu/2 \big)
$$
for any object $(g, \gvf)$ of $\Cob_G$, and by 
$$ 
\Mag (M, \gvf)  := \cl \big( \ker \big((-m_-)\oplus m_+~ \colon H^{\gvf_-}_1 ( F_{g_{-}}, \star ) \oplus H^{\gvf_+}_1 ( F_{g_{+}}, \star ) \rightarrow  H^{\gvf}_1( M, \, \star )\big) \big)
$$
for any morphism $(M, \gvf)~ \colon ( g_{-}, \gvf_{-} ) \to ( g_{+}, \gvf_{+} )$ in $\Cob_G$.
\end{theorem}

This is an analogue of \cite[Theorem 3.5]{CiTu05I} and \cite[Theorem 6.1]{CiTu05I} for cobordisms. The  proof is done in the next subsection.


\subsection{Proof of Theorem \ref{th:Magnus}} \label{ProofMagnus} 
  
We will use the same strategy of proof as in \cite[Section 4]{CiTu05I}. It suffices to prove the following two lemmas.

\begin{lemma} \label{LagrG}
	Let $(M, \gvf) \in \Cob{}_G ( ( g_{-}, \gvf_{-} ), ( g_{+}, \gvf_{+} ))$. Then the closure of 
	$$  \ker \big(   (-m_{-}) \oplus m_{+}~ \colon H^{\gvf_{-}}_1 ( F_{g_{-}}, \star ) \oplus H^{\gvf_{+}}_1 ( F_{g_{+}}, \star ) \longrightarrow 
	H^{\gvf}_1( M, \, \star ) \big) $$
	is a Lagrangian submodule of $H^{\gvf_{-}}_1 ( F_{g_{-}}, \star ) \oplus H^{\gvf_{+}}_1 ( F_{g_{+}}, \star )$. Furthermore, it contains $(\nu_-,\nu_+)$ where $\nu_\pm:= [1 \otimes \widehat{\nu}] $ denotes the ``boundary'' element of  $H^{\gvf_{\pm}}_1 ( F_{g_{\pm}}, \star )$.
\end{lemma}

\begin{lemma} \label{FunctG}
	Let $(M, \gvf) \in \Cob{}_G ( ( g_{-}, \gvf_{-} ), ( g_{+}, \gvf_{+} ))$ and $(N, \psi) \in \Cob{}_G ( ( h_{-}, \psi_{-} ), ( h_{+}, \psi_{+} ))$ be such that $(g_{+}, \gvf_{+}) = (h_{-}, \psi_{-})$. Then $\Mag ( N \circ M, \psi + \gvf ) = \Mag (N, \psi) \circ \Mag (M, \gvf)$.
\end{lemma}

Let $(M, \gvf) \in \Cob{}_G ( ( g_{-}, \gvf_{-} ), ( g_{+}, \gvf_{+} ))$. Observe that the ``vertical'' boundary $m(S^1 \times [-1,1])$ of the cobordism  $M$ can be collapsed onto the circle $m(S^1 \times \lbrace 0 \rbrace )$ without changing the homeomorphism type of $M$. Thus we can assume  that $m_+(\partial F_{g_{+}}) = m_-(\partial F_{g_{-}}) \subset \partial M$, so that the \lq\lq bottom boundary\rq\rq{} $\partial_- M=m_-(F_{g_-})$ and the \lq\lq top boundary\rq\rq{} $\partial_+M = m_+(F_{g_+})$ share the same base point $\star \in \partial M$. Let 
\begin{equation} \label{eq:i}
\mu~ \colon H^{\gvf_{-}}_1 ( F_{g_{-}}, \star ) \oplus H^{\gvf_{+}}_1 ( F_{g_{+}}, \star ) \longrightarrow H^{\gvf}_1 ( \dep M, \star )
\end{equation}
be the direct sum of the \emph{opposite} of  the homomorphism induced by $m_-~ \colon F_{g_-} \to \partial M$ with the homomorphism induced by $m_+~ \colon F_{g_+} \to \partial M$. The long exact sequence for the pair $(\partial M,\star)$ shows that $H_1^\gvf(\partial M)$ 
can be regarded as a submodule of $H_1^\gvf(\partial M,\star)$.
The following lemma is needed for the proof of Lemma \ref{LagrG}.

\begin{lemma} \label{Compatible}
	If $a, b \in H^{\gvf_{-}}_1 ( F_{g_{-}}, \star ) \oplus H^{\gvf_{+}}_1 ( F_{g_{+}}, \star )$ are such that $\mu(a), \mu(b) \in H^{\gvf}_1 ( \dep M )$, then 
	$$ \SSform{a}{b} = 2\, S_{\dep M} \big( \mu(a), \mu(b) \big) $$
	where  $\SSempty$ denotes the skew-Hermitian form $(- \SSempty) \oplus \SSempty$ on $H^{\gvf_{-}}_1 ( F_{g_{-}}, \star ) \oplus H^{\gvf_{+}}_1 ( F_{g_{+}}, \star )$,
	and  $S_{\partial M}~ \colon H^{\gvf}_1( \dep M ) \times H^{\gvf}_1( \dep M ) \rightarrow R$ is the equivariant intersection form of $\partial M$.
\end{lemma}

\begin{proof}  

We first describe how the maximal abelian cover $\widehat{\dep M}$ of $\dep M$ can be constructed from the maximal abelian covers  $\widehat{F}_{g_-}$  and  $\widehat{F}_{g_+}$  of $F_{g_-}$ and $F_{g_+}$, respectively. Consider a copy $y \! \cdot \widehat{F}_{g_-}$ of  $\widehat{F}_{g_-}$ for every $y \in H_1(F_{g_+})$ and, similarly, consider a copy $x \!\cdot\! \widehat{F}_{g_+}$ of  $\widehat{F}_{g_+}$ for every $x \in H_1(F_{g_-})$. In the sequel, the abelian group $H_1(F_{g_{\pm}})$ is denoted multiplicatively. The copy $1\!\cdot \widehat{F}_{g_\pm}$ indexed by the identity element $1\in H_1(F_{g_\mp})$ is simply denoted by $\widehat{F}_{g_\pm}$, and its preferred based point is simply denoted by $\widehat{\star}$. Let $\widehat{\nu}_\pm \subset \widehat{F}_{g_\pm}$ be the unique lift of the circle $\nu_\pm:= \partial F_{g_\pm}$ that contains the base-point $\widehat{\star}$: the preimage of $\nu_\pm$ in $\widehat{F}_{g_\pm}$ consists of all the circles $z\cdot \widehat{\nu}_\pm$ obtained from $\widehat{\nu}_\pm$  by the deck transformations $z\in H_1(F_{g_\pm})$. Then $\widehat{\dep M}$ is the connected surface obtained from the disjoint union
$$ \Big(\bigsqcup_{y \in H_1(F_{g_+})} y \!\cdot\!  \widehat{F}_{g_-} \Big) \ \sqcup \ \Big(\bigsqcup_{x \in H_1(F_{g_-})} x \!\cdot\! \widehat{F}_{g_+} \Big) $$
by connecting with a tube the circle $y\!\cdot\! ( x \!\cdot\! \widehat{\nu}_-)$ in the copy $ y \!\cdot\!  \widehat{F}_{g_-}$ to the circle $x\! \cdot\! ( y \!\cdot\! \widehat{\nu}_+)$ in the copy $x \!\cdot\!  \widehat{F}_{g_+}$, for any $x \in H_1(F_{g_-})$ and $y \in H_1(F_{g_+})$. 
The action of $H_1(\dep M) \simeq H_1(F_{g_-}) \oplus H_1(F_{g_+})$ by deck transformations on $\widehat{\dep M}$ is the obvious action that is suggested by our notations: see Figure~\ref{fig:cover}. 

 	\begin{figure}[h]
		\centering
		\includegraphics[scale=0.50]{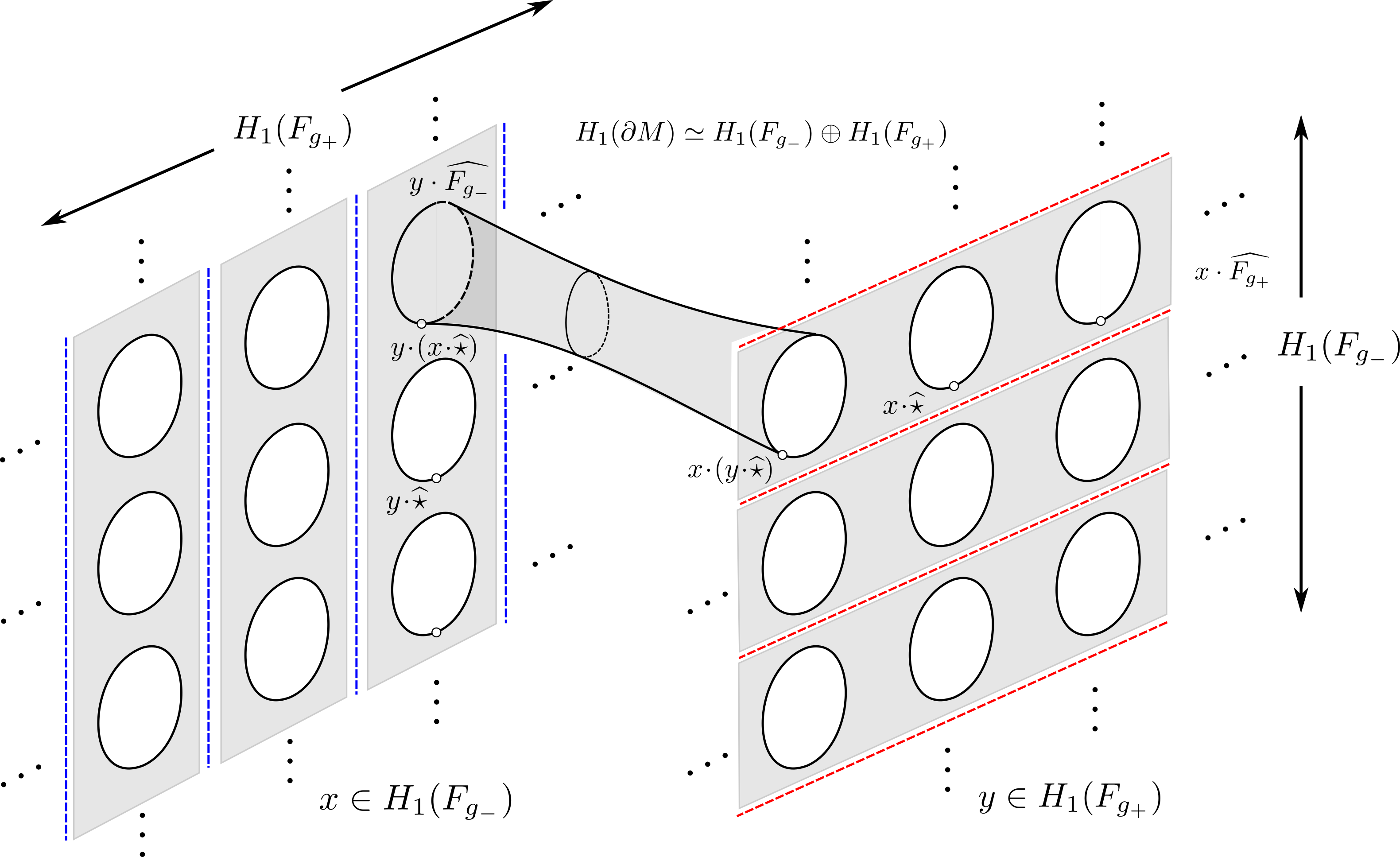}
		\caption{The maximal abelian cover $\widehat{\dep M}$ of $\dep M$.}
		\label{fig:cover}
	\end{figure}

Let $l, r\in \partial F_{g_\pm}$ be some points such that  $l < \star < r$ if we follow $\partial F_{g_\pm}$ in the positive direction. Let $\overline{\nu}_{\star l}$ be the oriented arc joining $\star$ to $l$ in $\partial F_{g_\pm} \setminus \{r\}$, and let ${\nu}_{\star r}$ be the oriented arc joining $\star$ to $r$ in $\partial F_{g_\pm} \setminus \{l\}$: the endpoints of the lifts of $\overline{\nu}_{\star l}$ and  ${\nu}_{\star r}$ 
to $\widehat{F}_{g_\pm}$ that start at $\widehat{\star}$ are denoted by $\widehat{l}$ and $\widehat{r}$ respectively. Consider now some elements $r_{-,k},r_{+,k}\in R$ (where $k$ ranges over a finite set) and some loops $\alpha_{-,k}$ in $F_{g_-}$ and $\alpha_{+,k}$ in $F_{g_+}$ based at $l$, such that 
\begin{equation} \label{eq:cond_a}
\sum_{k} r_{-,k} \big(\gvf_-([\alpha_{-,k}]) -1\big)= \sum_k r_{+,k} \big(\gvf_+([\alpha_{+,k}]) -1\big) \in R.
\end{equation}
Similarly, consider some elements $s_{-,n},s_{+,n} \in R$ (where $n$ ranges over a finite set) 
and some loops $\beta_{-,n}$ in $F_{g_-}$ and $\beta_{+,n}$ in $F_{g_+}$ based at $r$, such that 
\begin{equation} \label{eq:cond_b}
\sum_{n} s_{-,n} \big(\gvf_-([\beta_{-,n}]) -1\big)= \sum_{n} s_{+,n} \big(\gvf_+([\beta_{+,n}]) -1\big) \in R.
\end{equation}
Set
$$
a' =(a'_-,a'_+):=  \Big(\Big[\sum_k r_{-,k} \otimes \widehat{\alpha}_{-,k} \Big], \Big[\sum_k r_{+,k} \otimes \widehat{\alpha}_{+,k} \Big] \Big) 
\in  H^{\gvf_{-}}_1 ( F_{g_{-}}, l ) \oplus H^{\gvf_{+}}_1 ( F_{g_{+}}, l )
$$
where $\widehat{\alpha}_{\pm,k} \subset \widehat{F}_{g_\pm}$ denotes the unique lift of $\alpha_{\pm,k}$ starting at $\widehat{l}$,
and let $a=(a_-,a_+)$ be the image of $a'=(a'_-,a'_+)$ by the isomorphism $H^{\gvf_{-}}_1 ( F_{g_{-}}, l ) \oplus H^{\gvf_{+}}_1 ( F_{g_{+}}, l ) \stackrel{\simeq}{\longrightarrow} H^{\gvf_{-}}_1 ( F_{g_{-}}, \star ) \oplus H^{\gvf_{+}}_1 ( F_{g_{+}}, \star )$ that is induced by the arcs $\overline{\nu}_{\star l} \subset F_{g_-}$ and $\overline{\nu}_{\star l} \subset F_{g_+}$. Similarly, set 
$$ 
b' =(b'_-,b'_+) := \Big(\Big[\sum_{n} s_{-,n} \otimes \widehat{\beta}_{-,n} \Big], \Big[\sum_{n}  s_{+,n} \otimes \widehat{\beta}_{+,n} \Big] \Big) 
\in  H^{\gvf_{-}}_1 ( F_{g_{-}}, r ) \oplus H^{\gvf_{+}}_1 ( F_{g_{+}}, r )
$$
where $\widehat{\beta}_{\pm,n} \subset \widehat{F}_{g_\pm}$ denotes the unique lift of $\beta_{\pm,n}$ starting at $\widehat{r}$,
and let $b=(b_-,b_+)$ be the image of $b'=(b'_-,b'_+)$ by the isomorphism 
$H^{\gvf_{-}}_1 ( F_{g_{-}}, r ) \oplus H^{\gvf_{+}}_1 ( F_{g_{+}}, r) \stackrel{\simeq}{\longrightarrow} H^{\gvf_{-}}_1 ( F_{g_{-}}, \star ) \oplus H^{\gvf_{+}}_1 ( F_{g_{+}}, \star )$
that is  induced by the arcs ${\nu}_{\star r} \subset F_{g_-}$ and ${\nu}_{\star r} \subset F_{g_+}$. The assumption  \eqref{eq:cond_a} is equivalent to the condition
$
\dep_*(a'_-) = \dep_*(a'_+)
$
where  $\dep_*~ \colon H^{\gvf_{\pm}}_1 ( F_{g_{\pm}}, l ) \to H^{\gvf_\pm}_0 ( l ) \simeq R $ is the connecting homomorphism in the long exact sequence of the pair $(F_{g_\pm},l)$; the latter is also equivalent to the condition $\mu(a) \in H^{\gvf}_1 ( \dep M )$. Similarly, the assumption  \eqref{eq:cond_b}  is equivalent to the condition $\mu(b) \in H^{\gvf}_1 ( \dep M )$. Thus we have described the general form of some elements $a, b  \in H^{\gvf_{-}}_1 ( F_{g_{-}}, \star ) \oplus H^{\gvf_{+}}_1 ( F_{g_{+}}, \star )$ satisfying $\mu(a), \mu(b) \in H^{\gvf}_1 ( \dep M ) \subset H^{\gvf}_1 ( \dep M,\star ) $.

Note that, with the above description of $a$ and $b$,  we have 
$$
\mu(a) = \Big[\sum_k r_{+,k} \otimes \widehat{\alpha}_{+,k} - \sum_k r_{-,k} \otimes \widehat{\alpha}_{-,k}  \Big]
\quad \hbox{and} \quad
\mu(b) = \Big[ \sum_{n} s_{+,n} \otimes \widehat{\beta}_{+,n} - \sum_{n}  s_{-,n} \otimes \widehat{\beta}_{-,n}  \Big]
$$
where $\widehat{\alpha}_{\pm,k}$ and  $\widehat{\beta}_{\pm, n}$ are now regarded as $1$-chains in $\widehat{\partial M}$ thanks to the inclusion $\widehat{F}_{g_\pm} \subset \widehat{\partial M}$. (It follows from \eqref{eq:cond_a} and \eqref{eq:cond_b} that the above  $1$-chains are indeed $1$-cycles in $C^\varphi(\partial M) = {R \otimes_{\ZZ[H_1(\partial M)]} C(\widehat{\partial M})}$.) We deduce that 
\begin{eqnarray*}
S_{\dep M}( \mu(a), \mu(b) ) 
&=&  S_{(F_{g_+})}\Big( \Big[ \sum_k r_{+,k}  \otimes \widehat{\alpha}_{+,k} \Big],   \Big[ \sum_{n}  s_{+,n}  \otimes \widehat{\beta}_{+,n}  \Big]\Big) \\
&& + S_{(-F_{g_-})} \Big( \Big[ \sum_k r_{-,k}  \otimes \widehat{\alpha}_{-,k}  \Big],   \Big[ \sum_{n}  s_{-,n}  \otimes \widehat{\beta}_{-,n}  \Big]\Big) \\
 &=&  S_{(F_{g_+})}\big(  a'_+,  b'_+ \big) -  S_{(F_{g_-})}\big(  a'_-,  b'_- \big) 
 \ \by{com:forms} \ \langle  a_+,b_+ \rangle - \langle a_-,b_- \rangle
\end{eqnarray*}
where $S_{(F_{g_\pm})}:=S_{F_{g_\pm},\varphi_\pm,\{l\}\sqcup \{r\}}$ is the equivariant intersection form of $F_{g_\pm}$. Furthermore, we have $\partial_*(a_+) =  \partial_*(a_-)$ and  $\partial_*(b_+) = \partial_*(b_-)$ by \eqref{eq:cond_a} and \eqref{eq:cond_b}, respectively. We conclude that 
\begin{eqnarray*}
2 S_{\dep M}( \mu(a), \mu(b) ) &=& 2   \langle  a_+,b_+ \rangle - 2 \langle a_-,b_- \rangle \\
& \by{eq:form_s} & \langle  a_+,b_+ \rangle_s  +  \partial_*(a_+)  \overline{\partial_*(b_+)} - \langle a_-,b_- \rangle_s  -  \partial_*(a_-) \overline{\partial_*(b_-)} \\
&=& \langle  a_+,b_+ \rangle_s - \langle a_-,b_- \rangle_s \ = \  \langle  a ,b \rangle_s .
\end{eqnarray*}

\up
\end{proof}

The rest of this subsection is devoted to the proofs of Lemma \ref{LagrG} and Lemma \ref{FunctG}.

\begin{proof}[Proof of Lemma \ref{LagrG}]
The arguments below follow the same lines as  \cite[Lemma 3.3]{CiTu05I}, but there are also some technical complications due to  the peculiarities of the twisted intersection form  $\SSempty$. We start with a preliminary observation. Denote by $j~ \colon \partial M \to M$ and by $j^\star~\colon (\partial M,\star) \to (M,\star)$ the inclusions, and consider the following commutative diagram:
	\begin{center}
		\begin{tikzpicture}[>=angle 90,scale=2.2,text height=1.5ex, text depth=0.25ex]
		\node (a0) at (0,1) {\normalcolor $0$};
		\node (a1) [right=of a0] {\normalcolor $H^{\gvf}_1 ( \dep M )$};
		\node (a2) [right=of a1] {\normalcolor $H^{\gvf}_1 ( \dep M, \star )$};
		\node (a3) [right=of a2] {\normalcolor $\partial_*\big(H^{\gvf}_1 ( \dep M, \star )\big)$};
		\node (a4) [right=of a3] {\normalcolor $0$};
		\node (b0) [below=of a0] {\normalcolor $0$};
		\node (b1) [below=of a1] {\normalcolor $H^{\gvf}_1 ( M )$};
		\node (b2) [below=of a2] {\normalcolor $H^{\gvf}_1 ( M , \, \star )$};
		\node (b3) [below=of a3] {\normalcolor $\partial_*\big(H^{\gvf}_1 ( M, \star )\big)$};
		\node (b4) [below=of a4] {\normalcolor $0$};
		
		\normalcolor
		\draw[->]
		(a0) edge (a1)
		(a1) edge node[auto] {} (a2)
		(a2) edge node[auto] {\normalcolor $\partial_*$} (a3)
		(a3) edge node[auto] {} (a4)
		(b0) edge (b1)
		(b1) edge (b2)
		(b2) edge node[auto] {\normalcolor $\partial_*$}  (b3)
		(b3) edge node[auto] {} (b4)
		(a1) edge node[auto] {\normalcolor $j$}  (b1)
		(a2) edge node[auto] {\normalcolor $j^\star$}   (b2)
		(a3) edge node[auto] {\normalcolor } (b3);
		\end{tikzpicture}
	\end{center}
Here the rows are extracted from the long exact sequences of the pairs $( \dep M, \star )$ and $( M, \, \star )$, $\partial_*\big(H^{\gvf}_1 ( \dep M, \star )\big)$ is  the submodule of $H_0^\gvf(\star)\simeq R$ generated by the $\varphi j(h)-1$ for all $h\in H_1(\partial M)$, $\partial_*\big(H^{\gvf}_1 (  M, \star )\big)$ is the submodule generated by the $\varphi (h)-1$ for all $h\in H_1(M)$, and the map $\partial_*\big(H^{\gvf}_1 ( \dep M, \star )\big) \to \partial_*\big(H^{\gvf}_1 (  M, \star )\big)$ is the inclusion. 
It follows from the ``ker-coker'' lemma that  $\ker j =\ker j^\star$: in other words, if we continue to regard $H_1^\gvf(\partial M)$ as a submodule of $H_1^\gvf(\partial M,\star)$, then  our observation is that $K :=  \ker j^\star$  is contained in $H_1^\gvf(\partial M)$. 

The above being  taken into account, we now have to prove that the closure $\cl(L)$ of
$$ 
L:=   \ker \Big(   (-m_{-}) \oplus m_{+}~ \colon H^{\gvf_{-}}_1 ( F_{g_{-}}, \star ) \oplus H^{\gvf_{+}}_1 ( F_{g_{+}}, \star ) \longrightarrow 
	H^{\gvf}_1( M, \, \star ) \Big)
$$	
in $H^{\gvf_{-}}_1 ( F_{g_{-}}, \star ) \oplus H^{\gvf_{+}}_1 ( F_{g_{+}}, \star )$ is a Lagrangian submodule. We first  claim the following:
\begin{itemize}
	\item[(i)] The annihilator $\Ann(K)$ of $K$ with respect to  $S_{\partial M}~ \colon H^{\gvf}_1( \dep M ) \times H^{\gvf}_1( \dep M ) \rightarrow R$
	coincides with the closure $\cl(K)$ of $K$ in $H^{\gvf}_1( \dep M )$.
	
	\item[(ii)] We have $\mu(L)=K$ where $\mu$ is the homomorphism \eqref{eq:i}; in particular,  $\mu(L) \subset H^{\gvf}_1 ( \dep M )$.
\end{itemize}
To prove claim (i), we consider the equivariant intersection form of $M$
$$ S_M := S_{M,\varphi, \varnothing \sqcup \dep M}~ \colon H^{\gvf}_1( M ) \times H^{\gvf}_2( M, \dep M ) \longrightarrow R, $$
which is related to the form $S_{\dep M}$ through the following identity: 
$$ 
\forall x \in H^{\gvf}_1( \dep M ), \ \forall y \in H^{\gvf}_2 ( M, \dep M ), \quad S_{\dep M} ( x, \dep_*(y) ) = \varepsilon\, S_M ( j(x), y );
$$
here $\varepsilon$ is a constant sign which we do not need to specify.
It follows that
\begin{eqnarray*}
	\Ann(K) &= &\lbrace x \in H^{\gvf}_1(\partial M) :  S_{\dep M} (x, \ker j) = 0 \rbrace \\
	&= &\lbrace x \in H^{\gvf}_1(\partial M) :  S_{\dep M} (x, \partial_* H^{\gvf}_2 ( M, \dep M )) = 0 \rbrace \\
	&=&  \lbrace x \in H_1(\partial M) : S_{M} ( j(x), H_2^{\gvf}(M, \partial M) )= 0 \rbrace 
	\ = \ j^{-1} \big(\Tors_{R}\, H^{\gvf}_1 ( M ) \big) \ = \ \cl(K),
\end{eqnarray*}
where the penultimate identity follows from Blanchfield's duality theorem (see Theorem \ref{th:Blanchfield}).

We now prove claim (ii). The inclusion $\mu(L) \subset K$ follows immediately from the facts that  $L=\ker(j^\star \circ \mu)$ and $K=\ker j^\star$.
The converse inclusion follows by the same argument from the surjectivity of the map $\mu$. (This surjectivity is a consequence of the Mayer--Vietoris theorem and the fact that $H_0^\gvf(\nu,\star)=0$ where $\nu$ denotes the circle $m_+(F_{g_{+}}) \cap m_-(F_{g_{-}})$ in $\partial M$.)

To proceed, we observe that the closure $\cl(K)$ of $K$ in $H^{\gvf}_1( \dep M )$ coincides with the closure of $K$ in $H^{\gvf}_1( \dep M ,\star)$, since $H^{\gvf}_1( \dep M )$ is the kernel of $\partial_*~ \colon  H^{\gvf}_1( \dep M ,\star) \to H^{\gvf}_0( \star) \simeq R$ and $R$ has no zero-divisor.
Using  (ii), it follows that  $\cl(L) \subset  \mu^{-1}(\cl(K))$ and, in particular,  $\mu(\cl(L))$ is contained in $H^{\gvf}_1( \dep M )$.
The converse inclusion  $  \mu^{-1}(\cl(K)) \subset \cl(L)$ is also true: for any $a \in \mu^{-1}(\cl(K))$, we have  $r \mu(a) \in K$ for some  $r \in R \setminus \{0\}$; hence $0= j^{\star} (r \mu(a)) = (j^{\star} \circ \mu)(r a) $, which implies that $r a \in L$ so that $a \in \cl(L)$.  Thus, we obtain 
\begin{equation} \label{eq:almost-Lagrangian}
	\cl(L) = \mu^{-1}(\cl(K)) \stackrel{\operatorname{(i)}}{=}  \mu^{-1}( \Ann(K) ) \stackrel{\operatorname{(ii)}}{=}  \mu^{-1}( \Ann( \mu(L) ) ) 
\end{equation}
where $\Ann( \mu(L) )$ denotes the annihilator of $\mu(L)$ with respect to $S_{\partial M}$.

We can now prove that $\cl(L) \subset \Ann(\cl(L))$.  Since $R$ has no zero-divisors, we have $\Ann(\cl(L)) = \Ann(L)$ so that it is enough to prove that $\cl(L) \subset \Ann(L)$. Let $a\in \cl(L)$: using Lemma \ref{Compatible}, we obtain
$$ \forall b\in L, \ \SSform{a}{b} = 2 S_{\dep M} ( \mu(a), \mu(b) ) \by{eq:almost-Lagrangian} 0 $$
which shows that $a\in \Ann(L)$ as desired.

To conclude that $\cl(L)$ is Lagrangian, it remains to show that $\Ann(L) \subset \cl(L)$. Consider any element $a \in {\Ann(L) \cap  \mu^{-1}( H^{\gvf}_1( \dep M ) )}$: using Lemma~\ref{Compatible}, we obtain
$$ 
\forall b\in L, \
 2\, S_{\dep M} ( \mu(a), \mu(b) )  = \SSform{a}{b}  = 0 
 $$
and it follows that $\mu(a) \in \Ann(\mu(L))$; we deduce from \eqref{eq:almost-Lagrangian} that $a\in \cl(L)$. Therefore we are reduced to proving that $\mu(\Ann(L)) \subset  H^{\gvf}_1( \dep M ) $. Since $\ker(\mu) \subset L$, we have $ \Ann(L) \subset \Ann( \ker(\mu) )$ and it is enough to show that
\begin{equation} \label{eq:last_inclusion}
\mu \big(\Ann( \ker(\mu) )\big) \subset  H^{\gvf}_1( \dep M ).
\end{equation}
Let $a=(a_-,a_+) \in H^{\gvf_{-}}_1 ( F_{g_{-}}, \star ) \oplus H^{\gvf_{+}}_1 ( F_{g_{+}}, \star )$ be an arbitrary element of $\Ann( \ker(\mu) )$.
Consider the distinguished element $(\nu_-,\nu_+) \in H^{\gvf_{-}}_1 ( F_{g_{-}}, \star ) \oplus H^{\gvf_{+}}_1 ( F_{g_{+}}, \star )$ defined by
$ \nu_\pm:= [1 \otimes \widehat{\nu}]$. Clearly $\mu(\nu_-,\nu_+) =0$ (so that, in particular, $(\nu_-,\nu_+)$ belongs to $L$). Therefore 
$$ 0 =  \SSform{a}{(\nu_-,\nu_+)}  =  \SSform{a_+}{\nu_+} - \SSform{a_-}{\nu_-}  
\by{eq:boundary_curve_bis}  2\, \partial_*(a_+) - 2\, \partial_*(a_-) $$
which implies that $\mu(a)  \in  H^{\gvf}_1( \dep M )$. This proves \eqref{eq:last_inclusion} and concludes the proof of the lemma.
\end{proof}

\begin{proof}[Proof of Lemma \ref{FunctG}]
The arguments below follow very closely those of  \cite[Lemma 3.4]{CiTu05I}. Consider the composition of two morphisms in the category $\Cob_G$  
$$
( g_{-}, \gvf_{-} ) \stackrel{(M, \gvf) }{\longrightarrow} ( g_{+}, \gvf_{+} ) = ( h_{-}, \psi_{-} )  \stackrel{(N, \psi) }{\longrightarrow}  ( h_{+}, \psi_{+} )
$$
and recall that, after collapsing the vertical boundary of $M$ onto the ``middle'' circle $m(S^1 \times \{0\})$, we can assume that $m_-(F_{g_-})$ and $m_+(F_{g_+})$ share the same base point $\star \in \partial M$. The same consideration holds for $n_-(F_{h_-})$ and $n_+(F_{h_+})$ in $\partial N$. Thus we have only one base point $\star$ in what follows. 
Recall now the homomorphisms
\begin{align*}
	\mathfrak{M}:= (-m_{-}) \oplus m_{+} &: H^{\gvf_{-}}_1 ( F_{g_{-}}, \star ) \oplus H^{\gvf_{+}}_1 ( F_{g_{+}}, \star ) \longrightarrow H^{\gvf}_1( M, \, \star ), \\
	\mathfrak{N}:= (-n_{-}) \oplus n_{+} &: H^{\psi_{-}}_1 ( F_{h_{-}}, \star ) \oplus H^{\psi_{+}}_1 ( F_{h_{+}}, \star ) \longrightarrow H^{\psi}_1( N, \, \star ), \\
	\mathfrak{R} := (-m_{-}) \oplus n_{+} &: H^{\gvf_{-}}_1 ( F_{g_{-}}, \star ) \oplus H^{\psi_{+}}_1 ( F_{h_{+}}, \star ) \longrightarrow H^{\psi+\gvf}_1( N\circ M, \, \star )
\end{align*}
used in the definition of the Lagrangian submodules $\Mag (M, \gvf)$, $\Mag (N, \psi)$, $\Mag(N \circ M, \psi+\gvf)$ respectively. It is sufficient to prove that
\begin{equation} \label{eq:pre}
 \ker(\mathfrak{R}) = \ker(\mathfrak{N}) \, \check \circ \ker(\mathfrak{M}), 
\end{equation}
where we use the notation \eqref{eq:pre-composition}. Indeed this claim implies that
\begin{align*}
	\Mag (N \circ M, \psi + \gvf) &= \cl \big( \ker(\mathfrak{R}) \big) \\
								  &= \cl \big( \ker(\mathfrak{N}) \, \check \circ \, \ker(\mathfrak{M}) \big) \\
								  &= \cl \Big( \cl \big( \ker(\mathfrak{N}) \big) \, \check \circ \cl \big( \ker(\mathfrak{M}) \big) \Big) \\
								  &= \cl \big( \Mag ( N, \psi ) \; \check \circ \; \Mag ( M, \gvf ) \big) = \Mag ( N, \psi ) \circ \Mag ( M, \gvf )
\end{align*}
where the third equality holds by \cite[Lemma 2.6]{CiTu05I}.

The construction of $N \circ M$ by identifying $\dep_{+}M$ and $\dep_{-}N$ using the boundary-parametrizations leads to the following Mayer--Vietoris exact sequence of $R$-modules:
$$
H^{\gvf_{+}}_1 ( F_{g_{+}}, \star ) \overset{\bgm \iota \egm}{\longrightarrow}  H^{\gvf}_1 ( M, \star ) \oplus H^{\psi}_1 ( N, \star ) 
\overset{\bgm \varpi \egm}{\longrightarrow} H^{\psi + \gvf}_1( N \circ M, \star ) \longrightarrow 0
$$
Here $\bgm \iota \egm:=(m_+,-n_-)$ and $\bgm \varpi \egm$ is the sum of the homomorphisms induced by the inclusions of $M$ and $N$ in $N\circ M$. Consider the commutative diagram

\begin{center}
	\begin{tikzpicture}[>=angle 90,scale=2.2,text height=1.5ex, text depth=0.25ex]
	\normalcolor
	\node (a0) at (0,1) {\normalcolor $0$};
	\node (a1) [right=of a0] {\normalcolor $H^{\gvf_{+}}_{g_{+}}$};
	\node (a2) [right=of a1] {\normalcolor $H^{\gvf_{-}}_{g_{-}} \oplus H^{\gvf_{+}}_{g_{+}} \oplus H^{\psi_{+}}_{h_{+}}$};
	\node (a3) [right=of a2] {\normalcolor $H^{\gvf_{-}}_{g_{-}} \oplus H^{\psi_{+}}_{h_{+}}$};
	\node (a4) [right=of a3] {\normalcolor $0$};
	\node (b0) [below=of a0] {\normalcolor $0$};
	\node (b1) [below=of a1] {\normalcolor $\ker(\bgm \varpi \egm)$};
	\node (b2) [below=of a2] {\normalcolor $H^{\gvf}_1 ( M, \star ) \oplus H^{\psi}_1 ( N, \star )$};
	\node (b3) [below=of a3] {\normalcolor $H^{\psi + \gvf}_1 ( N \circ M, \star )$};
	\node (b4) [below=of a4] {\normalcolor $0$};
	
	\draw[->]
	(a0) edge (a1)
	(a1) edge node[auto] {\normalcolor $i$} (a2)
	(a2) edge node[auto] {\normalcolor $p$} (a3)
	(a3) edge (a4)
	(b0) edge (b1)
	(b1) edge (b2)
	(b2) edge node[auto] {\normalcolor $\bgm \varpi \egm$} (b3)
	(b3) edge (b4)
	(a1) edge node[auto] {\normalcolor $\bgm \iota \egm$} (b1)
	(a2) edge node[auto] {\normalcolor $\ggm$} (b2)
	(a3) edge node[auto] {\normalcolor $\mathfrak{R}$} (b3);
	\end{tikzpicture}
\end{center}

where the symbol $H^{\rho}_{n}$ denotes  a twisted homology group $H^{\rho}_1 ( F_{n}, \star )$, the homomorphism $i$ is the natural inclusion, $p$ is the natural projection and the map $\gamma$ is defined by  $\ggm(x, x', x'') := ( \mathfrak{M} (x, x'), \mathfrak{N} (x', x''))$. On the one hand, we have 
\begin{align*}
p \big( \ker(\ggm) \big) = & \big\lbrace (x , x'') \in H^{\gvf_{-}}_{g_{-}} \oplus H^{\psi_{+}}_{h_{+}}
:  \mathfrak{M}(x, x') = 0 \hbox{ and } \mathfrak{N}(x',x'')=0 \; \text{for some} \; x' \in H^{\gvf_{+}}_{g_{+}} \big\rbrace \\
= & \ker( \mathfrak{N} ) \check \circ \ker( \mathfrak{M} ).
\end{align*}
On the other hand, the ``ker-coker'' lemma applied to the above diagram shows that $p(\ker(\ggm)) = \ker( \mathfrak{R} )$ since \bgm $\iota$ \egm  is surjective. This proves the  claim \eqref{eq:pre}.
\end{proof}

 
\subsection{Properties of the functor $\Mag$} \label{Properties}

We show two fundamental properties of the Magnus functor. The first one involves the following relation among cobordisms. Two cobordisms $(M_1, \gvf_1),$ $(M_2,\gvf_2) \in \Cob_{G}((g_{-}, \gvf_{-}), (g_{+}, \gvf_{+}))$ 
are \emph{homology concordant} if there exists a compact connected oriented $4$-manifold $W$ with 
$$ \partial W = M_1 \cup_{m_1 \circ m_2^{-1}} (-M_2) $$
such that the inclusion maps $i_1~ \colon M_1 \hookrightarrow W$ and $i_2~ \colon M_2 \hookrightarrow W$ induce isomorphisms at the level of $H(\cdot ;\ZZ)$ and satisfy $\gvf_1 \circ (i_1)^{-1} = \gvf_2 \circ (i_2)^{-1}~ \colon H_1(W) \to G$. In such a situation, we  write  $(M_1, \gvf_1) \sim_H (M_2, \gvf_2)$. It is easily verified that $\sim_H$ is an equivalence relation on the set $\Cob_{G}((g_{-}, \gvf_{-}), (g_{+}, \gvf_{+}))$ for any objects $(g_{-}, \gvf_{-}), (g_{+}, \gvf_{+})$, and that $\sim_H$ defines a congruence relation on the category $\Cob_{G}$. Besides, the monoidal structure of $\Cob_G$ 
induces a \bgm strict \egm monoidal structure on the  quotient category  $\Cob_{G}/\!\sim_H$.

 In addition to the hypothesis~\eqref{eq:R_and_G} on the ring $R$ and the multiplicative subgroup $G \subset R^\times$, 
consider the following condition:
\begin{equation} \label{eq:R_again}
	 \begin{array}{l} 
		\hbox{\it  $R$ is equipped with  a ring homomorphism 
		$ \varepsilon_R~ \colon R \rightarrow \ZZ$  such that $\varepsilon_R(G) = \{1\}$.}
	\end{array} 
\end{equation}

\begin{proposition}  \label{prop:h_cobordism_rel}
 Under the assumption \eqref{eq:R_again},
the functor $\Mag~ \colon  \Cob_{G} \to \pLagr_R$ descends to the quotient $\Cob_{G}/\!\sim_H$.
\end{proposition}

\begin{proof}
Let $(M_1, \gvf_1),(M_2,\gvf_2) \in\Cob_{G}((g_{-}, \gvf_{-}), (g_{+}, \gvf_{+}))$ be such that $(M_1, \gvf_1) \sim_H (M_2,\gvf_2)$.
Let $j\in\{1,2\}$.
By  Lemma \ref{lem:KLW} below,
the fact that the map $i_j~ \colon M_j \hookrightarrow W$ induces an isomorphism in ordinary homology implies that 
$$
i_j~ \colon Q(R) \otimes_R H_1^{\gvf_j} (M_j,\star)  \simeq  H_1^{\gvf_{j,Q}}(M_j,\star) \longrightarrow 
H_1^{\gvf_Q}(W,\star) \simeq Q(R) \otimes_R H_1^\chi(W,\star)
$$ 
is an isomorphism where $\chi := \gvf_1 i_1^{-1}= \gvf_2 i_2^{-1}$.
Thus we get the following commutative diagram:
$$
\xymatrix @!0 @R=1cm @C=3cm {
&& Q(R) \otimes_R H_1^{\gvf_1}(M_1,\star)  \ar[drr]^{i_1}_\simeq &&\\
Q(R) \otimes_R H_1^{\gvf_\pm}(F_{g_\pm},\star) \ar[rru]^{m_{1,\pm}} \ar[rrd]_{m_{2,\pm}}  &&  && Q(R) \otimes_R H_1^{\chi}(W,\star)\\
& & Q(R) \otimes_R  H_1^{\gvf_2}(M_2,\star) \ar[rru]_{i_2}^\simeq   & &
}
$$
It follows from this diagram that 
\begin{eqnarray*}
&& \ker\big((-m_{1,-}) \oplus m_{1,+}\big) \subset \big((-m_{2,-}) \oplus m_{2,+}\big)^{-1}(\Tors_R H_1^{\gvf_2}(M_2,\star)) \\
\hbox{and}  &&   \ker\big((-m_{2,-}) \oplus m_{2,+}\big) \subset \big((-m_{1,-}) \oplus m_{1,+}\big)^{-1}( \Tors_R H_1^{\gvf_1}(M_1,\star)),
\end{eqnarray*} 
which immediately implies that
\begin{eqnarray*}
&& \ker\big((-m_{1,-}) \oplus m_{1,+}\big) \subset \operatorname{cl}\big(\ker\big((-m_{2,-}) \oplus m_{2,+}\big)\big)  \\
\hbox{and}  &&   \ker\big((-m_{2,-}) \oplus m_{2,+}\big) \subset \operatorname{cl}\big(\ker\big((-m_{1,-}) \oplus m_{1,+}\big)\big).
\end{eqnarray*}
We conclude that  $\Mag(M_1, \gvf_1) = \Mag(M_2, \gvf_2)$.
\end{proof}

\begin{lemma} \label{lem:KLW}
 Assume \eqref{eq:R_again}. 
Let $(X,Y)$ be a pair of CW-complexes such that  the inclusion map $i~ \colon Y \to X$ induces an isomorphism at the level of $H(\cdot \, ;\ZZ)$. For any group homomorphism $\varphi~ \colon H_1(X) \to G$, we have $H^{\varphi_Q}(X,Y)=0$ where $\varphi_Q~ \colon \ZZ[H_1(X)] \to Q(R)$ is the ring homomorphism induced by $\varphi$.
\end{lemma}

\begin{proof}
We closely follow  the arguments of \cite[Proposition 2.1]{KiLiWa01}.
Since the cell chain complex $ C(X,Y) $ with coefficients in $\ZZ$ is acyclic and free, we can find a chain contraction $\delta~ \colon C(X,Y) \to C(X,Y)$, i.e$.$ a degree $1$ map of graded $\ZZ$-modules satisfying
$\delta \partial + \partial \delta = \Id$. Let $p_X~ \colon \widehat{X} \to X$ be the maximal abelian cover of $X$ and choose, for every relative cell of $(X,Y)$, an orientation and a lift to $\widehat{X}$. We also order the relative cells of $(X,Y)$ in an arbitrary way. All those choices define a $\ZZ$-basis of $C(X,Y)$ and a $\ZZ[H_1(X)]$-basis of  $C(\widehat{X},p_X^{-1}(Y))$, and there is a one-to-one correspondence between these two basis: thus we can identify $C(\widehat{X},p_X^{-1}(Y))$ with ${\ZZ[H_1(X)]  \otimes_\ZZ C(X,Y)}$. So, $\delta$ extends to a  $\ZZ[H_1(X)]$-linear map $\widehat{\delta}~ \colon C(\widehat{X},p_X^{-1}(Y))  \to C(\widehat{X},p_X^{-1}(Y))$ of degree $1$. The map 
$$
N:= Q(R) \otimes_{\ZZ[H_1(X)]}   \left(\widehat{\delta} \widehat{\partial} + \widehat{\partial} \widehat{\delta}\right)~ \colon 
C^{\varphi_Q}(X,Y) \longrightarrow C^{\varphi_Q}(X,Y) 
$$
is clearly a chain map, which induces the null map at the level of homology. On the other hand, we have
$$
\varepsilon_R\,  \det\big( R \otimes_{\ZZ[H_1(X)]}   (\widehat{\delta} \widehat{\partial} + \widehat{\partial} \widehat{\delta})\big)
= \varepsilon_R\, \varphi \det\big(    \widehat{\delta} \widehat{\partial} + \widehat{\partial} \widehat{\delta} \big)
= \varepsilon   \det\big(    \widehat{\delta} \widehat{\partial} + \widehat{\partial} \widehat{\delta}\big)
= 1
$$ 
where $\varepsilon~ \colon \ZZ[H_1(X)]\to \ZZ$ denotes the augmentation of the group ring $ \ZZ[H_1(X)]$. The above identity implies that $\det(N)\neq 0$. We conclude that $N$ is an isomorphism of chain complexes and that $C^{\varphi_Q}(X,Y)$ is acyclic.
\end{proof}

The second property of the Magnus functor to be shown is the monoidality.

\begin{proposition}  \label{prop:monoidality}
\bgm $\Mag~ \colon \Cob{}_G \rightarrow \pLagr_R$ is  a strong monoidal functor.\egm
\end{proposition}

\begin{proof}
Recall that $\Cob_G$ and $\pLagr_R$ are viewed as strict monoidal categories (see Remark \ref{rem:not_strict}). We first define a natural transformation $T$ between the functors  $\Mag(\cdot) \boxtimes \Mag(\cdot) $ and $\Mag(\cdot \boxtimes \cdot)$. For any object $(k,\kappa)$ of the category $\Cob_G$, we denote $\Mag(k,\kappa) = \big( H^{\kappa}_1(F_k, \star), \SSempty^{(k)}, \nu^{(k)}/2 \big)$.
Thus, for any objects $(g, \gvf)$ and $(h, \psi)$ of the category $\Cob_G$, we have 
\begin{eqnarray*}
\Mag(g,\gvf) \boxtimes \Mag(h,\psi) =  
\left( H^{\gvf}_1(F_g, \star) \oplus H^{\psi}_1(F_h, \star)\, ,\,   \SSempty^{(g)}\, {}_{\nu^{(g)}/2}\! \oplus_{\nu^{(h)}/2}  \SSempty^{(h)}\, ,\,  \nu^{(g)}/2 +  \nu^{(h)}/2\right) 
\end{eqnarray*}
and 
$$
\Mag\big((g,\gvf)\boxtimes(h,\psi)\big) = \big( H^{\gvf\oplus\psi}_1(F_{g+h}, \star), \SSempty^{(g+h)}, \nu^{(g+h)}/2 \big).
$$
We claim that the $R$-linear map $H^{\gvf}_1(F_g, \star) \oplus H^{\psi}_1(F_h, \star) \to H^{\gvf\oplus\psi}_1(F_{g+h}, \star)$ induced by the inclusions
of $F_g$ and $F_h$ in $F_g \sharp_\partial F_h = F_{g+h}$ defines an isomorphism
$$
T_{(g,\gvf),(h,\psi)}~ \colon \Mag(g,\gvf) \boxtimes \Mag(h,\psi) \longrightarrow \Mag\big((g,\gvf)\boxtimes(h,\psi)\big)
$$
 in the category $\pUni$. Clearly, $T_{(g,\gvf),(h,\psi)}$
 maps $\nu^{(g)}/2+ \nu^{(h)}/2$ to $\nu^{(g+h)}/2$ since the boundary curve of $F_g \sharp_\partial F_h$ is the connected sum of the boundary curves of $F_g$ and $F_h$.
Thus the claim reduces to proving that $T_{(g,\gvf),(h,\psi)}$ is unitary.
Let  $x, y \in H^{\gvf}_1(F_g, \star)$ and $z,t \in H^{\psi}_1(F_h, \star)$, and denote by $x',y',z',t'$ their images in $H^{\gvf\oplus\psi}_1(F_{g+h}, \star)$: we have 
\begin{eqnarray*} 
&&	\big(\SSform{\cdot}{\cdot}^{(g)} \, {{}_{\nu^{(g)}/2}\oplus_{\nu^{(h)}/2}} \SSform{\cdot}{\cdot}^{(h)} \big) (x+z,y+t)\\
	&=& \SSform{x}{y}^{(g)} + \SSform{z}{t}^{(h)} + \SSform{x}{\nu^{(g)}/2}^{(g)}\, \SSform{\nu^{(h)}/2}{t}^{(h)} - \SSform{\nu^{(g)}/2}{y}^{(g)}\, \SSform{z}{\nu^{(h)}/2}^{(h)} \\
	& \by{eq:boundary_curve_bis} & \SSform{x}{y}^{(g)} + \SSform{z}{t}^{(h)} - \dep_{*}(x)\, \overline{\dep_{*}(t)} + \overline{\dep_*(y)}\, \dep_{*}(z)
\end{eqnarray*}
and 
\begin{eqnarray*}
	&& \SSform{x'+z'}{y'+t'}^{(g+h)} \\
	 &=& \SSform{x}{y}^{(g)} + \SSform{z}{t}^{(h)} + \SSform{x'}{t'}^{(g+h)} - \overline{\SSform{y'}{z'}^{(g+h)}} \\
	& \by{eq:form_s}  &  \SSform{x}{y}^{(g)} + \SSform{z}{t}^{(h)}  
	+ \big( 2 \form{x'}{t'}^{(g+h)}   - \dep_{*}(x')\, \overline{\dep_{*}(t')}  \big)-  \overline{ \big( 2 \form{y'}{z'}^{(g+h)}  - \dep_{*}(y')\, \overline{\dep_{*}(z')}  \big)} \\
	&=  & \SSform{x}{y}^{(g)} + \SSform{z}{t}^{(h)} - \dep_{*}(x) \overline{\dep_{*}(t)} + \overline{\dep_{*}(y)} \dep_{*}(z)
\end{eqnarray*}
which proves our claim about $T_{(g,\gvf),(h,\psi)}$. By applying  the ``graph'' functor $\pUni \to \pLagr_R$, we obtain an isomorphism 
$T_{(g,\gvf),(h,\psi)}~ \colon \Mag(g,\gvf) \boxtimes \Mag(h,\psi) \rightarrow \Mag\big((g,\gvf)\boxtimes(h,\psi)\big)$ in the category $\pLagr$.
This isomorphism is natural since, for any 
$$
(M,\gvf) \in \Cob_G\big((g_-,\gvf_-), (g_+,\gvf_+)\big) \quad \hbox{and} \quad (N,\psi) \in \Cob_G\big((h_-,\psi_-), (h_+,\psi_+)\big),
$$
we have
\begin{eqnarray*}
&& T^{-1}_{(g_+,\gvf_+),(h_+,\psi_+)} \circ \Mag\big( (M,\gvf) \boxtimes  (N,\psi) \big) \circ T_{(g_-,\gvf_-),(h_-,\psi_-)} \\
&=& \cl \left(\ker \big((-m_- \oplus -n_-) \oplus (m_+ \oplus n_+)\big) \right) \\
& \simeq  & \cl \left(\ker \big((-m_- \oplus m_+) \oplus (-n_- \oplus n_+)\big) \right) \\
&= & \cl \big(\ker (-m_- \oplus m_+) \oplus \ker (-n_- \oplus n_+) \big) \\
&=&  \cl \big(\ker (-m_- \oplus m_+) \big) \oplus \cl \big(\ker (-n_- \oplus n_+) \big) \ = \ \Mag(M,\gvf) \boxtimes \Mag(N,\psi).
\end{eqnarray*}
\bgm Recall that the unit object of $\Cob_G$ is the pair $I$ consisting of the integer $0$ 
and the trivial group homomorphism $H_1(F_0) \to G$,  and that \egm
the unit object of $\pLagr_R$ is $\mathcal{I}:= (\{0\},0,0)$. Thus we have $\Mag(I) = (H_1(F_0 ,\star), 0,0)=\mathcal{I}$,
and we define $U~ \colon I \to \Mag(I)$ to be the identity.
It is immediately checked  that the morphism $U$ and the natural transformation $T$ 
satisfy the coherence conditions that make $\Mag$ into a strong monoidal functor.
\end{proof}

\begin{remark}
Proposition \ref{prop:h_cobordism_rel} is the analogue of the invariance under concordance of tangles which Cimasoni and Turaev mentioned for their functor \cite[end of \S 3.3]{CiTu05I}. 
The tangle analogue of Proposition~\ref{prop:monoidality} does not seem to have been addressed in \cite{CiTu05I}.
\end{remark}


\section{Examples and computations} \label{Applications} 

 We give examples and  explain how to compute the Magnus functor using Heegaard splittings.


\subsection{The Magnus representation} \label{subsec:homology_cobordisms}

Fix an integer $g \geq 1$. A {\itshape homology cobordism} over $F_g$ is a morphism $M\in\Cob(g,g)$ such that $m_{\pm}~ \colon H_1(F_g) \rightarrow H_1(M)$ are isomorphisms. The set of equivalence classes of homology cobordisms defines  a submonoid $\HCob(F_g) \subset \Cob(g, g)$.

Let $G$ be an abelian group and fix a group homomorphism $\gvf~ \colon H_1(F_g) \rightarrow G$. Here we will only consider those $M \in \HCob(F_g)$ such that the composition
$$ H_1(F_g) \mathop{\longrightarrow}^{m_{-}}_\simeq H_1(M) \mathop{\longrightarrow}^{m^{-1}_{+}}_\simeq H_1(F_g) \overset{\gvf}{\longrightarrow} G $$
coincides with $\gvf$. By equipping any such cobordism $M$ with  the group homomorphism $\gvf:= \gvf\, m_+^{-1} =\gvf\, m_-^{-1}$, we define a  submonoid $\HCob^{\gvf}(F_g) \subset \Cob_G\big( (g, \gvf), (g, \gvf) \big)$.

Assume now that $G$ is a multiplicative subgroup of a commutative ring $R$, satisfying~\eqref{eq:R_and_G} and~\eqref{eq:R_again}. 
Let $Q := Q(R)$ be the field of fractions of $R$, and  denote by $\varphi_Q~ \colon \ZZ[H_1(F_g)] \to Q$  the ring homomorphism  induced by $\gvf~ \colon H_1(F_g)\to G$. For any $M \in \HCob^\varphi(F_g)$, the fact that  $m_{\pm}~ \colon H_1(F_g) \rightarrow H_1(M)$ is an isomorphism of abelian groups implies that $m_{\pm}~ \colon H^{\gvf_Q}_1 ( F_g, \star ) \rightarrow H^{\gvf_Q}_1 ( M, \star)$  is an isomorphism of $Q$-vector spaces (see Lemma~\ref{lem:KLW}). Thus,  $r^{\gvf}(M) := {(m_{+})^{-1} \, \circ \, m_{-}}$ is an automorphism of $H^{\gvf_Q}_1 ( F_g, \star)$. The \emph{Magnus representation} is the monoid homomorphism 
$$ r^{\gvf}~ \colon \HCob^{\gvf}(F_g) \longrightarrow \Aut_Q \big( H_1^{\gvf_Q} ( F_g, \star ) \big).$$
The reader is referred to \cite{Sa12} for a survey of this invariant of homology cobordisms. 

We now describe the group of units of the monoid  $\HCob^{\gvf}(F_g)$. Let $\mcg(F_g)$ be the \emph{mapping class group} of the surface $F_g$, 
which consists of the isotopy classes of self-homeomorphisms of~$F_g$ fixing $\partial F_g$  pointwise. The \emph{mapping cylinder} construction, which associates to any such homeomorphism $f$ the cobordism
$$ \bfc(f):= \big(F_g \times [-1,1], (f \times \{-1\}) \cup (\partial F_g \times \Id) \cup (\Id\times \{1\}) \big) $$
defines a monoid homomorphism $\bfc~ \colon \mcg(F_g) \to \HCob(F_g)$. It is well-known that $\bfc$ is injective and that its image is the group of units of $\HCob(F_g)$: see \cite[\S 2.2]{HaMa12}, for instance. Thus the  subgroup  
$$ \mcg^{\gvf}(F_g) := \big\lbrace f \in \mcg(F_g) : \varphi f = \varphi \in \Hom(H_1(F_g), G)  \big\rbrace $$
of the mapping class group is mapped by~$\bfc$ isomorphically onto the group of units  of $\HCob^\gvf(F_g)$. Classically, the \emph{Magnus representation} refers to the group homomorphism
$$ r^{\gvf}~ \colon \mcg^{\gvf}(F_g) \longrightarrow \Aut_R\big( H^{\gvf}_1 ( F_g, \star ) \big)$$
which maps any $f \in \mcg^{\gvf}(F_g)$ to the isomorphism of $R$-modules $f~ \colon H^{\gvf}_1(F_g, \star) \rightarrow H^{\gvf}_1(F_g, \star)$. Thus we have the following commutative diagram:
\begin{center}
	\iftikziii
\else
	\includegraphics{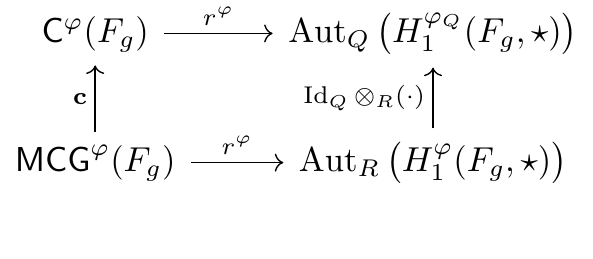}
	\fi
\end{center} \up

The next proposition says that the restriction of the functor $\Mag := \Mag_{R,G}$ to the submonoid $\HCob^{\gvf}(F_g)$ of homology cobordisms is equivalent to the Magnus representation $r^{\gvf}$.

\begin{proposition} \label{prop:magnus}
	Let $g\geq 1$ be an integer and let $\gvf~ \colon H_1(F_g) \to G$ be a group homomorphism.
	Under the assumptions \eqref{eq:R_and_G} and \eqref{eq:R_again} on $R$ and $G$, we have the commutative diagram
	\begin{center}
		\iftikziii
\else
		\includegraphics{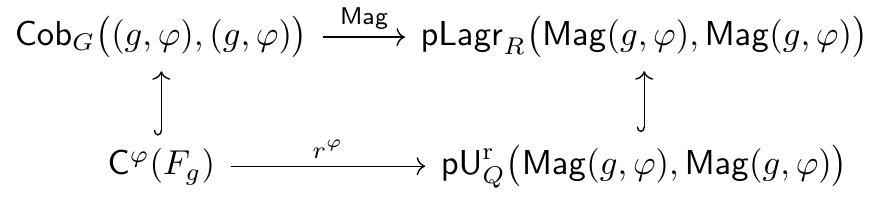}
		\fi
	\end{center}
	where $\Mag(g,\gvf)$ is the pointed skew-Hermitian $R$-module $\big(H^{\gvf}_1 (F_g,\star), \SSempty,  \nu/2  \big)$.
\end{proposition}

\begin{proof}
	For later use, we will prove a slightly more general result.
	We consider some group homomorphisms $\gvf_\pm~ \colon H_1(F_g) \to G$
	and assume that $M \in \HCob(F_g)$ is a cobordism such that $\gvf_- \circ m_-^{-1}= \gvf_+ \circ m_+^{-1}~ \colon H_1(M) \to G$.
	We claim that
	 \begin{equation} \label{eq:slightly_more_general}
	 \Mag (M, \gvf) = \Gamma_{\rho}^{\operatorname{r}}
	 \end{equation}  
	 where $\gvf :=  \gvf_\pm \circ m_\pm ^{-1}$ 
	 and $\gG_{\rho}^{\operatorname{r}}$ is the restricted graph of  the $Q$-linear isomorphism
	 $$
	 \rho := \big( (m_{+})^{-1} \circ m_{-}~ \colon H^{\gvf_{-,Q}}_1 ( F_g, \star  ) \longrightarrow H^{\gvf_{+,Q}}_1 ( F_g, \star  )\big).
	 $$ 
	The special case where $\gvf_-=\gvf_+$ implies the proposition.
	To prove \eqref{eq:slightly_more_general}, we consider the homomorphisms
	$$ \mathfrak{M} := (-m_{-}) \oplus m_{+}~ \colon H^{\gvf_-}_1 ( F_g, \star ) \oplus H^{\gvf_+}_1 ( F_g, \star ) \longrightarrow H^{\gvf}_1 ( M, \star )$$
	and 
	$$ 
	\mathfrak{M}_Q := (-m_{-}) \oplus m_{+}~ \colon H^{\gvf_{-,Q}}_1 ( F_g, \star ) \oplus  H^{\gvf_{+,Q}}_1 ( F_g, \star  ) \longrightarrow H^{\gvf_Q}_1 ( M, \star ).
	$$
	Since the torsion submodule of $ H^{\gvf}_1 ( M, \star )$ is the kernel of the canonical homomorphism from
	$H^{\gvf}_1 ( M, \star )$ to $Q \otimes_R H^{\gvf}_1 ( M, \star ) \simeq H^{\gvf_Q}_1 ( M, \star )$, we have
	\begin{eqnarray*}
	 \Mag (M, \gvf) \ = \   \cl \big( \ker \mathfrak{M} \big)   
	 &=& \mathfrak{M}^{-1}\big(\Tors_R H^{\gvf}_1 ( M, \star ) \big)\\
	 &=&      (\ker \mathfrak{M}_Q)    \cap \big( H^{\gvf_-}_1 ( F_g, \star ) \oplus  H^{\gvf_+}_1 ( F_g, \star ) \big) \\
	 &=&   \Gamma_{\rho} \cap  \big( H^{\gvf_-}_1 ( F_g, \star ) \oplus  H^{\gvf_+}_1 ( F_g, \star ) \big) 	 \ = \ \gG_{\rho}^{\operatorname{r}}.
	 \end{eqnarray*}  
	 
	 \up
\end{proof}

\begin{remark}
Proposition  \ref{prop:magnus} implies that $r^\gvf(M)$ is unitary with respect to the skew-Hermitian form $\SSform{\cdot}{\cdot}$ on $H_1^{\gvf_Q}(F_g,\star)$. This is well-known: see \cite[Theorem 2.4]{Sa07} for an equivalent result.
\end{remark}


\subsection{Computations with Heegaard splittings} \label{subsec:Heegaard}

For any $g \geq 0$, let $(\ga_1, \ldots, \ga_g, \gb_1, \ldots, \gb_g)$ be the system of \lq\lq meridians and parallels\rq\rq{} 
in the surface $F_g$ shown in Figure \ref{fig:meridians_parallels}. We denote by $C_0^g \in \Cob(0,g)$ the cobordism obtained from $F_g \times [-1,1]$ by attaching $g$ $2$-handles along the curves $\ga_1 \times \{-1\},\dots,\ga_g \times \{-1\}$. Similarly, let $C_g^0 \in \Cob(g,0)$ be the cobordism obtained from $F_g \times [-1,1]$ by attaching $g$ $2$-handles along the curves  $\gb_1 \times \{1\},\dots,\gb_g \times \{1\}$. Note that both $C_g^0$ and $C^g_0$ are handlebodies of genus $g$, which we call \emph{upper handlebody} and \emph{lower handlebody}, respectively. 

A \emph{Heegaard splitting} of a cobordism $M \in \Cob(g_-, g_+)$ is a decomposition in the monoidal category $\Cob$ of the form

\begin{equation}\label{eq:Heegaard}
	M = \big( C^0_{r_+}  \boxtimes  \Id_{g_+}  \big) \circ \bfc(f) \circ \big( C_0^{r_-}  \boxtimes  \Id_{g_-} \big)
\end{equation}

where $r_-,r_+$ are some non-negative integers such that $g_+ + r_+ = g_- + r_-$ and $f\in \mcg(F_{g_\pm + r_\pm})$. Such a decomposition of $M$ always exists: see \cite[Theorem 5]{Ke03}, for instance.

Assume now given a ring $R$ and a multiplicative subgroup $G \subset R^{\times}$ satisfying \eqref{eq:R_and_G}. Let $(M,\gvf) \in \Cob_G((g_-,\gvf_-),(g_+,\gvf_+))$ for which we wish to compute the Magnus functor. 
Any Heegaard splitting \eqref{eq:Heegaard} of $M$ induces a decomposition 
\begin{equation}
	\Mag(M,\gvf) = \Big( \Mag\big(C^0_{r_+}, \overline{\gvf}\big)  \boxtimes  \Id_{\Mag(g_+,\gvf_+)} \Big) 
	\circ \Mag\big( \bfc(f), \overline{\gvf} \big) \circ \Big( \Mag\big( C_0^{r_-}, \overline{\gvf} \big)  \boxtimes  \Id_{\Mag(g_-,\gvf_-)} \Big)
\end{equation}
in the monoidal category $\pLagr_R$ where, for any submanifold $S$ of $M$, the map $\overline{\gvf}~ \colon H_1(S) \rightarrow G$ denotes the group homomorphism induced by $\gvf$ through the inclusion $S\hookrightarrow M$. Recall that the identity of $\Mag(g_\pm,\gvf_\pm) = \big(H^{\gvf_\pm}_1 (F_{g_\pm},\star), \SSempty,  \nu/2  \big)$ in $\pLagr_R$ is simply the diagonal of $H^{\gvf_\pm}_1 (F_{g_\pm},\star)$. 
Thus the computation of the functor $\Mag$ can be reduced to its evaluations on upper handlebodies, lower handlebodies and mapping cylinders, which we determine below.

\begin{lemma} \label{lem:handlebodies}
Let $g\geq 0$ be an integer. For any  homomorphism $\gvf~ \colon H_1(C_0^g) \to G$ inducing  $\gvf_+~ \colon H_1(F_g) \to G$ on the ``top  boundary'', the submodule $\Mag(C_0^{g},\gvf)$ of $H_1^{\gvf_+}(F_g,\star)$ is freely generated~by 
$$
a_1^{\gvf_+}:=\big[1 \otimes \widehat \alpha_1  \big],\dots,a_g^{\gvf_+}:= \big[1 \otimes \widehat \alpha_g  \big].
$$ 
Similarly, for any homomorphism $\gvf~ \colon H_1(C_g^0) \to G$  inducing  $\gvf_-~ \colon H_1(F_g) \to G$ on the ``bottom boundary'',
the submodule  $\Mag(C_g^{0},\gvf) $ of $H_1^{\gvf_-}(F_g,\star)$ is freely generated by 
$$
b_1^{\gvf_-}:= \big[1 \otimes \widehat \beta_1  \big],\dots,b_g^{\gvf_-}:= \big[1 \otimes \widehat \beta_g  \big].
$$
\end{lemma}

\begin{proof}
Let $c_+~ \colon F_g \to C_{0}^g$ be the parametrization of the ``top boundary''. Then 
 $$ 
\Mag(C_0^{g},\gvf) = \cl \left( \ker \left(c_+~ \colon H_1^{\gvf_+}(F_g,\star) \longrightarrow H_1^\gvf(C_0^g,\star) \right) \right). 
 $$ 
The $3$-manifold $C_0^g$ retracts to a wedge of circles, corresponding to the curves $\beta_1,\dots,\beta_g$: hence the $R$-module $H_1^\gvf(C_0^g,\star)$ is freely generated by $c_+(b_1^{\gvf_+}),\dots, c_+(b_g^{\gvf_+})$. Furthermore,  $c_+(a_i^{\gvf_+})=0$  for all $i\in \{1,\dots,g\}$, since $\alpha_i$ is null-homotopic in $C_0^g$. It follows that $\ker c_+$ is the submodule $\langle a_1^{\gvf_+},\dots,a_g^{\gvf_+}\rangle_R$ of $ H_1^{\gvf_+}(F_g,\star)$ spanned by  $a_1^{\gvf_+},\dots,a_g^{\gvf_+}$.
Since the quotient
$$
H_1^{\gvf_+}(F_g,\star)/ \langle a_1^{\gvf_+},\dots, a_g^{\gvf_+}\rangle_R \simeq R^g 
$$ 
is torsion-free, $ \ker c_+$ coincides with its closure. We conclude that $\Mag(C_0^{g},\gvf) = \langle a_1^{\gvf_+},\dots,a_g^{\gvf_+}\rangle_R$.
The second statement of the lemma is proved similarly.
\end{proof}

\begin{lemma} \label{lem:mcyl}
Let $f \in \mcg(F_g)$ and let $\gvf_\pm~ \colon H_1(F_g) \rightarrow G$ be  group homomorphisms such that $\gvf_-= \gvf_+ \circ f$. Denote by $\gvf~ \colon H_1(F_g \times [-1,1]) \rightarrow G$ the homomorphism $\gvf_+ \circ p$, where $p~ \colon F_g \times [-1,1] \rightarrow F_g$ is the cartesian projection. Then, $\Mag(\bfc(f), \gvf )$ is the graph of the unitary $R$-isomorphism
$f~ \colon H_1^{\gvf_-} (F_g,\star) \to H_1^{\gvf_+}(F_g,\star)$.
\end{lemma}

\begin{proof}
This  immediately follows from \eqref{eq:slightly_more_general}.
\end{proof}

It is well-known that the map  $f~ \colon H_1^{\gvf_-} (F_g,\star) \to  H_1^{\gvf_+}(F_g,\star)$ in Lemma \ref{lem:mcyl} can be explicitly computed using Fox's free differential calculus. Specifically, the matrix of  this $R$-isomorphism in the bases $(a_1^{\gvf_\pm},\dots,a_g^{\gvf_\pm},b_1^{\gvf_\pm},\dots,b_g^{\gvf_\pm})$ of $H_1^{\gvf_\pm}(F_g,\star)$  is the ``Jacobian matrix''
$${\scriptsize
 \gvf_+ \begin{pmatrix} 
\frac{\partial f_*(\alpha_1)}{\partial \alpha_1} & \cdots&\frac{\partial f_*(\alpha_g)}{\partial \alpha_{1}}  & \frac{\partial f_*(\beta_1)}{\partial \alpha_1} & \cdots&\frac{\partial f_*(\beta_g)}{\partial \alpha_{1}} \\
\vdots & &\vdots & \vdots & & \vdots\\
\frac{\partial f_*(\alpha_1)}{\partial \alpha_g} & \cdots&\frac{\partial f_*(\alpha_g)}{\partial \alpha_{g}}  & \frac{\partial f_*(\beta_1)}{\partial \alpha_g} & \cdots&\frac{\partial f_*(\beta_g)}{\partial \alpha_{g}}  \\
\frac{\partial f_*(\alpha_1)}{\partial \beta_1} & \cdots&\frac{\partial f_*(\alpha_g)}{\partial \beta_{1}}  & \frac{\partial f_*(\beta_1)}{\partial \beta_1} & \cdots&\frac{\partial f_*(\beta_g)}{\partial \beta_{1}} \\
\vdots & &\vdots & \vdots & & \vdots\\
\frac{\partial f_*(\alpha_1)}{\partial \beta_g} & \cdots&\frac{\partial f_*(\alpha_g)}{\partial \beta_{g}}  & \frac{\partial f_*(\beta_1)}{\partial \beta_g} & \cdots&\frac{\partial f_*(\beta_g)}{\partial \beta_{g}}  
\end{pmatrix}}
$$
where  $f_*~ \colon \pi_1(F_g,\star) \to  \pi_1(F_g,\star)$ is the group isomorphism induced by the homeomorphism~$f$.
This makes the link between the group-theoretical and the topological formulations of the Magnus representation of mapping class groups.


\subsection{The case of trivial coefficients} \label{subsec:trivial_coef}

Consider the simplest case where $G := \{1\}$, $R := \ZZ$ and the involution $r\mapsto \overline{r}$ of $R$ is the identity. Then, the Magnus functor $\Mag:= \Mag_{R,G}$ provides a \bgm strong \egm monoidal functor
$$
\Mag~ \colon \Cob/\!\sim_H\, \longrightarrow \Lagr_\ZZ
$$
where  $\Lagr_\ZZ$ denotes the category of Lagrangian relations between symplectic $\ZZ$-modules. Specifically, the functor $\Mag$ assigns to any object $g$ the symplectic $\ZZ$-module
\begin{equation}
\Mag(g) =  \big(H_1 (F_g),  \formempty\big),
\end{equation}
where $H_1 (F_g)=H_1(F_g;\ZZ)$  is the ordinary homology and $\formempty$ is the usual homology intersection form,
and it assigns to any morphism $M\in \Cob(g_{-},g_{+})$ the Lagrangian submodule
\begin{equation}  \label{eq:closure_here}
 \Mag (M)  =  \cl  \big( \ker \big((-m_-)\oplus m_+~ \colon H_1 ( F_{g_{-}} ) \oplus H_1 ( F_{g_{+}}) \longrightarrow  H_1( M )\big) \big),
 \end{equation}
which is a free direct summand of $H_1 ( F_{g_{-}} ) \oplus H_1 ( F_{g_{+}})$ of rank $g_-+g_+$. Note that,  with respect to the general definition of $\Mag$ provided by Theorem \ref{th:Magnus}, we have done two simplifications  which are only possible for trivial coefficients:
\begin{itemize}
\item  we have multiplied the intersection forms by  $1/2$ in \eqref{eq:form_s}, so that the symplectic forms under consideration are non-singular (and not only non-degenerate);
\item we have ignored the ``distinguished'' element $0\in H_1 (F_g)$ so that $\Mag$ takes values in the monoidal category $\Lagr_\ZZ$ (instead of $\pLagr_\ZZ$ --- see Remark \ref{rem:monoidal}).
\end{itemize}

If we now take coefficients in $R:=\RR$ and we still assume that $G=\{1\}$, then the same definitions apply and provide a \bgm strong \egm monoidal functor $\Mag: \Cob/\!\sim_H\, \to \Lagr_\RR$. (Of course, the closure in \eqref{eq:closure_here} then becomes needless.) In this case, the functor $\Mag$ is essentially  the ``TQFT''  introduced by Donaldson in \cite{Do99}.  (See also \cite{HoLiWa15}.) However, our context is slightly different since Donaldson's functor applies to cobordisms between \emph{closed} surfaces, and those cobordisms $M$ are assumed to satisfy the following: the linear map $H_1(\partial M;\RR) \to H_1(M;\RR)$ induced by the inclusion is surjective. This homological assumption is used in \cite{Do99} to give an alternative definition in terms of moduli spaces of flat $U(1)$-connections, which recovers the construction of Frohman and Nicas \cite{FN91,FN94}. (See also \cite{Ke03a}.)


\section{Relation with the Alexander functor}

We relate the Magnus functor to the Alexander functor, which has been introduced in~\cite{FlMa14}. In this section, $G$ is a finitely-generated free abelian group, $R:=\ZZ[G]$ is the group  ring of $G$ and $Q:=Q(R)$ is the field of fractions of $R$. For any $R$-module $N$, we denote $N_Q:= Q \otimes_R N$ and, when $N$ is torsion-free,    $N$ is viewed as a submodule of $N_Q$ via the canonical map $n \mapsto 1 \otimes n$. For any $R$-linear map $f~ \colon N \to N'$, the map $\bgm \Id_Q \egm \otimes_R f$ is denoted $f_Q~ \colon  N_Q \to N'_Q$.  


\subsection{Review of the Alexander function} \label{subsec:review_Af}

The construction of the Alexander functor is based on the notion of ``Alexander function'' for  $3$-manifolds with boundary, 
which is due to Lescop \cite{Lescop}. We start by recalling this notion.

Let $M$ be a compact connected orientable $3$-manifold with \bgm non-empty \egm connected boundary. 
We fix a base point $\star\in \partial M$ and  a group homomorphism $\varphi~ \colon H_1(M) \to G$. Using a handle decomposition of~$M$, it is easily seen that the $R$-module $H:= H_1^\varphi(M,\star)$ has a presentation of deficiency $g:= 1-\chi(M)$: see  \cite[Lemma 2.1]{FlMa14}. We choose such a presentation:
\begin{equation}   \label{eq:choosen_presentation}
H = \big\langle \ggm_1,\dots, \ggm_{g+r}\, \vert\, \rho_1, \ldots, \rho_r \big\rangle.
\end{equation}
Let $\gG$ be the $R$-module freely generated by the symbols $\gamma_1,\dots, \gamma_{g+r}$, and regard $\rho_1,\dots,\rho_r$  as elements of $\Gamma$. 
Then the \emph{Alexander function} of $M$ with coefficients twisted by  $\varphi$ is the $R$-linear map $\cA_M^\varphi~ \colon \Lambda^g H \to R$
defined by 
$$ \cA^{\gvf}_M (u_1\wedge \cdots \wedge u_g) \cdot \gamma_1\wedge \cdots \wedge  \gamma_{g+r}= 
\rho_1\wedge \cdots \wedge \rho_r \wedge \widetilde{u_1} \wedge \cdots \wedge \widetilde{u_g} \ \in \Lambda^{g+r} \Gamma $$
for any  $u_1,\dots, u_g \in H$, which we lift to some $\widetilde{u_1}, \dots, \widetilde{u_g} \in \Gamma$ in an arbitrary way. 
Due to the choice of the presentation \eqref{eq:choosen_presentation}, the map $ \cA^{\gvf}_M$ is only defined up to multiplication by an element of~$\pm G$.


\subsection{A formula for the Alexander function}

We now give a general formula for the map $\cA^{\gvf}_M$, keeping the  notations of Section \ref{subsec:review_Af}. Consider the $R$-linear map
$$
j^\star~ \colon H^\partial :=  H_1^\varphi(\partial M,\star) \longrightarrow H_1^\varphi(M,\star) = H 
$$
induced by the inclusion \bgm $j^\star~ \colon (\partial M,\star) \to (M,\star)$, \egm and consider the same map $j^\star_Q~ \colon H^\partial_Q \to H_Q$ with coefficients in the field~$Q$.  

\begin{lemma} \label{lem:Alexander_function}
The following three conditions are equivalent:
\begin{itemize}
\item[(1)] $\cA^{\gvf}_M \neq 0$;
\item[(2)] $\dim H_Q=g$;
\item[(3)] $j_Q^\star$ is surjective.
\end{itemize}
Assume this and choose some elements $w_1,\dots,w_g \in H^\partial$ such that  $j^\star_Q(w_1),\dots, j^\star_Q( w_g )$ generate the vector space $H_Q $. 
Then, for any $y_1,\dots,y_g \in H$, we have
\begin{equation} \label{eq:compute_Alex}
\cA_M^\varphi\big( y_1 \wedge \cdots \wedge y_g \big) = \ord(H/j^\star(W)) \cdot 
 \det\left(\begin{array}{c} \hbox{matrix of $\big( q(y_1),\dots,  q(y_g)\big)$ in} \\
\hbox{the basis $\big(  j^\star_Q(w_1),\dots, j^\star_Q(w_g ) \big)$}  \end{array}\right)
\end{equation}
where $q~ \colon H \to H_Q$ is the canonical map, $W$ is  the  $R$-submodule  of $H^\partial$ generated by $w_1,\dots,w_g$, and $ \ord( \cdot )\in R/\!\pm G$ denotes the order of a finitely-generated $R$-module.
\end{lemma}

\begin{proof}
The equivalence between (1) and (2) is implicit in \cite{Lescop} and proved in \cite[Lemma~2.3]{FlMa14}.
We now prove that (2) is equivalent to (3). Let $\varphi_Q~ \colon \ZZ[H_1(M)] \to Q$ be the ring homomorphism induced by $\varphi~ \colon H_1(M) \to G$,
and consider the same commutative diagram as the one at  the beginning of the proof of Lemma \ref{LagrG}:
\begin{center}
		\begin{tikzpicture}[>=angle 90,scale=2.2,text height=1.5ex, text depth=0.25ex]
		\node (a0) at (0,1) {\normalcolor $0$};
		\node (a1) [right=of a0] {\normalcolor $H^{\gvf_Q}_1 ( \dep M )$};
		\node (a2) [right=of a1] {\normalcolor $H^{\gvf_Q}_1 ( \dep M, \star )$};
		\node (a3) [right=of a2] {\normalcolor $\partial_*\big(H^{\gvf_Q}_1 ( \dep M, \star )\big)$};
		\node (a4) [right=of a3] {\normalcolor $0$};
		\node (b0) [below=of a0] {\normalcolor $0$};
		\node (b1) [below=of a1] {\normalcolor $H^{\gvf_Q}_1 ( M )$};
		\node (b2) [below=of a2] {\normalcolor $H^{\gvf_Q}_1 ( M , \, \star )$};
		\node (b3) [below=of a3] {\normalcolor $\partial_*\big(H^{\gvf_Q}_1 ( M, \star )\big)$};
		\node (b4) [below=of a4] {\normalcolor $0$};
		
		\normalcolor
		\draw[->]
		(a0) edge (a1)
		(a1) edge node[auto] {} (a2)
		(a2) edge node[auto] {\normalcolor $\partial_*$} (a3)
		(a3) edge node[auto] {} (a4)
		(b0) edge (b1)
		(b1) edge (b2)
		(b2) edge node[auto] {\normalcolor $\partial_*$}  (b3)
		(b3) edge node[auto] {} (b4)
		(a1) edge node[auto] {\normalcolor $j_Q$}  (b1)
		(a2) edge node[auto] {\normalcolor $j^\star_Q  $}   (b2)
		(a3) edge node[auto] {\normalcolor } (b3);
		\end{tikzpicture}
\end{center}
It follows from the ``ker-coker'' lemma that
\begin{equation} \label{eq:2_cokers}
\dim( \coker j^\star_Q )=  \dim( \coker j_Q )+  \dim\partial_*\big(H^{\gvf_Q}_1 ( M, \star )\big) - \dim\partial_*\big(H^{\gvf_Q}_1 ( \dep M, \star )\big).
\end{equation}
According to the statement (i) in  the proof of Lemma \ref{LagrG}, $\ker j_Q$ is a Lagrangian subspace of the $Q$-vector space $H_1^{\varphi_Q}(\partial M)$ equipped with the equivariant intersection form. Since the surface $\partial M$ has genus $g$, we deduce that
\begin{eqnarray*}
\dim (\ker j_Q) \ = \  \dim \big(H^{\gvf_Q}_1 ( \dep M )\big)/2 
\notag &= &\left\{ \begin{array}{ll} g & \hbox{if $\varphi j (H_1(\partial M))=1$}  \\ g-1 & \hbox{otherwise} \end{array}\right. \\
&=&   g- \dim\partial_*\big(H^{\gvf_Q}_1 ( \dep M, \star )\big).
\end{eqnarray*}
 Hence 
\begin{eqnarray*}
 \dim (\coker j_Q  ) & = & \dim H^{\gvf_Q}_1 ( M )  -  \dim H^{\gvf_Q}_1 ( \dep M ) + \dim (\ker j_Q) \\
 &  = & \dim H^{\gvf_Q}_1 ( M )  -  \dim (\ker j_Q) \\ 
 & =&  \dim H^{\gvf_Q}_1 ( M )  - g+ \dim\partial_*\big(H^{\gvf_Q}_1 ( \dep M, \star )\big)  \\
 & =  &   \dim H^{\gvf_Q}_1 ( M,\star )  -  \dim\partial_*\big(H^{\gvf_Q}_1 (  M, \star )\big) - g+ \dim\partial_*\big(H^{\gvf_Q}_1 ( \dep M, \star )\big).
\end{eqnarray*}
By comparing this identity with \eqref{eq:2_cokers}, we deduce that 
$$ \dim H^{\gvf_Q}_1 ( M,\star ) = \dim( \coker j^\star_Q ) +g $$
which proves the equivalence between conditions  (2) and  (3).

We now prove the second part of the lemma. Conditions (2) and (3) imply the existence of some elements $w_1,\dots,w_g \in H^\partial_Q$ whose images by $j_Q^\star$ generate $H_Q$. 
After multiplication by some element of $R \setminus \{0\}$, we can assume that  $w_1,\dots,w_g \in H^\partial \subset H^\partial_Q$. 
Set $w'_1 := j^\star(w_1), \dots, w'_g := j^\star(w_g)\in H$. We can derive from any presentation of $H$ of deficiency~$g$ another presentation of the following form:
\begin{equation}   \label{eq:pre_pres_H}
H = \big\langle u_1,\dots, u_r, w'_1,\dots, w'_g\, \vert\, \delta_1,\dots, \delta_r \big\rangle.
\end{equation}
Let $x_1,\dots,x_g \in H^\partial $ and  set $x'_1 := j^\star(x_1), \dots, x'_g := j^\star(x_g)\in H$. Then, we can derive from \eqref{eq:pre_pres_H} another presentation of $H$ of the form 
\begin{equation}   \label{eq:pres_H}
H = \big\langle u_1,\dots, u_r, w'_1,\dots, w'_g,x'_1,\dots,x'_g
\, \vert\, \delta_1,\dots, \delta_r, \theta_1,\dots, \theta_{g} \big\rangle
\end{equation}
where, for each $i\in\{1,\dots, g\}$, the relation $\theta_i$ writes $x'_i$ as an $R$-linear combination of $ u_1,\dots, u_r,$ $w'_1,\dots, w'_g$. We denote by $\Gamma$ the $R$-module freely generated by the symbols $u_1,\dots, u_r,$ $w'_1\dots, w'_g,$ $x'_1,\dots,x'_g$, and we regard the relations $\delta_1,\dots, \delta_r, \theta_1,\dots, \theta_{g}$ as elements of $\Gamma$.

We now analyse the relations of the presentation \eqref{eq:pres_H}. For each $i\in \{1,\dots, r\}$,  $\delta_i\in \Gamma$ can be decomposed as $ \delta_i(u) + \delta_i(w')$, where $\delta_i(u)$ belongs to the submodule of $\Gamma$ freely generated by $u_1,\dots, u_r$ and $\delta_i(w')$ belongs to the submodule of $\Gamma$ freely generated by $w'_1,\dots,w'_g$. 
 Besides, for each $j\in \{1,\dots,r\}$, we can find a $z_j\in R \setminus \{0\}$ such that  $z_j u_j\in H$ is an $R$-linear combination of $w'_1, \dots, w'_g \in H$: consider the product $z:= z_1 \cdots z_r \in  R \setminus \{0\}$. 
 Using \eqref{eq:pre_pres_H}, we deduce that $z u_j \in \Gamma$ can be written as an $R$-linear combination of $w'_1,\dots, w'_g$  and $\delta_1,\dots, \delta_r$. Therefore, for each $i\in \{1,\dots,g\}$, there exist an $R$-linear combination $D_i\in \Gamma$ of $\delta_1,\dots, \delta_r$ 
 and an $R$-linear combination $W'_i \in \Gamma$ of $w'_1,\dots, w'_g$ such that
\begin{equation} \label{eq:syzygy}
z \theta_i + D_i = z x'_i - W'_i   \in \Gamma.
\end{equation}
We now compute $z^g \cA_M^\varphi\big( x'_1 \wedge \cdots \wedge x'_g\big)$ taking \eqref{eq:pres_H} as a presentation of $H$. In what follows, for any family of elements of $\Gamma$ such as $(\delta_1,\dots,\delta_r)$, we denote by $\delta \in \Lambda^r\Gamma$ the multivector $\delta_1 \wedge \cdots \wedge \delta_r$:
\begin{eqnarray*}
\Gamma^{2g+r}  \ni z^{g}\, \delta \wedge \theta \wedge x' 
&=& \delta \wedge (z\theta_1) \wedge \cdots \wedge (z\theta_g)  \wedge x' \\
&=& \delta \wedge (z\theta_1+D_1) \wedge \cdots \wedge (z\theta_g+D_g)  \wedge x' \\
&\by{eq:syzygy}& \delta \wedge (zx'_1-W'_1) \wedge \cdots \wedge (zx'_g-W'_g)  \wedge x' \\
&=&  (-1)^g\, \delta \wedge W'_1 \wedge \cdots \wedge W'_g  \wedge x' \\
&=& (-1)^g\, \delta(u) \wedge W'_1 \wedge \cdots \wedge W'_g  \wedge x' \\
&=& (-1)^g \ord(H/j^\star(W))\, u\wedge W'_1 \wedge \cdots \wedge W'_g  \wedge x'
\end{eqnarray*}
where $W$ is  the  $R$-submodule  of $H^\partial$ generated by $w_1,\dots,w_g$. Since $\cA_M^\varphi$ is only defined up to multiplication by an element of $\pm G$, we obtain that 
$$ z^g \cA_M^\varphi\big( j^\star(x_1) \wedge \cdots \wedge j^\star(x_g)\big) 
= \ord(H/j^\star(W))\, \det\left(\!\! \begin{array}{c}  \hbox{ \small matrix of $(W'_1,\dots,W'_g)$ in} \\ \hbox{ \small the basis $(w'_1,\dots,w'_g)$} \end{array}\!\right). $$
We deduce that
\begin{eqnarray*}
\cA_M^\varphi\big( j^\star(x_1) \wedge \cdots \wedge j^\star(x_g)\big)  &=&
	\ord(H/j^\star(W))\, \det\left( \!\!\begin{array}{c}  \hbox{ \small matrix of $(W'_1/z,\dots,W'_g/z)$ in } \\ \hbox{ \small the basis $\big(j^\star_Q(w_1),\dots,j^\star_Q(w_g)\big)$} \end{array}\!\right) \\
 &\by{eq:syzygy} & 
	\ord(H/j^\star(W))\, \det\left( \!\! \begin{array}{c}  \hbox{ \small matrix of $\big(j^\star_Q(x_1),\dots,j^\star_Q(x_g)\big)$ in} \\ 
	\hbox{ \small the basis $\big(j^\star_Q(w_1),\dots,j^\star_Q(w_g)\big)$} \end{array}\! \right).
\end{eqnarray*}
This proves \eqref{eq:compute_Alex} in the case where $y_1,\dots,y_g$ belong to $j^\star(H^\partial)$. The general case easily follows using the fact, for any element $h\in H$, there is a $z\in R\setminus\{0\}$ such that $zh \in j^\star(W)$.
\end{proof}


\subsection{Review of the Alexander functor} \label{subsec:review_AF}

Let $\grMod_{,\pm G}$ be the category whose objects are $\ZZ$-graded $R$-modules and whose morphisms are graded $R$-linear maps of arbitrary degree, up to multiplication by an element of $\pm G$. The \emph{Alexander functor} $ \Alex := \Alex_{R,G}~ \colon \Cob_G \to \grMod_{,\pm G} $ associates to any object $(g,\gvf)$ of $\Cob_G$ the graded $R$-module 
$$ \Alex(g,\gvf) := \Lambda\, H_1^\gvf (F_g,\star) = \bigoplus_{i=0}^{2g}  \Lambda^i H_1^\gvf (F_g,\star),$$
and  to any morphism $(M,\varphi) \in \Cob_G\big(({g_-},\varphi_-),({g_+},\varphi_+)\big)$ an $R$-linear map 
$$ \Alex(M,\varphi)~ \colon \Lambda\,  H_1^{\varphi_-}(F_{g_-},\star) \longrightarrow \Lambda\,  H_1^{\varphi_+}(F_{g_+},\star) $$
of degree $\gd \! g := g_+-g_-$, which is defined as follows.

Denote $H_\pm:=H_1^{\varphi_\pm}(F_{g_\pm},\star)$. 
 For any integer $j\geq 0$, the image $\Alex(M,\varphi)(x)\in \Lambda^{j+ \delta\!g}H_+ $ of any  $x \in \Lambda^j H_-$ is defined by requiring that
$$ \forall y \in \Lambda^{g - j} H_+, \ \cA_M^\varphi \left( \Lambda^j m_-(x) \wedge \Lambda^{g-j} m_+(y) \right) =  \vol  \big( \Alex(M,\varphi)(x) \wedge y\big) $$
where $g:=g_+ + g_-$ and $\vol~ \colon \Lambda^{2g_+} H_+  \to R$ is an arbitrary $R$-linear isomorphism. Note that, due to the choices of $\vol$ and of a presentation of $H$ to define the Alexander function $ \cA_M^\varphi$, the map  $\Alex(M,\varphi)$   is only defined up to multiplication by an element of $\pm G$. 


\subsection{Computation of the Alexander functor}  \label{subsec:compute_Alex}

Let $(M,\varphi) \in \Cob_G((g_-,\varphi_-),(g_+,\varphi_+))$ and keep the notations of Section \ref{subsec:review_AF}. Recall that the Magnus functor assigns to the morphism $(M,\varphi)$ the Lagrangian submodule
$$ \Mag(M,\varphi) = \cl(\ker \mathfrak{M}) \subset H_-\oplus H_+ $$
where $\mathfrak{M} := (-m_-)\oplus m_+~ \colon H_- \oplus H_+ \to H$ is induced by the parametrizations $m_\pm~ \colon F_{g_\pm} \to M$ of the top/bottom boundary of $M$. We also  set $P:= (H_- \oplus H_+) /\!\cl(\ker \mathfrak{M})$ and denote by $p~ \colon H_- \oplus H_+ \to P$ the canonical projection. Since ${Q \otimes_R \cl(\ker \mathfrak{M})}$ is a Lagrangian subspace of $H_{-,Q} \oplus H_{+,Q}$, the $Q$-vector space $P_Q $ has dimension $g=g_-+g_+$. A \emph{transversal} of $(M,\varphi)$ is  a  free submodule $W \subset H_-\oplus H_+$ of rank $g$ such that $p(W)\subset P$ generates $P_Q$. Clearly, there always exists such a $W$ and the $R$-module $p(W)$ is free of rank~$g$.

Given a transversal $W$ of $(M,\varphi)$, we can do the following construction. Choose an $R$-linear isomorphism $\vol~ \colon \Lambda^{2g_+} H_+ \to R$, which induces a  $Q$-linear isomorphism $\vol_Q~ \colon \Lambda^{2g_+} H_{+,Q} \to Q$. Consider the $Q$-linear  map
$$ \Mag_{W} (M,\varphi)~ \colon  \Lambda H_{-,Q} \longrightarrow \Lambda H_{+,Q} $$
of degree $\delta\! g= g_+-g_-$ that assigns to any $x := x_1\wedge \cdots \wedge x_j\in \Lambda^j H_{-,Q}$ the multivector $\Mag_W (M,\varphi)(x) \in \Lambda^{j+ \delta\! g} H_{+,Q}$ defined by the identity
$$ 
\vol_Q \big( \Mag_W (M,\varphi)(x)  \wedge y \big)  = \det\left(\begin{array}{c} \hbox{\small matrix of $\big(p_Q(x_1),\dots, p_Q(x_j), p_Q(y_1),\dots, p_Q(y_{g-j})\big)$ in} \\
\hbox{\small a $Q$-basis of $P_Q$ arising from an arbitrary $R$-basis of $p(W)$}  \end{array}\right)
 $$
for any  $ y_1, \cdots,y_{g-j} \in H_{+,Q}$. Observe the following:
\begin{itemize}
\item[(i)] $\Mag_W(M,\varphi)$ is well-defined up to a multiplication by an element of $\pm G$;
\item[(ii)] if $W'$ is another transversal of $(M,\varphi)$, then $\Mag_{W'}(M,\varphi) = d_{W',W}\, \Mag_{W}(M,\varphi)$ 
where $d_{W',W}$ $\in Q\setminus\{0\}$ is the determinant of a matrix of change of $Q$-bases of $P_Q$: from an arbitrary $R$-basis of $p(W)$
to an arbitrary $R$-basis of $p(W')$.
\end{itemize}
The next lemma shows that the operator $\Mag_W(M,\varphi)$ carries the same information as the submodule $\Mag(M,\varphi)$.

\begin{lemma} \label{lem:W,W'}
Let $(M,\varphi), (M',\varphi') \in \Cob_G((g_-,\varphi_-),(g_+,\varphi_+))$. We have $\Mag(M,\varphi)=\Mag(M',\varphi')$ if, and only if, there exist transversals $W,W'\subset H_-\oplus H_+$ of $(M,\varphi),(M',\varphi')$, respectively, such that $\Mag_W(M,\varphi)= \Mag_{W'}(M',\varphi')$.
\end{lemma}

\begin{proof}
Assume that $\Mag(M,\varphi)=\Mag(M',\varphi)$. Let $W$ be a transversal of $(M,\varphi)$: then $W$ is also a transversal of $(M',\varphi')$ 
and we clearly have $\Mag_W(M,\varphi)=\Mag_W(M',\varphi)$.

Conversely, assume that $\Mag_W(M,\varphi)= \Mag_{W'}(M',\varphi')$ for some transversals $W,W'$. The canonical isomorphism 
$$ \bigoplus_{j=0}^g  \Lambda^j H_{-,Q} \otimes_Q \Lambda^{g-j} H_{+,Q} \stackrel{\simeq}{\longrightarrow} \Lambda^g (H_{-,Q} \oplus H_{+,Q}) $$
can be used to assemble the $Q$-linear maps
$$ \Lambda^j H_{-,Q} \otimes_Q \Lambda^{g-j} H_{+,Q} \longrightarrow Q, \ x \otimes y \longmapsto \vol_Q \big( \Mag_W (M,\varphi)(x)  \wedge y \big) $$
defined for all $j\in \{0,\dots,g\}$, into a single $Q$-linear map $f~ \colon \Lambda^g (H_{-,Q} \oplus H_{+,Q}) \to Q$. This map, which is defined up to multiplication by an element of $\pm G$, satisfies
$$ f(z_1 \wedge \dots \wedge z_g)= 
 \det\left(\begin{array}{c} \hbox{\small matrix of $\big(p_Q(z_1),\dots, p_Q(z_{g})\big)$ in  a $Q$-basis of $P_Q$} \\
\hbox{ \small   arising from an arbitrary $R$-basis of $p(W)$}  \end{array}\right) $$
for any $z_1, \dots, z_g \in H_{-,Q} \oplus H_{+,Q}$. It follows that
$$ \ker p_Q = \big\{z\in H_{-,Q} \oplus H_{+,Q}: f(z\wedge \cdot )=0 \in \Hom_Q\big(\Lambda^{g-1}( H_{-,Q} \oplus H_{+,Q}), Q\big)\big\}. $$
Similarly, we have
$$ \ker p'_Q = \big\{z\in H_{-,Q} \oplus H_{+,Q}: f'(z\wedge \cdot)=0 \in \Hom_Q\big(\Lambda^{g-1}( H_{-,Q} \oplus H_{+,Q}), Q\big)\big\} $$
where $p'~ \colon H_- \oplus H_+ \to P':= (H_- \oplus H_+ )/ \Mag(M',\varphi')$ is the canonical projection and $f'$ is defined from $\Mag_{W'}(M', \varphi')$ 
in the same way as  $f$ is  defined from $\Mag_W(M,\varphi)$. Thus the assumption $\Mag_W(M,\varphi)= \Mag_{W'}(M',\varphi')$ implies that $f=f'$, and we conclude that
$$ \Mag(M,\varphi) = (\ker p_Q) \cap (H_-\oplus H_+) =  (\ker p'_Q) \cap (H_-\oplus H_+) = \Mag(M',\varphi'). $$

\up
\end{proof}

We now give the main result of this section, showing that the Alexander functor can be factorized into two parts. One part is scalar and is a kind of relative Alexander polynomial, while the other part is the operator version of the Magnus functor and, so, is invariant under homology concordance (Proposition~\ref{prop:h_cobordism_rel}). 

\begin{theorem} \label{th:Alex_Magnus}
Let $(M,\varphi) \in \Cob_G((g_-,\varphi_-),(g_+,\varphi_+))$ with transversal $W$. Then
$$
\Alex(M,\varphi) = \ord\big(H / \mathfrak{M} (W)\big) \cdot \Mag_W(M,\varphi) \big\vert_{\Lambda H_-}
$$
where $H= H_1^\varphi(M,\star),   H_\pm = H_1^{\varphi_\pm}(F_{g_\pm},\star)$ and $\mathfrak{M} = (-m_-)\oplus m_+~ \colon H_- \oplus H_+ \to H$.
\end{theorem}

\begin{proof}
Two cases have to be distinguished.
If $\mathfrak{M}_Q~ \colon H_{-,Q} \oplus H_{+,Q} \to H_Q$ is not surjective, then the map $j^\star_Q~ \colon Q \otimes_R H_1^\varphi(\partial M,\star) \to H_Q$ induced by the inclusion is not surjective; it follows from the first part of Lemma \ref{lem:Alexander_function} that $\cA_M^\varphi=0$, which implies that $\Alex(M,\varphi)=0$, and it also follows that $\dim H_Q >g$; but $\dim (Q\otimes_R \mathfrak{M}(W)) \leq \dim W_Q=g$, and we deduce that  $\ord\big(H/\mathfrak{M} (W)\big)=0$. If $\mathfrak{M}_Q~ \colon H_{-,Q} \oplus H_{+,Q} \to H_Q$ is  surjective, then the result follows from the second part of Lemma~\ref{lem:Alexander_function}.
\end{proof}

\bgm 
\begin{remark} \label{rem:hc}
 Consider  the monoid of homology cobordisms $\HCob^\varphi(F_g)$   introduced in Section~\ref{subsec:homology_cobordisms}. \egm
Here $g \geq 1$   
and  $\varphi~ \colon H_1(F_g) \to G$ is a group homomorphism. It follows from Lemma~\ref{lem:KLW} that  $W_\pm:=H_\pm$  is a transversal of $(M,\varphi)$ for any $M \in \HCob^\varphi(F_{g})$. On the one hand, we have
$$ \ord\big(H/ \mathfrak{M} (W_\pm)\big) =  \ord\big(H/m_\pm(H_\pm )\big)  = \ord\, H_1^\varphi(M,\partial_\pm M) \in (R\setminus\{0\})/\!\pm\! G $$
which we denote by $\Delta(M,\partial_\pm M)$ and view as a relative version of the  Alexander polynomial. On the other hand, it follows  from the definitions that
$$
\Mag_{W_+} (M,\varphi) = \Lambda\big(r^\varphi(M)\big) \quad \hbox{and} \quad \Mag_{W_-} (M,\varphi) = \det(r^\varphi(M)\big)^{-1} \cdot \Lambda\big(r^\varphi(M)\big) 
$$ 
where $r^\varphi(M) \in \Aut_Q(H_Q)$ is the Magnus representation of $M$. By Theorem \ref{th:Alex_Magnus}, we get
$$ 
\Alex(M,\varphi) = \Delta(M,\partial_+M) \cdot \Lambda\big(r^\varphi(M)\big) = \Delta(M,\partial_-M) \det(r^\varphi(M)\big)^{-1} \cdot \Lambda\big(r^\varphi(M)\big)  
$$
which is already proved in \cite[Proposition 7.3 \bgm \& \egm  Example 9.6]{FlMa14} using the theory of Reidemeister torsions. This example shows that the ``non-Magnus'' part of $\Alex(M,\varphi)$ in the factorization formula of Theorem \ref{th:Alex_Magnus} is \emph{not} invariant under homology concordance. (See \cite[Remark~7.5]{FlMa14}.) 
\end{remark}

\egm


\subsection{The free case}

We keep the notations of Section \ref{subsec:compute_Alex}, which we specialize to the case where the $R$-module $P=(H_- \oplus H_+) /\!\cl(\ker \mathfrak{M})$ is \emph{free}. (For instance, this condition is certainly satisfied if $G=\{1\}$ and $R=\ZZ$.) Then a natural choice of transversal for $(M,\varphi)$ is  the image of an arbitrary section $s~ \colon P \to H_- \oplus H_+$ of the projection $p$. Since $p(s(P))=P$, the map $\Mag_{s(P)} (M,\varphi) $ does not depend on the choice of $s$, and it can  be described as follows. Because  $\Mag(M,\varphi)=\cl(\ker \mathfrak{M})$ is then a  direct summand of the free $R$-module $H_- \oplus H_+$, it is a finitely-generated  projective module so that, by a classical result of Quillen and Suslin, it is free. Thus, by  Pl\"ucker embedding, we can consider instead the multivector
$$ \Pl( \Mag(M,\varphi)) := e_1 \wedge \dots \wedge e_g \in \gL^g (H_- \oplus H_+) /\! \pm G $$
where $(e_1, \ldots, e_g)$ is an arbitrary basis of $\Mag(M,\varphi)$. Through the  sequence of  isomorphisms 
\begin{eqnarray*}
\Lambda^g (H_- \oplus H_+) \ \simeq \  \bigoplus_{k=0}^g\,  \Lambda^k H_- \otimes \Lambda^{g-k} H_+ 
&\simeq &  \bigoplus_{k=0}^g\,  \Hom\big( \Lambda^{2g_--k} H_-,R\big) \otimes  \Lambda^{g-k} H_+  \\
& \simeq & \bigoplus_{k=0}^g \,  \Hom\big( \Lambda^{2g_--k} H_-, \Lambda^{g-k}H_+ \big)
\end{eqnarray*}
(whose second one  is determined by  the choice of an isomorphism $\Lambda^{2g_-} H_- \simeq R$), 
the multivector $\Pl(\Mag(M,\varphi))$ can be viewed as an $R$-linear map $\Pl (\Mag(M,\varphi))~ \colon \Lambda H_{-} \to \Lambda H_{+}$ which is well-defined up to multiplication by an element of $\pm G$. It can be verified that (for an appropriate choice of signs in the above sequence of isomorphisms)
$$
\Mag_{s(P)}(M,\varphi) =  Q \otimes_R  \Pl\big(\Mag(M,\varphi)\big)
$$
and we deduce the following from Theorem \ref{th:Alex_Magnus}.

\begin{corollary} \label{coro:Alex_Magnus}
Let $(M,\varphi) \in \Cob_G((g_-,\varphi_-),(g_+,\varphi_+))$ \bgm be \egm such that $P=(H_- \oplus H_+) /\!\cl(\ker \mathfrak{M})$ is free,
where $H= H_1^\varphi(M,\star),   H_\pm = H_1^{\varphi_\pm}(F_{g_\pm},\star)$ and $\mathfrak{M} = (-m_-)\oplus m_+~ \colon H_- \oplus H_+ \to H$. Then
$$
\Alex(M,\varphi) = \ord\big(H/\mathfrak{M} (W)\big) \cdot \Pl\big(\Mag(M,\varphi)\big)
$$
where $W \subset H_-\oplus H_+$ is an arbitrary free  submodule of rank $g$ which projects onto $P$.
\end{corollary}

\begin{remark}
The use of the Pl\"ucker embedding to transform submodules into linear operators already appears in \cite{Do99} for $G:= \{1\}$ and $R:= \RR$. (See Section \ref{subsec:trivial_coef} in this connection.)
\end{remark}

\appendix

\bgm 
\section{Glossary}

A \emph{monoidal category} is a category $\mathsf{C}$ with a \emph{monoidal structure}, which consists of
\begin{itemize}
\item a bifunctor $\boxtimes: \mathsf{C} \times \mathsf{C} \to \mathsf{C}$,  called the \emph{tensor product},
\item an object $\mathcal{I}$ of $\mathsf{C}$, called the \emph{unit object},
\item  a natural isomorphism  $A_{\mathcal{X},\mathcal{Y},\mathcal{Z}}: (\mathcal{X} \boxtimes \mathcal{Y}) \boxtimes\mathcal{Z} 
\to \mathcal{X} \boxtimes(\mathcal{Y}\boxtimes\mathcal{Z}) $, called the \emph{associativity constraint},
for any objects $\mathcal{X},\mathcal{Y},\mathcal{Z}$ of $\mathsf{C}$,
\item a natural isomorphism  $L_{\mathcal{X}}:  \mathcal{I} \boxtimes  \mathcal{X}   \to \mathcal{X}$ 
(respectively, {$R_{\mathcal{X}}: \mathcal{X} \boxtimes \mathcal{I} \to \mathcal{X}$}),
called the \emph{left} (respectively, \emph{right}) \emph{unit constraint},  for any object $\mathcal{X}$ of $\mathsf{C}$, 
\end{itemize}
satisfying two kinds of coherence conditions: see \cite[\S VII.1]{ML98}. 
If all the natural isomorphisms $A_{\mathcal{X},\mathcal{Y},\mathcal{Z}}$, $L_{\mathcal{X}}$, and $R_{\mathcal{X}}$ are  identities,
the monoidal category $\mathsf{C}$ is said to be \emph{strict}.\\

Let $(\mathsf{C}, \boxtimes, \mathcal{I}, A,L,R)$ and $(\mathsf{C}', \boxtimes', \mathcal{I}', A',L',R')$ be monoidal categories. 
A \emph{monoidal functor} from $\mathsf{C}$ to $\mathsf{C'}$ is a functor $F: \mathsf{C} \to \mathsf{C}'$ coming with
\begin{itemize}
\item a natural transformation $T_{\mathcal{X}, \mathcal{Y}}:F(\mathcal{X} ) \boxtimes' F(\mathcal{Y} )  \to F(\mathcal{X}  \boxtimes \mathcal{Y} )$, 
for any objects $\mathcal{X}, \mathcal{Y}$ of $\mathsf{C}$,
\item a morphism $U: \mathcal{I}' \to F(\mathcal{I})$ in $\mathsf{C}'$,
\end{itemize}
satisfying three kinds of coherence conditions which involve the associativity/unit constraints of $\mathsf{C}$ and $\mathsf{C'}$: 
see \cite[\S XI.2]{ML98}. If  $U$ is an isomorphism (respectively, the identity) 
and if  $T_{\mathcal{X}, \mathcal{Y}}$ is a natural isomorphism (respectively, the identity),
then the monoidal functor $F$ is said to be \emph{strong} (respectively, \emph{strict}).

\egm

\textbf{Acknowledgements.} 
The authors would like to thank the referee for helpful comments and suggestions.
The third author is partially supported by the Spanish Ministry of Education (grants MTM2013-45710-C2-1-P and MTM2016-76868-C2-2-P), by the ``Grupo de Geometría del Gobierno de Aragón'' and by the European Social Fund. \\
	


\bibliographystyle{amsalpha}
\addcontentsline{toc}{section}{\refname}
{\small
\bibliography{bibliography_fe}
}


\Addresses

\end{document}